\documentclass[british, a4paper, 12pt]{amsart}
\usepackage{babel}
\usepackage[utf8]{inputenc}
\usepackage[T1]{fontenc}
\usepackage{lmodern}
\usepackage{microtype}
\usepackage{csquotes}
\usepackage[margin=5em]{geometry}
\usepackage{amsthm}
\usepackage{thmtools}
\usepackage{multicol}
\usepackage[hidelinks,hypertexnames=false]{hyperref}
\usepackage{afterpage}
\usepackage{mathtools}
\usepackage{amssymb}
\usepackage{wasysym}
\usepackage[multiple]{footmisc}
\usepackage{caption}
\usepackage{subcaption}
\usepackage{tikz}
\usetikzlibrary{shapes, arrows.meta, positioning}
\declaretheorem[name=Theorem,style=plain,parent=section]{theorem}
\declaretheorem[name=Lemma,style=plain,numberlike=theorem]{lemma}
\declaretheorem[name=Corollary,style=plain,numberlike=theorem]{corollary}
\declaretheorem[name=Conjecture,style=plain,numberlike=theorem]{conjecture}
\declaretheorem[name=Definition,style=definition,numberlike=theorem]{definition}
\declaretheorem[name=Remark,style=remark,numberlike=theorem]{remark}
\declaretheorem[name=Example,style=remark,numberlike=theorem]{example}
\usepackage{biblatex}
\title{Barnette Graphs with Faces up to Size 8 are Hamiltonian}
\date{\today}
\author{Tobias Schnieders}
\usepackage[capitalise,noabbrev]{cleveref}
\crefname{conjecture}{Conjecture}{Conjectures}
\begin{document}
\begin{abstract}
Barnette's conjecture states that every cubic, bipartite, planar and $3$-connected graph is Hamiltonian. Goodey verified Barnette's conjecture for all graphs with faces of size up to $6$.

We substantially strengthen Goodey's result by proving Hamiltonicity for cubic, bipartite, planar and ($2$-)connected graphs with faces of size up to $8$. Parts of the proof are computational, including a distinction of $339.068.624$ cases.
\end{abstract}
\maketitle
\section{Introduction}
\emph{Hamiltonian cycles} in graphs are cycles that visit every vertex exactly once. They are of intrinsic combinatorial interest and have many applications, for instance, concerning codes in computer science.

A famous problem in graph theory and complexity theory and a special case of the Travelling Salesman Problem is the so-called \emph{Hamiltonian Path Problem}. It asks, whether a given finite, simple graph is \emph{Hamiltonian}, that is the graph has a Hamiltonian cycle. Due to the following theorem, the Hamiltonian Path Problem remains a hard problem for finite simple graphs, which are cubic, bipartite, planar and connected.
\begin{theorem}[Akiyama, Nishizeki, Saito $1980$ \cite{AkiNisSai80}]\label{thm-2-conc-NP-complete}
The Hamiltonian Path Problem is NP-complete for finite, simple graphs that are cubic, bipartite, planar and $2$-connected.
\end{theorem}
\begin{remark}\label{rem-con->2-conn}
For graphs like those considered in \cref{thm-2-conc-NP-complete}, we will not distinguish between connectedness and $2$-connectedness, since every finite, simple, cubic, bipartite, connected graph is $2$-connected, see \cref{lem:conn->2-conn}.
\end{remark}
For a comparable set of graphs, Tait conjectured Hamiltonicity in 1884. Later, Tutte disproved Tait's conjecture and proposed another similar conjecture which also turned out to be false. In 1969, Barnette conjectured a combination of these two disproven conjectures. It is still open today and, over time, became a famous question in graph theory.
\begin{conjecture}[Barnette's Conjecture, Problem $5$ in \cite{conference-conjecture}]\label{conj:Barnette}
Every finite, simple, cubic, bipartite, planar and $3$-connected graph is Hamiltonian.
\end{conjecture}
\begin{definition}[Barnette Graph]
A finite, simple graph is called a \emph{Barnette graph} if it is cubic, bipartite, planar and $3$-connected.
\end{definition}
According to Steinitz's Theorem (see Theorem $2.1$ in \cite{Gru07}), Barnette's conjecture can also be formulated as follows: ``\emph{The polyhedral graph of every convex, simple polyhedron with even sided faces is Hamiltonian.}'' Likewise, equivalent forms in terms of simplicial polyhedrons also exist.

Although the problem still remains unsolved, some partial results have been proven. Namely, the following result became prominent, as it has been generalised in several different ways. It makes use of the fact that the dual graph of a $3$-connected, planar graph is unique. The main result of this paper, \cref{thm-mainthminintro}, is also a generalisation of this.
\begin{theorem}[Goodey $1975$ \cite{Goo75}]\label{thm-goodey}
Every Barnette graph with faces of size up to $6$ is Hamiltonian.
\end{theorem}
Other partial results, including generalisations of this, are given in the following.
\begin{theorem}[Hertel $2005$ \cite{Her05}]
Barnette's conjecture holds if and only if every Barnette graph is $\overline{x}$-$y$-$\overline{z}$-Hamiltonian, that is for every path of length 3, there is a Hamiltonian cycle passing through the middle edge of the path and avoiding the other two edges.
\end{theorem}
\begin{theorem}[Florek $2016$ \cite{Flo16}, Alt, Payne, Schmidt, Wood $2016$ \cite{AltPayMic16}]
Let $G$ be a Barnette graph. Consider a vertex colouring of the dual graph of $G$ in red, green and blue and suppose that every red-green cycle contains a vertex of degree $4$. Then $G$ is Hamiltonian.
\end{theorem}
\begin{theorem}[Kardo\v s $2020$ \cite{Kar20}]
Every finite, simple, cubic, planar, $3$-connected graph with faces of size up to $6$ is Hamiltonian. 
\end{theorem}
\begin{theorem}[Brinkmann, Goedgebeur, Mckay $2022$ \cite{BriGoeMck22}]
Every Barnette graph with at most $90$ vertices is Hamiltonian.
\end{theorem}
\begin{theorem}[Florek $2023$ \cite{Flo23}]
Every Barnette graph, in which every face is adjacent to at most $4$ faces of size greater than $4$, is Hamiltonian.
\end{theorem}
\begin{theorem}[Shao, Wu $2025$ \cite{ShaWu25}]
Every finite, simple, cubic, planar, $2$-connected graph with faces of size up to $6$ is Hamiltonian.
\end{theorem}
The main \cref{thm-mainthminintro} of this paper is a generalisation of \cref{thm-goodey}. \cref{cor-mainthmintro} follows using the uniqueness of the dual graph of a $3$-connected graph (see Theorem $4.3.2$ in \cite{Die05}).
\begin{theorem}[See \cref{thm-maintheorem}]\label{thm-mainthminintro}
Let $G$ be a finite, simple graph that is cubic, bipartite, planar and connected. Fix an embedding of the graph in $\mathbb{R}^{2}$ and suppose that every face of the embedding is of size at most $8$. Then $G$ is Hamiltonian.
\end{theorem}
\begin{corollary}\label{cor-mainthmintro}
Every Barnette graph with faces of size up to $8$ is Hamiltonian.
\end{corollary}
Parts of the proof of \cref{thm-maintheorem} are computer-aided, including distinction of $339.068.624$ cases.

The upper bound of $8$ in \cref{thm-mainthminintro} is sharp. That is, there exist cubic, bipartite, planar and connected graphs, which are not Hamiltonian and in which every face of every embedding in $\mathbb{R}^2$ is of size up to $10$. An example of such a graph is given in \cref{f:10HolesNotHamil}.

Nonetheless, we discuss in \cref{sec:prospect} how the method of our proof could perhaps be used to prove even stronger statements. Namely, computations suggest that Hamiltonicity of Barnette graphs could turn out to be, in a sense, a ``local'' property.
\begin{figure}[ht]
\centering
\includegraphics[width=0.25\textwidth]{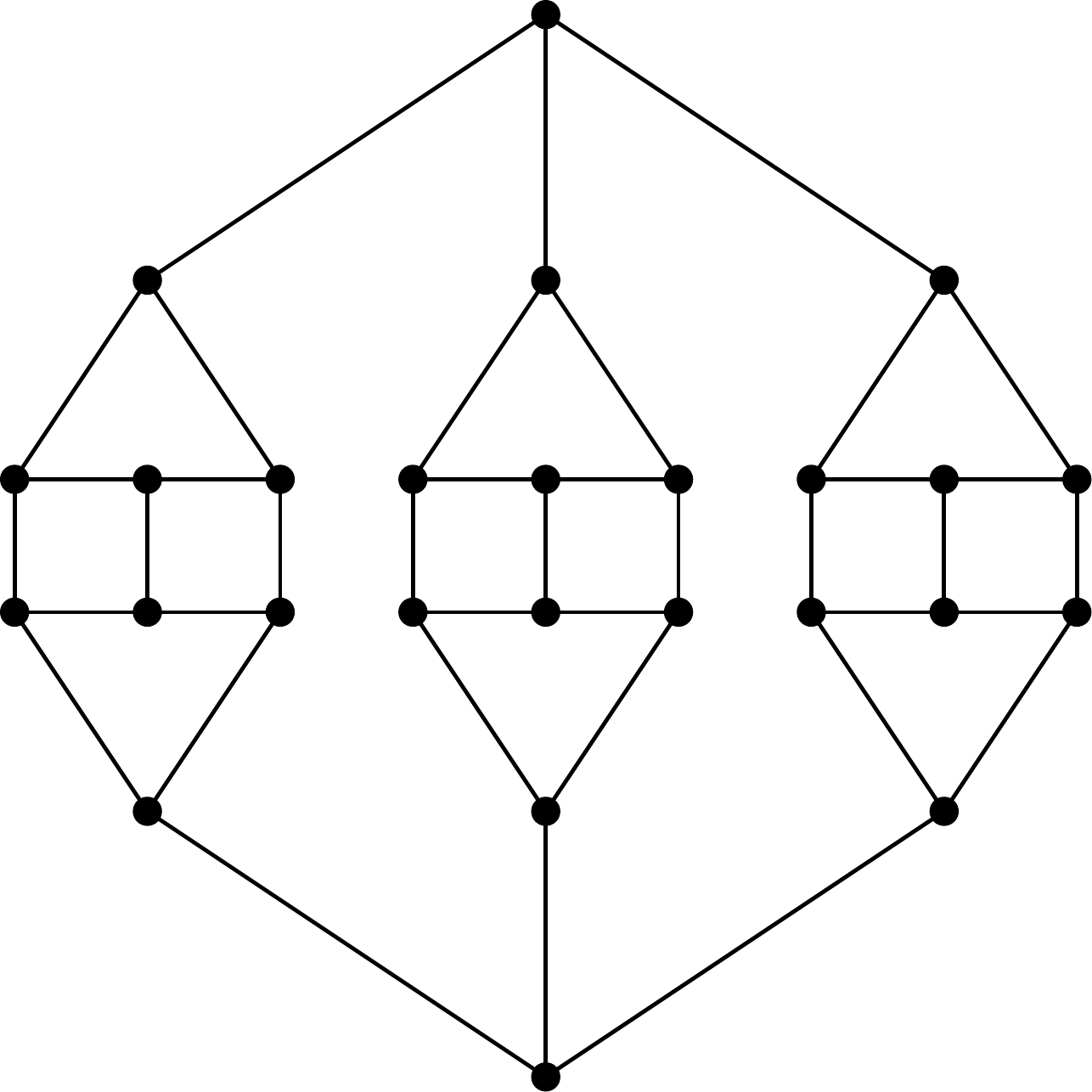}
\caption[A cubic, bipartite, planar, connected, but not Hamiltonian graph.]{A cubic, bipartite, planar and connected graph, which is not Hamiltonian; its construction is inspired by \cite{AkiNisSai80}. It is NP-complete to check, whether cubic, bipartite, planar and connected graphs are Hamiltonian, see \cref{thm-2-conc-NP-complete} and \cref{rem-con->2-conn}.}
\label{f:10HolesNotHamil}
\end{figure}
\subsection*{Acknowledgement}
The author would like to thank Michael Cuntz, Lukas Kühne and Raman Sanyal\footnote{Thank you for proofreading!\label{fn:proofreading}} for organising the Dive into Research of the DFG priority programme SPP 2458 Combinatorial Synergies, during which he familiarised himself with Barnette graphs. He would also like to thank Daniele Agostini\footref{fn:proofreading}, Simon Brandhorst\footref{fn:proofreading}, Jonathan Frey, Veronika Körber\footref{fn:proofreading}, Hannah Markwig\footref{fn:proofreading}\footnote{Thank you for the great supervision!}, Florian Rieg\footnote{Thank you for the introduction in graph theory!}, Jan Stricker\footref{fn:proofreading} and Jasmin Walizadeh for the helpful discussions. The author acknowledges support by the DFG SFB-TRR 195 Symbolic Tools in Mathematics and their Application.
\section{Preliminaries}
\begin{definition}[Hamiltonian Graph]
A finite, simple graph is called \emph{Hamiltonian} if it has a \emph{Hamiltonian cycle}, which is a cycle that visits every vertex of the graph exactly once.
\end{definition}
\begin{definition}[Cubic Graph]
A simple graph is called \emph{cubic} if every vertex is adjacent to precisely three vertices.
\end{definition}
\begin{definition}[Vertex Colouring]
A \emph{vertex colouring} of a simple graph is a map from the set of vertices to a set of colours such that adjacent vertices differ in colour.
	
The \emph{chromatic number} of a finite, simple graph is the minimal cardinality of a set of colours, such that a vertex colouring of the graph with this set of colours exists.
\end{definition}
\begin{definition}[Bipartite Graph]
A finite, simple graph is called \emph{bipartite} if its chromatic number is at most $2$.
\end{definition}
\begin{definition}[Planar Graph]
A finite, simple graph is called \emph{planar} if there exists an embedding of the graph in $\mathbb{R}^{2}$ such that edges do not cross.
\end{definition}
\begin{definition}[$k$-Connected Graph]
A finite, simple graph is called \emph{connected} if for every pair of vertices, there exists a path from one vertex to the other.

Let $k\in\mathbb{N}^{*}$. A finite, simple graph is called \emph{$k$-connected} if it is connected and stays so, if less than $k$ vertices are removed from the graph.
\end{definition}
\begin{definition}[Size of a Face]
Let $n\in\mathbb{N}$. The \emph{size} of a face of an embedding of a finite, simple, planar graph in $\mathbb{R}^{2}$ is the number of faces that are adjacent to this faces.
\end{definition}
\begin{definition}[Dual Graph]
Let $G$ be a finite, simple, planar graph. Fix an embedding $\varphi$ of $G$ in $\mathbb{R}^{2}$. The \emph{dual graph} of $G$ with respect to $\varphi$ is a graph whose vertices correspond to the faces of the embedding. Two vertices are adjacent if their corresponding faces are adjacent.
\end{definition}
\begin{definition}[Graph metric]
Let $G$ be a finite, simple, connected graph. We define the \emph{distance} between any two vertices $v$ and $w$ of $G$ as the number of edges on the shortest path in $G$ connecting them.
\end{definition}
\begin{remark}
The distance between vertices of a finite, simple, connected graph is a metric on the set of vertices.
\end{remark}
\begin{definition}[Neighbourhood of a Subgraph]
Let $G$ be a finite, simple, connected graph and let $G''\subseteq G'\subseteq G$ be subgraphs of $G$. $G'$ is called a \emph{neighbourhood} of $G''$ if the set of vertices of $G'$ is a neighbourhood of the set of vertices of $G''$ with respect to the graph metric of $G$.
\end{definition}
\section{Lifting Hamiltonian Cycles via Graph Substitutions}
\label{sec:LiftingHamiltonianCycles}
This section contains the computational aspects of this paper. Here, we introduce the notion of graph substitutions and give a computational proof of \cref{lem:graphsubstitutions}, which will play a crucial role in the proof of the main \cref{thm-maintheorem}.
\begin{definition}[Graph Substitution]
A \emph{graph substitution} is a quadruple $(G,H,\varphi,\psi)$ consisting of two finite, simple, planar graphs $G$ and $H$ and embeddings $\varphi$ and $\psi$ of $G$ and $H$ in $\mathbb{R}^2$, respectively, such that
\begin{itemize}
\item $G$ has more vertices than $H$,
\item the unbounded, connected component of $\mathbb{R}^2\setminus\varphi(G)$ coincides with the unbounded, connected component of $\mathbb{R}^2\setminus\psi(H)$, and
\item for all vertices $v$ of $G$ such that $\varphi(v)$ is adjacent to the unbounded, connected component of $\mathbb{R}^2\setminus\varphi(G)$ it holds that $\psi^{-1}(\varphi(v))$ is a vertex of $H$ of the same degree as $v$ and vice versa for $H$.
\end{itemize}
\end{definition}
\begin{example}
Several graph substitutions are given in \cref{f:20} and \crefrange{f:18}{f:13ca}. A posteriori, \crefrange{lem:tech1}{lem:tech3} explain where $G$ comes from. The corresponding $H$ was developed by hand for each $G$.
\end{example}
\begin{lemma}[Lifting Hamiltonian Cycles via Graph Substitutions]\label{lem:graphsubstitutions}
Let $G$ and $H$ be finite, simple, cubic, bipartite, planar and connected graphs. Fix embeddings $\varphi$ and $\psi$ of $G$ and $H$ in $\mathbb{R}^{2}$, respectively, and suppose that every face of each embedding is of size up to $8$. Suppose, furthermore, that there exists subgraphs $G'$ of $G$ and $H'$ of $H$ such that $G\setminus G'=H\setminus H'$, $\varphi_{\mid G\setminus G'}=\psi_{\mid H\setminus H'}$ and $(G',H',\varphi_{\mid G'},\psi_{\mid H'})$ is one of the graph substitutions given in \cref{f:20} and \crefrange{f:18}{f:13ca}. Then, any Hamiltonian cycle in $H$ can be changed in a neighbourhood of $H'$ in order to obtain a Hamiltonian cycle in $G$. In particular, if $H$ is Hamiltonian, then so is $G$.
\end{lemma}
\begin{proof}
The proof is computational. The purpose of the computer program is to check whether Hamiltonian cycles can be lifted via the graph substitutions. The source code and the graph substitutions can be found in \cite{My-Github-Repo}. The computer algebra system used is sage, see \cite{Sage}.

The algorithm works in the following way. The input is the part $G'$ of $G$ and the substituted part $H'$ of $H$ as given in Appendix \labelcref{sec:GraphSubstitutions} for each graph substitution. Here, the open ends of each part of a graph are given as vertices of degree $1$. It is necessary that the labels of the vertices of degree $1$ match those of the substituted neighbourhood. The algorithm returns a string stating whether Hamiltonian cycles can be lifted via the given graph substitution and, if so, how many cases it took to check the given graph substitution. These numbers of cases can be found in \cref{t:NumberOfCases}. After a short piece of starting-code, the algorithm works recursively.

It takes a list of open ends, i.e. of labels of vertices of degree $1$ and iterates over every possibility for a Hamiltonian cycle to enter and exit the part of the graph. To do this, the algorithm takes a copy of the part of the graph and removes all degree-$1$-vertices, whose corresponding open ends shall not be contained in the Hamiltonian cycle, that is, the Hamiltonian cycle does not pass through this edge. The others are pairwise connected, each with one further vertex in between, forcing any Hamiltonian cycle to pass through these pairs of edges. This is done in order to check whether this neighbourhood can be part of a Hamiltonian graph, such that a Hamiltonian cycle of this graph passes through precisely the selected open ends of the neighbourhood. The same is done for the part of the substituted graph.

Using sagemath's build-in methods, the algorithm then checks, whether these parts are Hamiltonian or not. That is, does a Hamiltonian cycle on the graph exist, that enters and exits the neighbourhood, respectively the substituted neighbourhood, in precisely the selected open ends? Or can the existence of such a Hamiltonian cycle be ruled out?

If the substituted part with the fixed open ends is Hamiltonian and the original neighbourhood with the fixed open ends is not, the algorithm chooses two adjacent open ends of the graphs and calls the recursion step again with a face of size $4$, a face of size $6$ and a face of size $8$ placed at the position of the two chosen adjacent open ends in both neighbourhoods. In this way, the recursion step is called again with an enlarged neighbourhood, in the hope that it works this time, i.e. the new substituted part with the open ends is not Hamiltonian or the new original part with the fixed open ends is Hamiltonian. Otherwise, the recursion is called again, et cetera.

The place where to attach the new face to the neighbourhood is clearly mathematically irrelevant, however, the number of cases to distinguish varies a lot from a small finite number over large finite numbers up to infinity, including possibly infinite descents.

Therefore, we have tested and implemented heuristic methods to choose this position. In the end, we used two different heuristics depending on the graph substitution, see the GitHub-page \cite{My-Github-Repo} for further details.

If the original part had not been Hamiltonian and did not have any open ends, the algorithm would have called an error saying that Barnette's \cref{conj:Barnette} is false, as this original graph would then be a counterexample. Luckily, this error was never called during the calculation.

If, on the contrary, the substituted part is not Hamiltonian or the original part is Hamiltonian, this substitution worked in this case, that is every Hamiltonian cycle on $H$ with the given neighbourhoods of $H'$ and $G'$ and the given conditions on where the Hamiltonian cycle enters and exits the neighbourhood, respectively the substituted neighbourhood, lifts from $H$ to $G$. In this case, the case-counter is increased by one.

The rest of the source code is mainly made up by additional case checkers in order to make the code run faster.
\end{proof}
\begin{table}[htp]
\centering
\begin{tabular}{l|r}
Substitution&Number of Cases\\\hline
\cref{f:20}&$50$\\
\cref{f:18}&$374.704$\\
\cref{f:22}&$8$\\
\cref{f:23b}&$3$\\
\cref{f:25ab}&$1$\\
\cref{f:25aa}&$8$\\
\cref{f:16c}&$322$\\
\cref{f:16b}&$272.791.991$\\
\cref{f:14cb}&$20$\\
\cref{f:13cb}&$126$\\
\cref{f:14b}&$322$\\
\cref{f:1b}&$50$\\
\cref{f:1cbc}&$1$\\
\cref{f:1cbbb}&$1$\\
\cref{f:1cbba}&$1.964$\\
\cref{f:1cbab}&$8$\\
\cref{f:1cbaa}&$468.362$\\
\cref{f:1ca}&$1.964$\\
\cref{f:subtestexample}&$20.319$\\
\cref{f:2}&$65.392.683$\\
\cref{f:13b}&$1.842$\\
\cref{f:1a}&$322$\\
\cref{f:14a}&$1.871$\\
\cref{f:13a}&$10.848$\\
\cref{f:13ca}&$834$\\\hline
Sum&$339.068.624$
\end{tabular}
\caption{Number of cases the calculation took per substitution}
\label{t:NumberOfCases}
\end{table}
\section{Suitable Graph Substitutions}
\crefrange{lem:tech1}{lem:tech3} are technical and fragment the proof of the main \cref{thm-maintheorem}, making it hopefully easier to understand. The following \cref{def:n-face} is used to make the notation shorter and will be used in proofs only.
\begin{definition}[$n$-Face]\label{def:n-face}
We call a face of an embedding of a planar graph $G$ a \emph{$n$-face} for some integer $n\in\mathbb{N}$, if it is of size $n$.
\end{definition}
\begin{lemma}\label{lem:tech1}
Let $G$ be a finite, simple graph that is cubic, bipartite, planar and connected. Fix an embedding $\varphi$ of $G$ in $\mathbb{R}^{2}$ and suppose that every face of the embedding is of size up to $8$. Suppose, furthermore, that there exists a face of size $4$ that is neither adjacent to any face of size $8$ nor adjacent to four faces of size $4$.
	
Then, there exists a finite, simple, cubic, bipartite, planar and connected graph $H$ together with an embedding $\psi$ of $H$ in $\mathbb{R}^{2}$ and subgraphs $G'$ of $G$ and $H'$ of $H$ such that $G\setminus G'=H\setminus H'$, $\varphi_{\mid G\setminus G'}=\psi_{\mid H\setminus H'}$ and $(G',H',\varphi_{\mid G'},\psi_{\mid H'})$ is one of the graph substitutions given in \cref{f:20}, in \crefrange{f:18}{f:25aa} and in \cref{f:14cb}.
\end{lemma}
\begin{proof}
We change $G$ locally in a neighbourhood of the $4$-face in order to obtain $H$. More precisely, we consider a suitable graph substitution according to \cref{f:4NotAdjacentToAny8Face} and obtain $H$ from $G$ by replacing the neighbourhood by its substitution. Since $G$ is bipartite, every face of the embedding of $G$ is of size $4$, $6$ or $8$. We note that $H$ must be finite, simple, cubic, bipartite, planar and connected.
	
We note, that in some of the referred Figures, there are more faces drawn than stated in the flowchart in \cref{f:4NotAdjacentToAny8Face}. This is because there is no choice for the remaining faces, so there is no restriction made in the proof. For instance, suppose there is precisely one $6$-face adjacent to the $4$-face and there is another $6$-face that is adjacent to the $6$-face on the opposite side of the $4$-face. Since there is no $8$-face adjacent to the $4$-face, the other three faces that are adjacent to the $4$-face are $4$-faces as well. In this case, the flowchart refers to \cref{f:25ab}. However, despite the $4$-face, the two $6$-faces and the three neighbouring $4$-faces, \cref{f:25ab} also shows an $8$-face. This is because there exists another $n$-face that is adjacent to the three neighbouring $4$-faces. By construction, this $n$-face is adjacent to at least $8$ vertices. Therefore, this $n$-face can neither be a $4$- nor a $6$-face and, thus, is forced to be an $8$-face.
\begin{figure}[ht]
\centering
\begin{tikzpicture}
\node(d1)[draw, rectangle] at (0,0) {How many $6$-faces are adjacent to the $4$-face?};
\node(s1)[draw, rounded rectangle] at (5,-2) {See \cref{f:18}};
\node(s2)[draw, rounded rectangle] at (2,-2) {See \cref{f:20}};
\node(d2)[draw, rectangle] at (-2.5,-2) {Are the two $6$-faces adjacent?};
\node(s3)[draw, rounded rectangle] at (5,-4) {See \cref{f:22}};
\node(d3)[draw, rectangle,align=center] at (-2.5,-4) {Thus, the $4$-face is adjacent to\\two neighbouring $4$-faces. What\\other $n$-face is adjacent to those?};
\node(s4)[draw, rounded rectangle] at (-1,-6) {See \cref{f:14cb}};
\node(s5)[draw, rounded rectangle] at (-4,-6) {See \cref{f:23b}};
\node(d4)[draw, rectangle, align=center] at (0,-8) {What $n$-face is adjacent to the $6$-face\\on the opposite side of the $4$-face?\\It cannot be a $4$-face as $G$ is cubic\\and every face is a $4$-, $6$- or $8$-face.};
\node(s6)[draw, rounded rectangle] at (2,-6) {See \cref{f:25ab}};
\node(s7)[draw, rounded rectangle] at (5,-6) {See \cref{f:25aa}};
\coordinate(p1) at (-6,0);
\coordinate(p2) at (-6,-8);
\draw[-to] (d1) -- (s1) node[pos=0.7,above]{$4$};
\draw[-to] (d1) -- (s2) node[pos=0.7,above]{$3$};
\draw[-to] (d1) -- (d2) node[pos=0.7,above]{$2$};
\draw[-to] (d2) -- (s3) node[pos=0.5,below]{No};
\draw[-to] (d2) -- (d3) node[pos=0.5,right]{Yes};
\draw[-to] (d3) -- (s4) node[pos=0.5,right]{$8$};
\draw[-to] (d3) -- (s5) node[pos=0.5,left]{$6$};
\draw[-] (d1) -- (p1) node[pos=0.5,below]{$1$};
\draw[-] (p1) -- (p2);
\draw[-to] (p2) -- (d4);
\draw[-to] (d4) -- (s6) node[pos=0.9,below]{$6$};
\draw[-to] (d4) -- (s7) node[pos=0.8,below]{$8$};
\end{tikzpicture}
\caption{Substitution for a $4$-face that is not adjacent to any $8$-face}
\label{f:4NotAdjacentToAny8Face}
\end{figure}
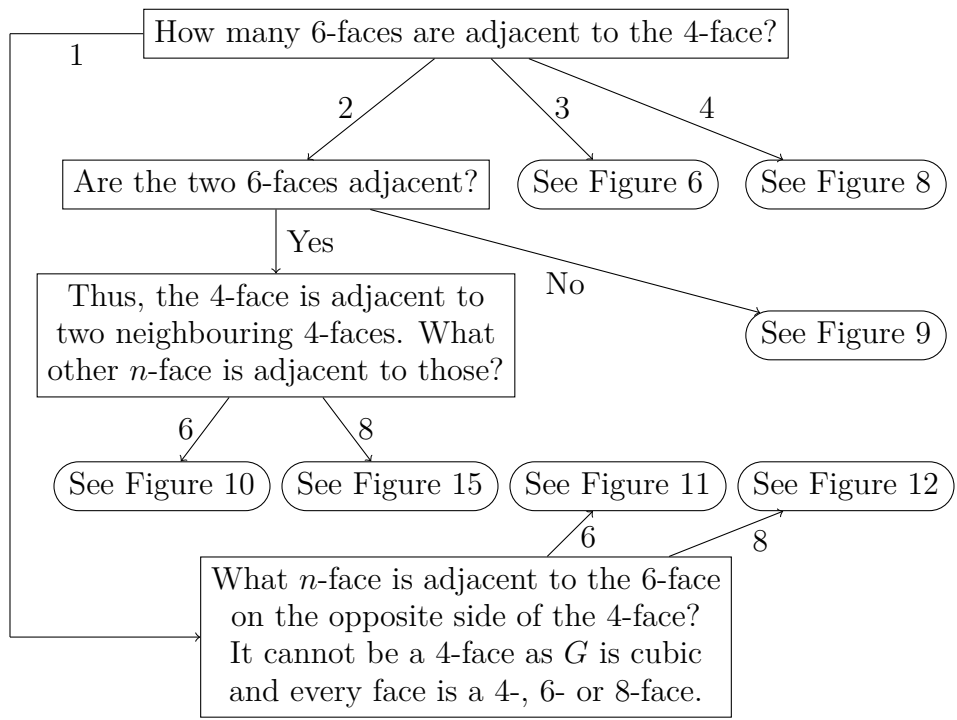
\end{proof}
\begin{lemma}\label{lem:tech2}
Let $G$ be a finite, simple graph that is cubic, bipartite, planar and connected. Fix an embedding $\varphi$ of $G$ in $\mathbb{R}^{2}$ and suppose that every face of the embedding is of size up to $8$. Suppose, furthermore, that there exists a face of size $4$ that is adjacent to precisely one face of size $8$.

Then, there exists a finite, simple, cubic, bipartite, planar and connected graph $H$ together with an embedding $\psi$ of $H$ in $\mathbb{R}^{2}$ and subgraphs $G'$ of $G$ and $H'$ of $H$ such that $G\setminus G'=H\setminus H'$, $\varphi_{\mid G\setminus G'}=\psi_{\mid H\setminus H'}$ and $(G',H',\varphi_{\mid G'},\psi_{\mid H'})$ is one of the graph substitutions given in \crefrange{f:16c}{f:1ca}.
\end{lemma}
\begin{proof}
Again, we change $G$ locally in a neighbourhood of the $4$-face in order to obtain $H$. More precisely, we consider a suitable graph substitution according to \cref{f:4AdjacentTo8Face} and obtain $H$ from $G$ by replacing the neighbourhood by its substitution. Since $G$ is bipartite, every face of the embedding of $G$ is a $4$-, $6$- or $8$-face. We note that $H$ must be finite, simple, cubic, bipartite, planar and connected.
\begin{figure}[hpt]
\centering
\begin{tikzpicture}
\node(d1)[draw, rectangle, align=center] at (0,0) {The $4$-face and the $8$-face meet in two vertices. How\\many other $4$-faces are adjacent to one of these vertices?};
\node(d2)[draw, rectangle, align=center] at (0,-2) {What $n$-face is adjacent to the $4$-face\\on the opposite side of the $8$-face?};
\node(s1)[draw, rounded rectangle] at (-5.5,-2) {See \cref{f:16c}};
\node(s2)[draw, rounded rectangle] at (5.5,-2) {See \cref{f:16b}};
\node(d4)[draw, rectangle, align=center] at (0,-4) {What $n$-face is adjacent to the $4$-face\\on the opposite side of the $8$-face?};
\node(d5)[draw, rectangle, align=center] at (0,-6) {By assumption, this $4$-face is adjacent\\to two $4$-faces and one $6$-face. What\\further $n$-face is is adjacent to?};
\node(s3)[draw, rounded rectangle] at (-5.5,-6) {See \cref{f:14cb}};
\node(s4)[draw, rounded rectangle] at (5.5,-6) {See \cref{f:13cb}};
\node(s5)[draw, rounded rectangle] at (-5.5,-4) {See \cref{f:14b}};
\node(d3)[draw, rectangle, align=center] at (0,-8) {Does there exist a $6$-face\\adjacent to the first $4$-face?};
\node(s6)[draw, rounded rectangle] at (5.5,-8) {See \cref{f:1b}};
\node(d6)[draw, rectangle, align=center] at (0,-10) {The $n$-face that is adjacent to all faces\\that are adjacent to the first $4$-face has\\at least $6$ edges. How many are there?};
\node(s7)[draw, rounded rectangle] at (5.5,-10) {See \cref{f:1ca}};
\node(d7)[draw, rectangle, align=center] at (0,-12) {What $n$-face is adjacent to the $8$-face\\on the opposite side of the first $4$-face?};
\node(s8)[draw, rounded rectangle] at (-5.5,-10) {See \cref{f:1cbc}};
\node(d8)[draw, rectangle, align=center] at (-3.3,-14.25) {The only $n$-face that is adjacent\\to this $6$-face and the $8$-face has at\\least $6$ edges. How many are there?};
\node(d9)[draw, rectangle, align=center] at (3.3,-14.25) {The only $n$-face that is adjacent\\to the two $8$-faces has at least\\$6$ edges. How many are there?};
\node(s9)[draw, rounded rectangle] at (-4.5,-16.25) {See \cref{f:1cbbb}};
\node(s10)[draw, rounded rectangle] at (-1.5,-16.25) {See \cref{f:1cbba}};
\node(s11)[draw, rounded rectangle] at (1.5,-16.25) {See \cref{f:1cbab}};
\node(s12)[draw, rounded rectangle] at (4.5,-16.25) {See \cref{f:1cbaa}};
\coordinate(p1) at (7.25,0);
\coordinate(p2) at (7.25,-4);
\coordinate(p3) at (-7.25,0);
\coordinate(p4) at (-7.25,-8);
\coordinate(p5) at (-7.25,-12);
\coordinate(p7) at (-7.25,-10);
\draw[-to] (d1) -- (d2) node[pos=0.5,left]{$0$};
\draw[-to] (d2) -- (s1) node[pos=0.5,below]{$4$};
\draw[-to] (d2) -- (s2) node[pos=0.5,below]{$6$};
\draw[-] (d1) -- (p1) node[pos=0.5,above]{$1$};
\draw[-] (p1) -- (p2);
\draw[-to] (p2) -- (d4);
\draw[-] (p3) -- (p4);
\draw[-to] (p4) -- (d3);
\draw[-] (d1) -- (p3) node[pos=0.5,above]{$2$};
\draw[-to] (d4) -- (d5) node[pos=0.5,left]{$4$};
\draw[-to] (d5) -- (s3) node[pos=0.5,below]{$6$};
\draw[-to] (d5) -- (s4) node[pos=0.5,below]{$8$};
\draw[-to] (d4) -- (s5) node[pos=0.5,below]{$6$};
\draw[-to] (d3) -- (s6) node[pos=0.5,below]{Yes};
\draw[-to] (d3) -- (d6) node[pos=0.5,left]{No};
\draw[-to] (d6) -- (s7) node[pos=0.5,below]{$8$};
\draw[-to] (d6) -- (d7) node[pos=0.5,left]{$6$};
\draw[-to] (d7) -- (d8) node[pos=0.7,above]{$6$};
\draw[-to] (d7) -- (d9) node[pos=0.7,above]{$8$};
\draw[-] (d7) -- (p5) node[pos=0.5,below]{$4$};
\draw[-] (p5) -- (p7);
\draw[-to] (p7) -- (s8);
\draw[-to] (d8) -- (s9) node[pos=0.5,right]{$6$};
\draw[-to] (d8) -- (s10) node[pos=0.5,left]{$8$};
\draw[-to] (d9) -- (s11) node[pos=0.5,right]{$6$};
\draw[-to] (d9) -- (s12) node[pos=0.5,left]{$8$};
\end{tikzpicture}
\caption{Substitution for a $4$-face that is adjacent to precisely one $8$-face}
\label{f:4AdjacentTo8Face}
\end{figure}
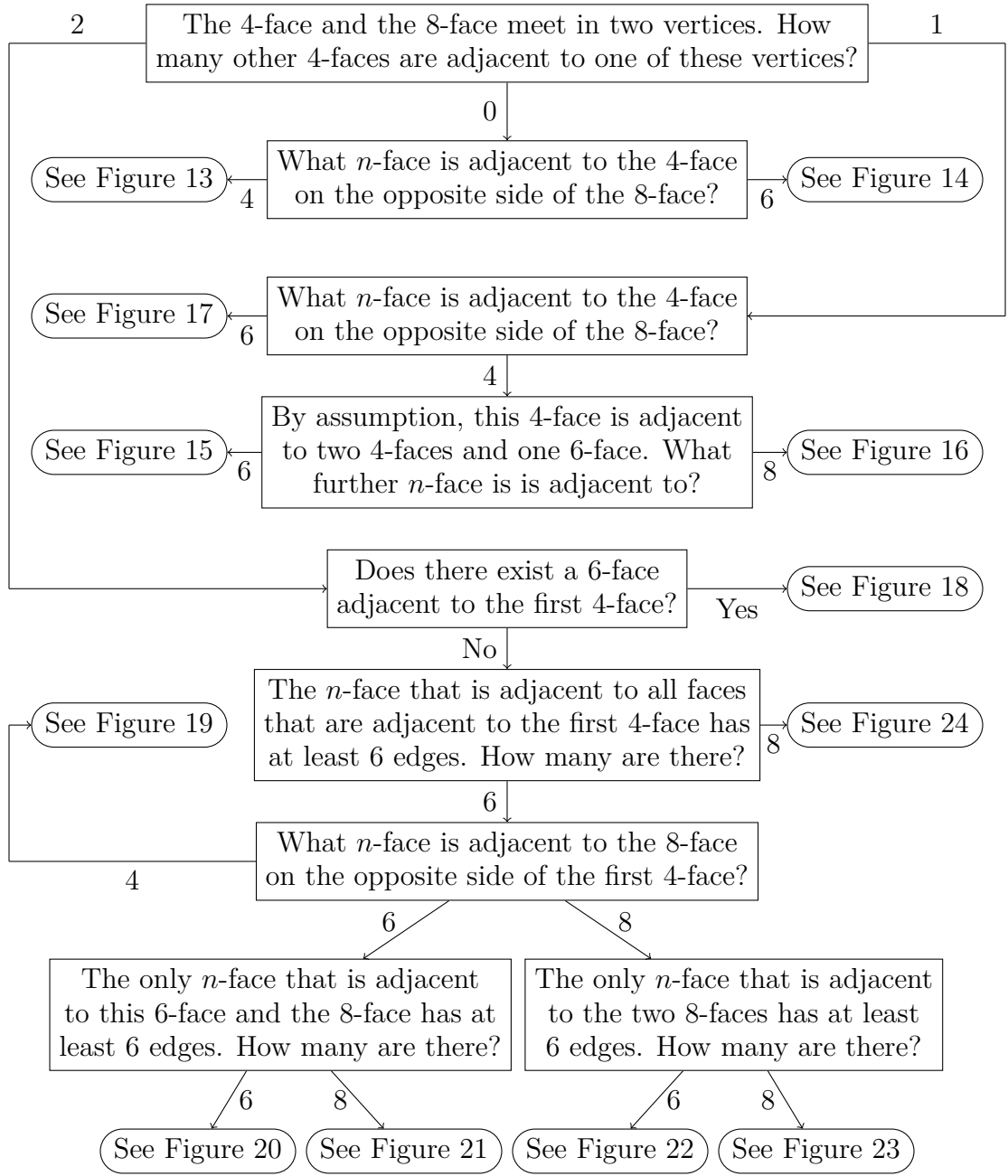
\end{proof}
\begin{lemma}\label{lem:tech3}
Let $G$ be a finite, simple graph that is cubic, bipartite, planar and connected. Fix an embedding $\varphi$ of $G$ in $\mathbb{R}^{2}$ and suppose that every face of the embedding is of size up to $8$. Suppose, furthermore, that there exists a face of size $8$ that is adjacent to three or more faces of size $4$.

Then, there exists a finite, simple, cubic, bipartite, planar and connected graph $H$ together with an embedding $\psi$ of $H$ in $\mathbb{R}^{2}$ and subgraphs $G'$ of $G$ and $H'$ of $H$ such that $G\setminus G'=H\setminus H'$, $\varphi_{\mid G\setminus G'}=\psi_{\mid H\setminus H'}$ and $(G',H',\varphi_{\mid G'},\psi_{\mid H'})$ is one of the graph substitutions given in \crefrange{f:14cb}{f:13ca}.
\end{lemma}
\begin{proof}
Again, we change $G$ locally in a neighbourhood of the $8$-face in order to obtain $H$. More precisely, we consider a suitable graph substitution according to the following. In all cases, we obtain $H$ from $G$ by replacing the neighbourhood by its substitution and we note that $H$ must be finite, simple, cubic, bipartite, planar and connected.

Since $G$ is bipartite, every face of the embedding of $G$ is a $4$-, $6$- or $8$-face. First, we suppose that every two $4$-faces adjacent to the $8$-face are not adjacent. By the box principle, there exist two $4$-faces that are adjacent to the $8$-face at edges that are separated by one other edge of the $8$-face only. By assumption, the other face adjacent to this edge in the middle cannot be a $4$-face. If it is a $6$-face, we follow \cref{f:subtestexample}. Otherwise, it will be an $8$-face and we follow \cref{f:2}.

Now, we suppose that there exist two adjacent $4$-faces that are adjacent to the $8$-face. Let $k\in\{4,6,8\}$ be such that the other face that is adjacent to both $4$-faces is a $k$-face. First, we suppose $k\ne4$. Let $l\in\{4,6,8\}$ be such that one of the $4$-faces is furthermore adjacent to an $l$-face. We consider graph substitutions as presented in \cref{t:Two4AdjacentTo8FaceNo4OnOtherSide}.
\begin{table}[ht]
\centering
\begin{tabular}{c||c|c|c}
&$l=4$&$l=6$&$l=8$\\\hline\hline
$k=6$&See \cref{f:1b}&See \cref{f:14b}&See \cref{f:13b}\\\hline
$k=8$&See \cref{f:1a}&See \cref{f:14a}&See \cref{f:13a}
\end{tabular}
\caption{Substitution for an $8$-face that is adjacent to two adjacent $4$-faces which are not adjacent to a common $4$-face}
\label{t:Two4AdjacentTo8FaceNo4OnOtherSide}
\end{table}

Now, we consider $k=4$. There are precisely two faces that are adjacent to the $8$-face and to one of the two $4$-faces that also are adjacent to the $8$-face. First, we assume that both of them are not a $4$-face. If both of them are $6$-faces, we consider the graph substitution given in \cref{f:14cb}. If only one of them is a $6$-face, the other one is an $8$-face and we follow the graph substitution presented in \cref{f:13cb}. Otherwise, both of them are $8$-faces and we consider the graph substitution given in \cref{f:13ca}. Finally, we assume that at least one of them is a $4$-face. We follow \cref{f:Subst8FaceWithThree4FacesInRowAndAnother4FaceBehindThem} to find a graph substitution.
\begin{figure}[ht]
\centering
\begin{tikzpicture}
\node(d6)[draw, rectangle, align=center] at (0,-0.5) {By assumption, three $4$-faces in a row are adjacent to the $8$-face.\\Thus, there exists one other $n$-face which is adjacent to the\\two marginal $4$-faces. Since this $n$-face is adjacent to the $8$-face\\in two edges and to three $4$-faces, we find $n\ne4$. What is $n$?};
\node(s7)[draw, rounded rectangle] at (3,-2.5) {See \cref{f:1ca}};
\node(d7)[draw, rectangle, align=center] at (0,-4) {What $n$-face is adjacent to the $8$-face\\on the opposite side of the central $4$-face?};
\node(s8)[draw, rounded rectangle] at (-3,-2.5) {See \cref{f:1cbc}};
\node(d8)[draw, rectangle, align=center] at (-3.25,-6.25) {The $n$-face that is adjacent to\\this $6$-face and the $8$-face has at\\least $6$ edges. How many are there?};
\node(d9)[draw, rectangle, align=center] at (3.25,-6.25) {The $n$-face that is adjacent to\\the two $8$-faces has at least\\$6$ edges. How many are there?};
\node(s9)[draw, rounded rectangle] at (-4.5,-8.25) {See \cref{f:1cbbb}};
\node(s10)[draw, rounded rectangle] at (-1.5,-8.25) {See \cref{f:1cbba}};
\node(s11)[draw, rounded rectangle] at (1.5,-8.25) {See \cref{f:1cbab}};
\node(s12)[draw, rounded rectangle] at (4.5,-8.25) {See \cref{f:1cbaa}};
\coordinate(p5) at (-5,-4);
\coordinate(p7) at (-5,-2.5);
\draw[-to] (d6) -- (s7) node[pos=0.15,below]{$8$};
\draw[-to] (d6) -- (d7) node[pos=0.5,left]{$6$};
\draw[-to] (d7) -- (d8) node[pos=0.7,above]{$6$};
\draw[-to] (d7) -- (d9) node[pos=0.7,above]{$8$};
\draw[-] (d7) -- (p5) node[pos=0.5,below]{$4$};
\draw[-] (p5) -- (p7);
\draw[-to] (p7) -- (s8);
\draw[-to] (d8) -- (s9) node[pos=0.5,right]{$6$};
\draw[-to] (d8) -- (s10) node[pos=0.5,left]{$8$};
\draw[-to] (d9) -- (s11) node[pos=0.5,right]{$6$};
\draw[-to] (d9) -- (s12) node[pos=0.5,left]{$8$};
\end{tikzpicture}
\caption{Substitution for an $8$-face that is adjacent to three $4$-faces in a row with another $4$-face which is adjacent to the three $4$-faces}
\label{f:Subst8FaceWithThree4FacesInRowAndAnother4FaceBehindThem}
\end{figure}
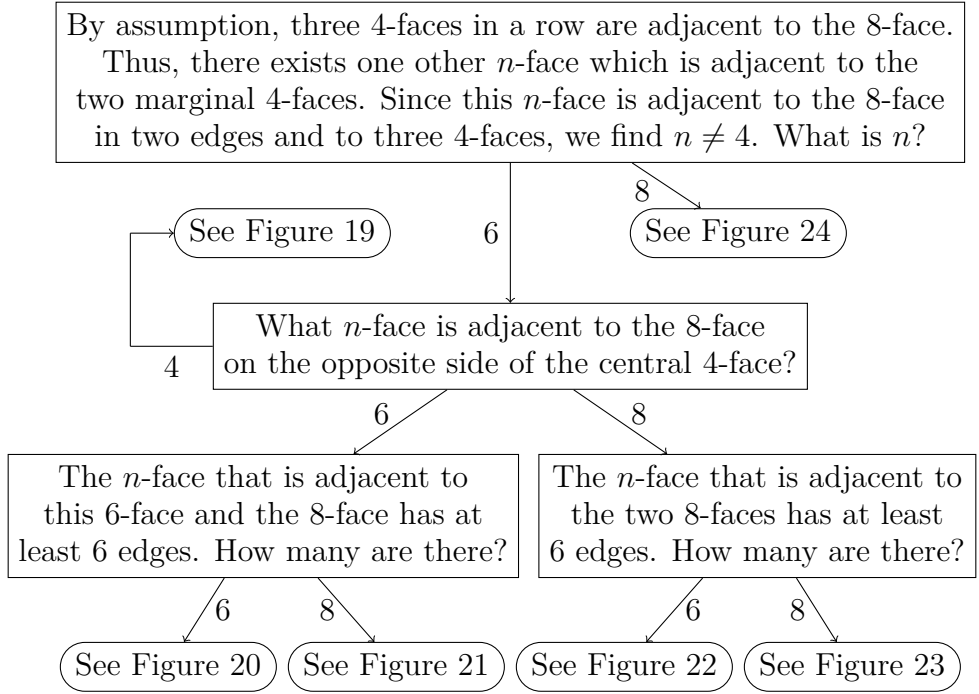
\end{proof}
\section{The Main Theorem}
In this section, we finally combine \cref{lem:graphsubstitutions} and \crefrange{lem:tech1}{lem:tech3} and prove the main \cref{thm-maintheorem}.

\cref{lem:conn->2-conn} is an easy observation, however, for the sake of completeness, it is included here.
\begin{lemma}\label{lem:conn->2-conn}
Every finite, simple, cubic, bipartite, connected graph is $2$-connected.
\end{lemma}
\begin{proof}
For the sake of contradiction, we assume there exists a finite, simple, cubic, bipartite, connected graph that is not $2$-connected. Since the graph is bipartite, we can colour the vertices of the graph in green and red such that two adjacent vertices differ in colour. We choose and mark a vertex such that the graph with this vertex removed is disconnected. We choose a connected component of the graph without the marked vertex, but including the edges that connect the component with the marked vertex, and draw all vertices of this graph, one by one, starting at an arbitrary chosen vertex of in the connected component, together with all three edges leaving each vertex and link them, if necessary. During this drawing process, we colour all open edges, that is edges of which only one of the two adjacent vertices is drawn already, in the same colour as the drawn vertex adjacent to it.

At the beginning of the drawing process, there is no coloured edge, so the difference of green open edges and red open edges is $0$. In each drawing step, this difference is invariant modulo $3$. Thus, we find inductively that the number of green open edges minus the number of red open edges is divisible by three in each step.

However, when we have drawn the connected component completely, we still have open edges. These are adjacent to the marked vertex, so they have the same colour, because the graph is bipartite. Moreover, as the graph without the marked vertex is disconnected, there are precisely one or two such edges. In particular the difference of green open edges and red open edges is not divisible by three. This is a contradiction.
\end{proof}
\begin{theorem}\label{thm-maintheorem}
Let $G$ be a finite, simple graph that is cubic, bipartite, planar and connected. Fix an embedding of the graph in $\mathbb{R}^{2}$ and suppose that every face of the embedding is of size at most $8$. Then $G$ is Hamiltonian.

In particular, every Barnette graph, with faces of size up to $8$ is Hamiltonian.
\end{theorem}
\begin{proof}
For the sake of contradiction, we assume the opposite. Let $G$ be a counterexample of minimal number of vertices, that is, $G$ is a finite, simple graph that is cubic, bipartite, planar and connected, together with a fixed embedding of $G$ in $\mathbb{R}^{2}$, such that every face of the embedding of $G$ is of size up to $8$ and we suppose that $G$ is not Hamiltonian and every other such graph with less vertices than $G$ is Hamiltonian.

Since $G$ is bipartite, every face of the embedding of $G$ is a $4$-, $6$- or $8$-face. We denote by $\nu_{4}$, $\nu_{6}$ and $\nu_{8}$ the number of $4$-, $6$- and $8$-faces, respectively. Applying Euler's formula (see Theorem $4.2.9$ in \cite{Die05}) yields
\begin{align*}
2&=\left(\frac{1}{3}(4\nu_{4}+6\nu_{6}+8\nu_{8})\right)-\left(\frac{1}{2}(4\nu_{4}+6\nu_{6}+8\nu_{8})\right)+\left(\nu_{4}+\nu_{6}+\nu_{8}\right)\\
&=\frac13\nu_{4}-\frac13\nu_{8},
\end{align*}
that is
\begin{align}
\nu_{4}=\nu_{8}+6.\label{eq:nu4=nu8+6}
\end{align}
In the following, we will prove $\nu_{4}\le\nu_{8}$ which is a contradiction to \cref{eq:nu4=nu8+6}.

We assume that there exists a $4$-face that is adjacent to at most one $8$-face. If this $4$-face is adjacent to four other $4$-faces, $G$ will be the graph given in \cref{f:cube}, which is clearly Hamiltonian. Hence, it cannot be equal to our minimal counterexample, which is a contradiction.
\begin{figure}[htp]
\centering
\includegraphics[width=0.2\textwidth]{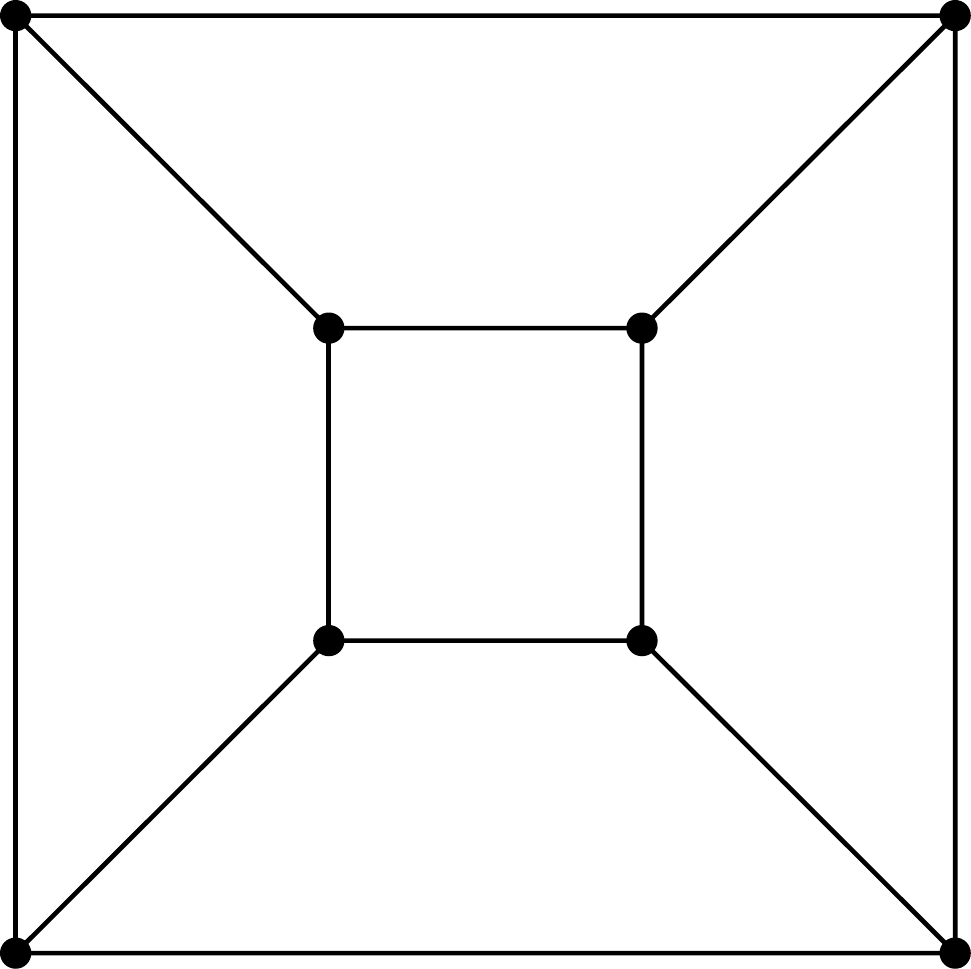}
\caption{The Cube Graph}
\label{f:cube}
\end{figure}

Otherwise, we apply \cref{lem:tech1} or \cref{lem:tech2} and obtain a graph $H$ that has less vertices than $G$ such that it fulfils the conditions given in \cref{thm-maintheorem} and that $G$ and $H$ fulfil the conditions given in \cref{lem:graphsubstitutions}. By the assumed minimality, $H$ has a Hamilton cycle. By \cref{lem:graphsubstitutions}, $G$ is Hamiltonian, which is a contradiction. Therefore, such a $4$-face that is adjacent to at most one $8$-face does not exist in our minimal counterexample.

Now, we assume that there exists an $8$-face that is adjacent to three or more $4$-faces. We apply \cref{lem:tech3} and obtain a graph $H$ that has less vertices than $G$ such that it fulfils the conditions given in \cref{thm-maintheorem} and that $G$ and $H$ fulfil the conditions given in \cref{lem:graphsubstitutions}. By the assumed minimality, $H$ has a Hamilton cycle. By \cref{lem:graphsubstitutions}, $G$ is Hamiltonian, which is a contradiction. Thus, such an $8$-face that is adjacent to three or more $4$-face does not exist in our minimal counterexample.

Combined, we have proven that every $4$-face is adjacent to at least two $8$-faces and that every $8$-face is adjacent to at most two $4$-faces. Applying the box principle yields that for any given set of $4$-faces, the cardinality of the set of $8$-faces that are adjacent to at least one $4$-face in the given set, is larger or equal to the cardinality of the given set of $4$-faces. Hall's marriage theorem (see Theorem $2.1.2$ in \cite{Die05}) implies the existence of a perfect matching of all $4$-faces to a subset of the $8$-faces, where matched faces are adjacent. In particular, this implies $\nu_{4}\le\nu_{8}$ which is a clear contradiction to \cref{eq:nu4=nu8+6} and, hence, finishes the proof.
\end{proof}
\section{Prospect}
\label{sec:prospect}
In this section, we provide a more abstract discussion of the proof of the main \cref{thm-maintheorem} and question, whether and how the proof method could be used to deduce even stronger results.

At first glance, the existence of Hamiltonian cycles seems to be a global property. However, computations suggest that ``large'' Barnette graphs admit ``many'' Hamiltonian cycles. Thus, one might ask whether transforming a Barnette graph only locally to another Barnette graph does preserve any Hamiltonian cycle expect for changes in a neighbourhood of the part of the graph that has been changed. That is, might the existence of Hamiltonian cycles perhaps be a \emph{local property} for Barnette graphs?

In the proof of \cref{lem:graphsubstitutions} we have actually proven this for all not necessarily $3$-connected Barnette graphs in which every face is of size at most $8$.

Therefore, we can hope that the existence of Hamiltonian cycles is a local property for Barnette graphs in which every face is of size at most $10$ or $12$ or perhaps even for all Barnette graphs. This could imply that the methods of the proof of \cref{lem:graphsubstitutions} and, thus, also \cref{thm-maintheorem} could be used to prove stronger results such as the existence of Hamiltonian cycles for Barnette graphs with these bounds or even for all Barnette graphs, that is proving \cref{conj:Barnette}. We would ``simply'' need to list and test several graph substitutions such as we did in the proof of \cref{lem:graphsubstitutions}. However, there could be the following three obstacles on the road.

Firstly, we distinguished a large amount of cases. The number of cases to distinguish might grow rapidly when the bound lifts from $8$ to $10$, to $12$, etc. and it is unclear whether the number of cases remains finite when considering more and more, perhaps even all Barnette graphs. So, \emph{the computer-aided calculations might become unfeasible for modern computers}. However, the code used for the calculations of this paper does not make use of symmetry arguments. Despite that, the code as such is also far from being optimal and could, hence, be optimized, including for instance GPU-multithreading. Moreover, it is still unclear whether the number of cases to consider grows so fast or if it might be bounded, even for all Barnette graphs.

Secondly, when considering all Barnette graphs by only finitely many cases, we would necessarily have an upper bound for the size of a neighbourhood. Since we would consider arbitrarily large faces, \emph{there would exist faces that cannot be covered by any neighbourhood}. Therefore, we would need more arguments that can limit the number of cases to consider. The use of Euler's formula in the proof of \cref{thm-maintheorem} is an example of such arguments, but it is likely that more arguments would be needed to do this.

Finally, we did not \emph{require $3$-connectedness} in the proof of \cref{lem:graphsubstitutions} and \cref{thm-maintheorem}. In fact, without the additional assumption of $3$-connectedness, the generalisation from $8$ to $10$ is false already, see for instance \cref{f:10HolesNotHamil}. Similar to \cite{Goo75}, one could, perhaps, just consider $3$-connected graph substitutions, that is, avoiding substitutions like \cref{f:25ab}, \cref{f:1cbc} and \cref{f:1cbbb}. To check these substitutions with a computer, the program code would need to be rewritten in such a way, that it only checks $3$-connected graphs.
\printbibliography
\appendix
\section{A List of Graph Substitutions}
\label{sec:GraphSubstitutions}
Here, we list all graphs and graph substitutions needed for \cref{lem:graphsubstitutions}, \crefrange{lem:tech1}{lem:tech3} and \cref{thm-maintheorem}.
\begin{figure}[thp]
\centering
\begin{subfigure}{0.3\textwidth}
\centering
\includegraphics[width=0.9\textwidth]{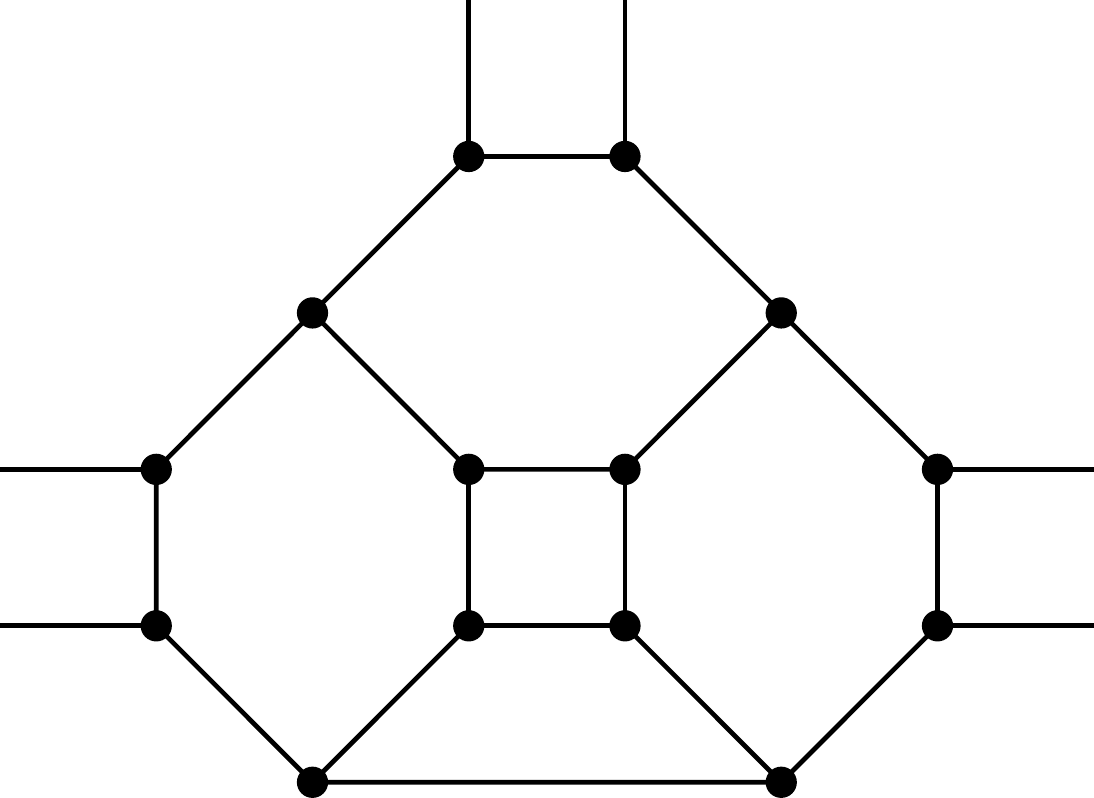}
\caption{Neighbourhood}
\label{f:o20}
\end{subfigure}
\begin{subfigure}{0.3\textwidth}
\centering
\includegraphics[width=0.9\textwidth]{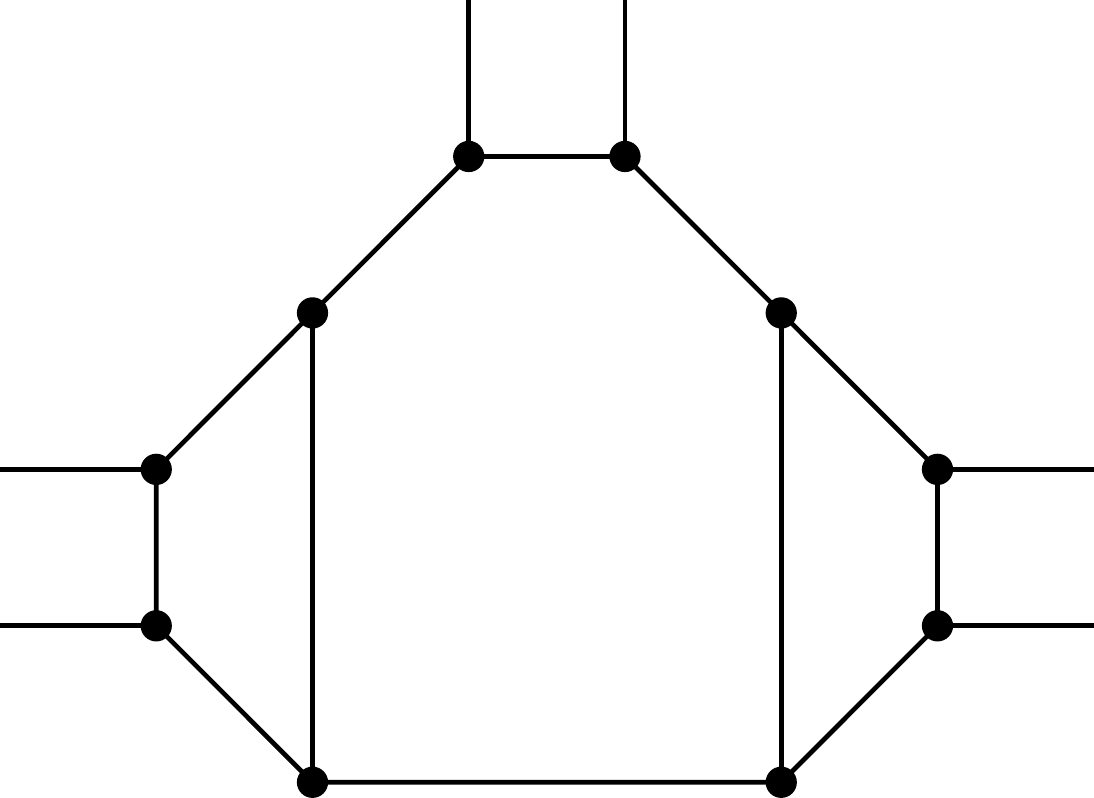}
\caption{Substitution}
\label{f:20sub}		
\end{subfigure}
\caption{Graph substitution}
\label{f:20}
\end{figure}

The graph substitution shown in \cref{f:20} should be interpreted as follows. Given a finite, simple graph that is cubic, bipartite, planar and connected. Consider a fixed embedding of the graph in $\mathbb{R}^{2}$ and suppose that every face of the embedding is of size at most $8$. Suppose furthermore that there exists a neighbourhood where the graph is as depicted in \cref{f:o20}. This neighbourhood is separated from the graph by its boundary as given in \cref{f:20boundary}. \cref{f:20sub} shows another part of a graph with the same boundary. Therefore, we can replace the neighbourhood in the graph, as given in \cref{f:o20}, by the one given in \cref{f:20sub} to obtain a new graph which, again, is cubic, bipartite, planar and connected. Moreover, the embedding stays fixed, and every face of the new graph is, again, of size at most $8$. The two graphs differ in this neighbourhood only. We note that the new graph has less vertices than the original graph. All other graph substitutions presented in this section work alike.
\begin{figure}[thp]
\centering
\includegraphics[width=0.27\textwidth]{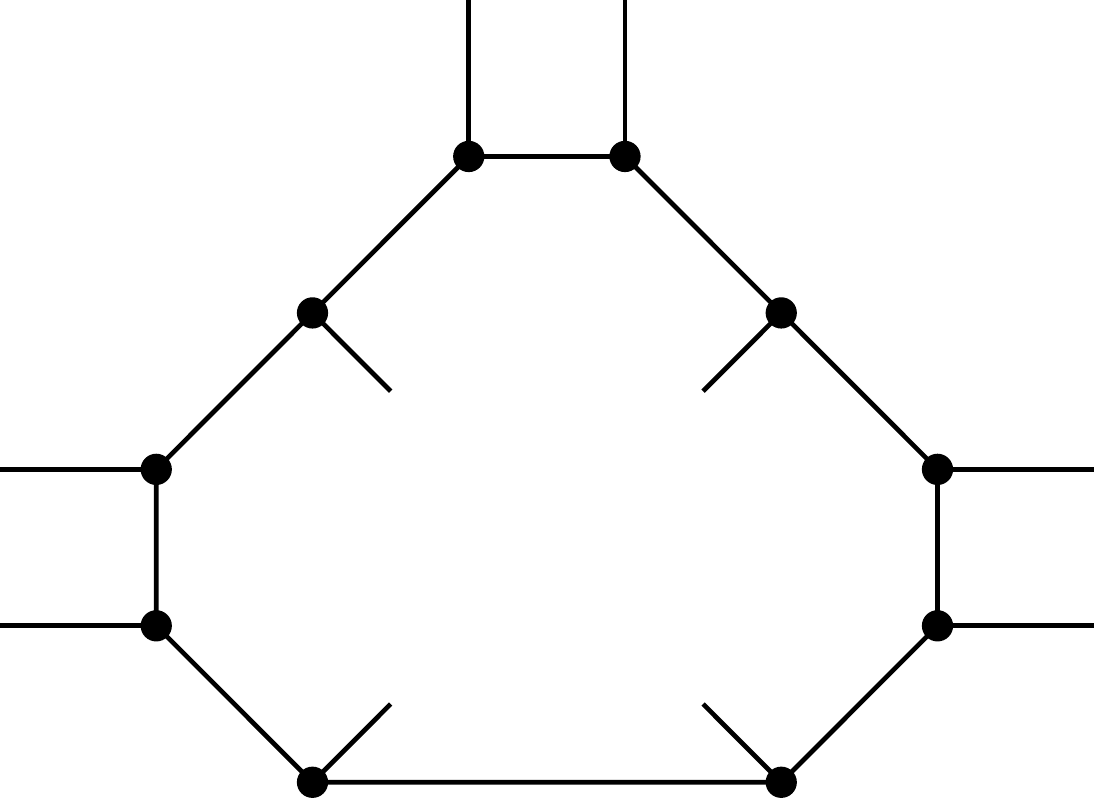}
\caption{The boundary of the graph substitution given in \cref{f:20}}
\label{f:20boundary}
\end{figure}

\begin{figure}[hp]
\centering
\captionbox{Graph substitution\label{f:18}}[0.49\textwidth]{
\subcaptionbox{Neighbourhood\label{f:o18}}[0.24\textwidth]{\includegraphics[width=0.22\textwidth]{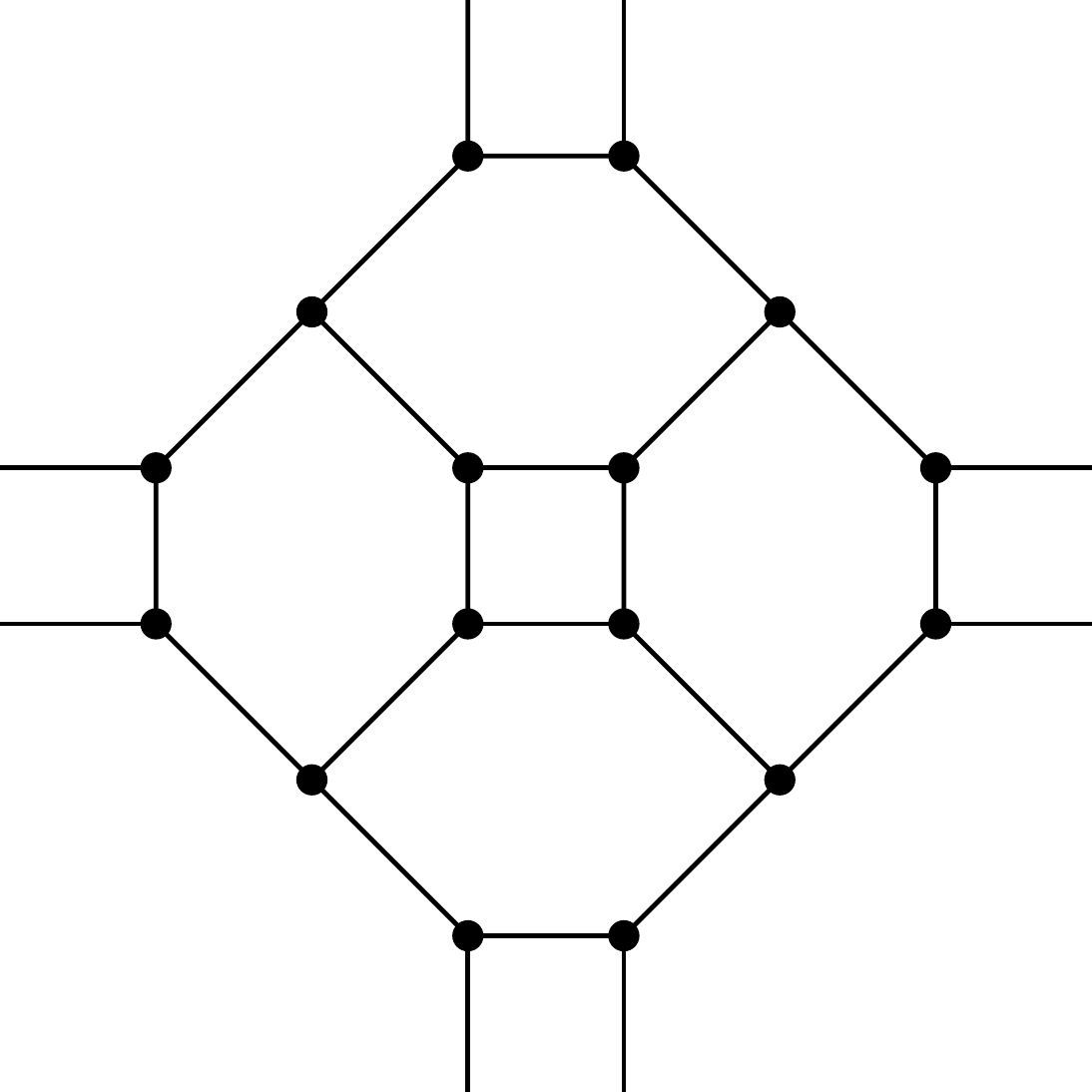}}
\subcaptionbox{Substitution\label{f:18sub}}[0.24\textwidth]{\includegraphics[width=0.22\textwidth]{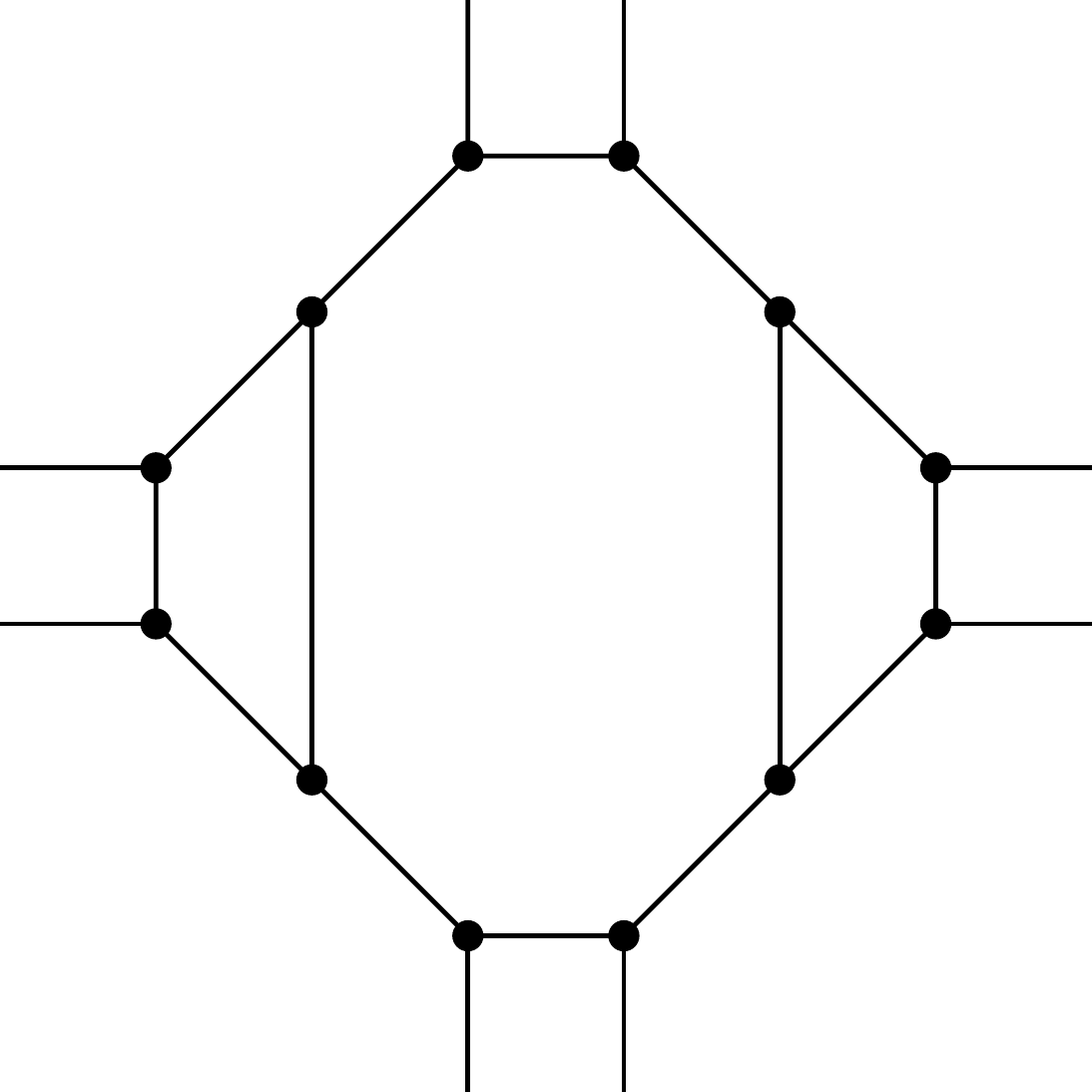}}
}
\captionbox{Graph substitution\label{f:22}}[0.49\textwidth]{
\subcaptionbox{Neighbourhood\label{f:o22}}[0.24\textwidth]{\includegraphics[width=0.22\textwidth]{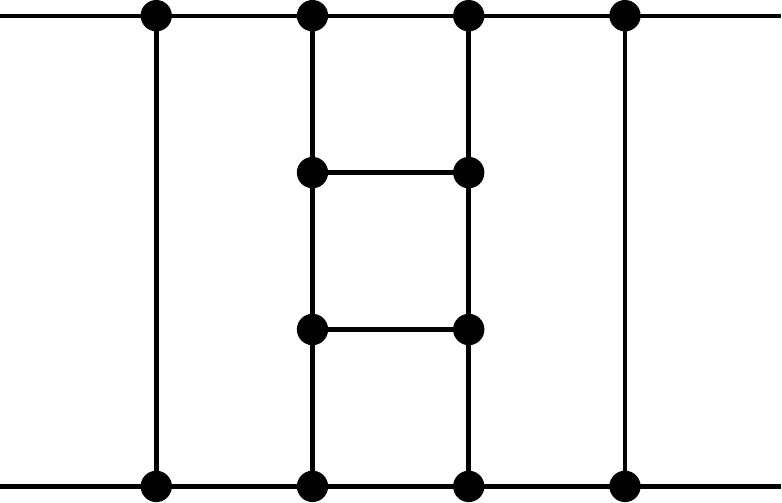}}
\subcaptionbox{Substitution\label{f:22sub}}[0.24\textwidth]{\includegraphics[width=0.22\textwidth]{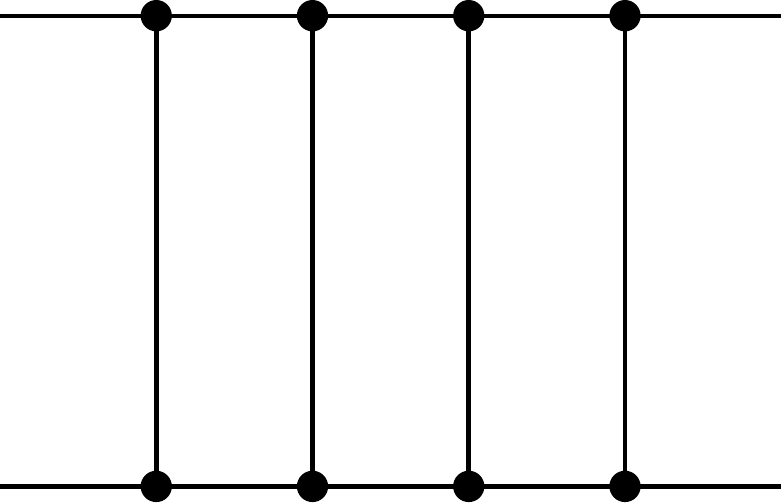}}
}
\end{figure}
\begin{figure}[hp]
\centering
\captionbox{Graph substitution\label{f:23b}}[0.49\textwidth]{
\subcaptionbox{Neighbourhood\label{f:o23b}}[0.24\textwidth]{\includegraphics[width=0.22\textwidth]{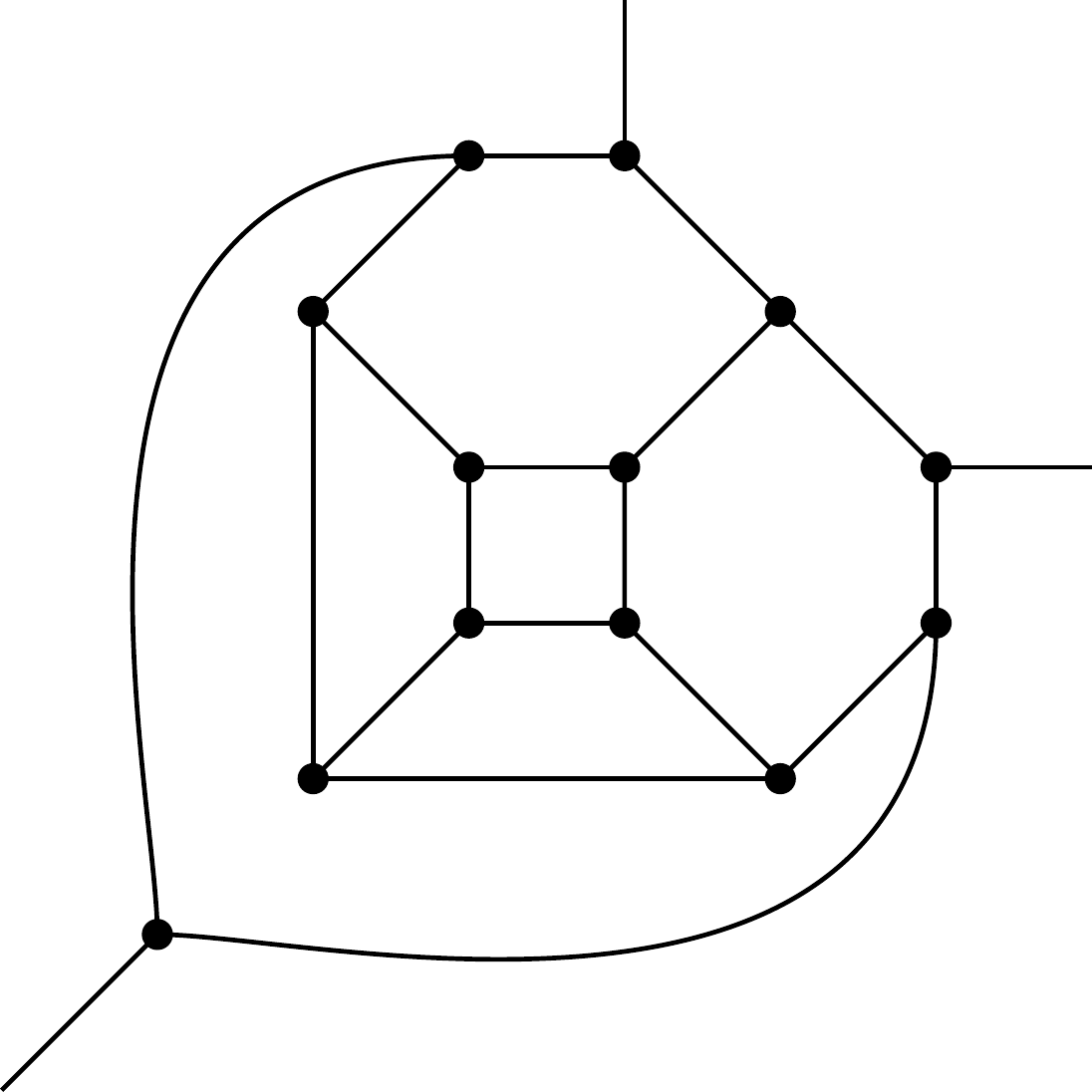}}
\subcaptionbox{Substitution\label{f:23bsub}}[0.24\textwidth]{\includegraphics[width=0.22\textwidth]{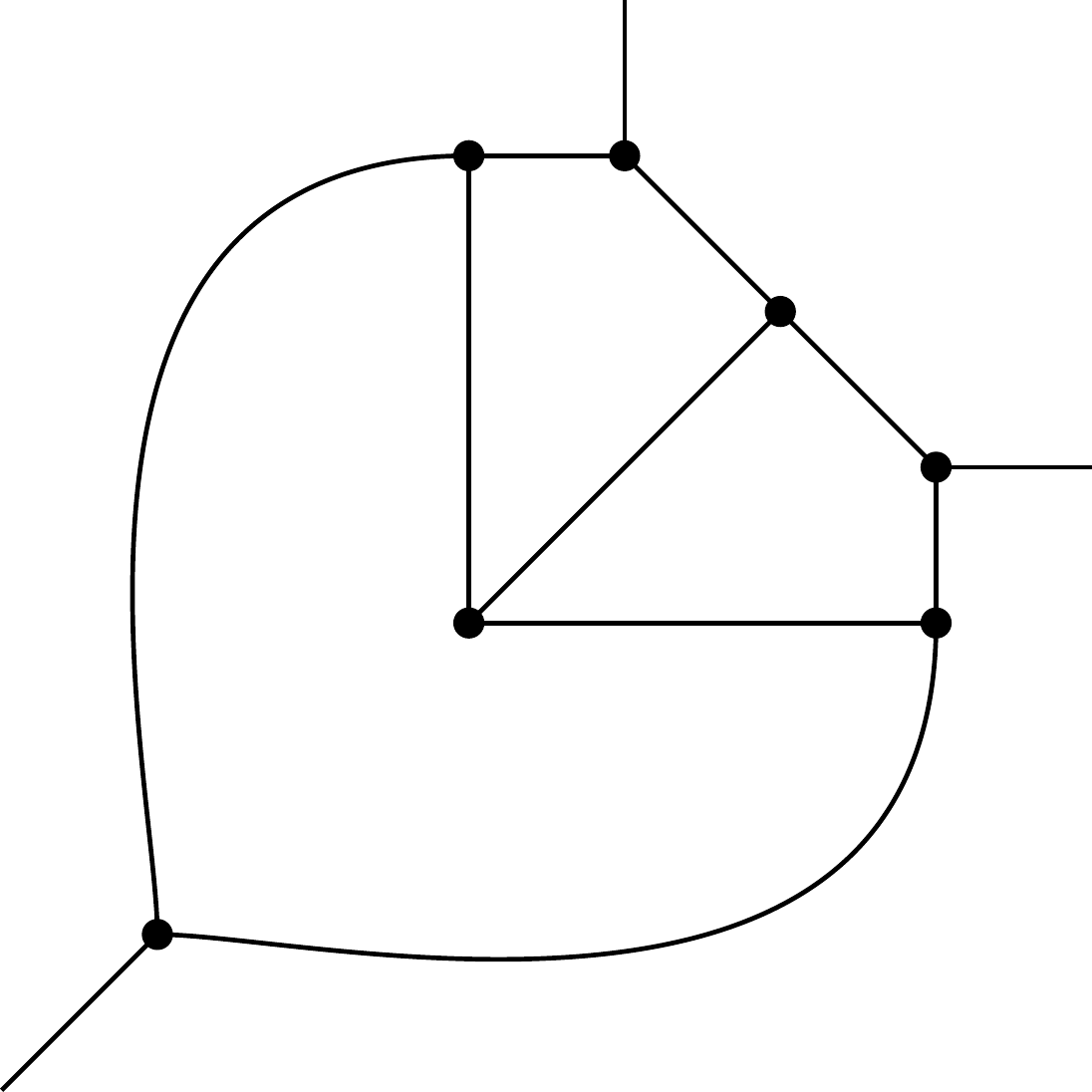}}
}
\captionbox{Graph substitution\label{f:25ab}}[0.49\textwidth]{
\subcaptionbox{Neighbourhood\label{f:o25ab}}[0.24\textwidth]{\includegraphics[width=0.22\textwidth]{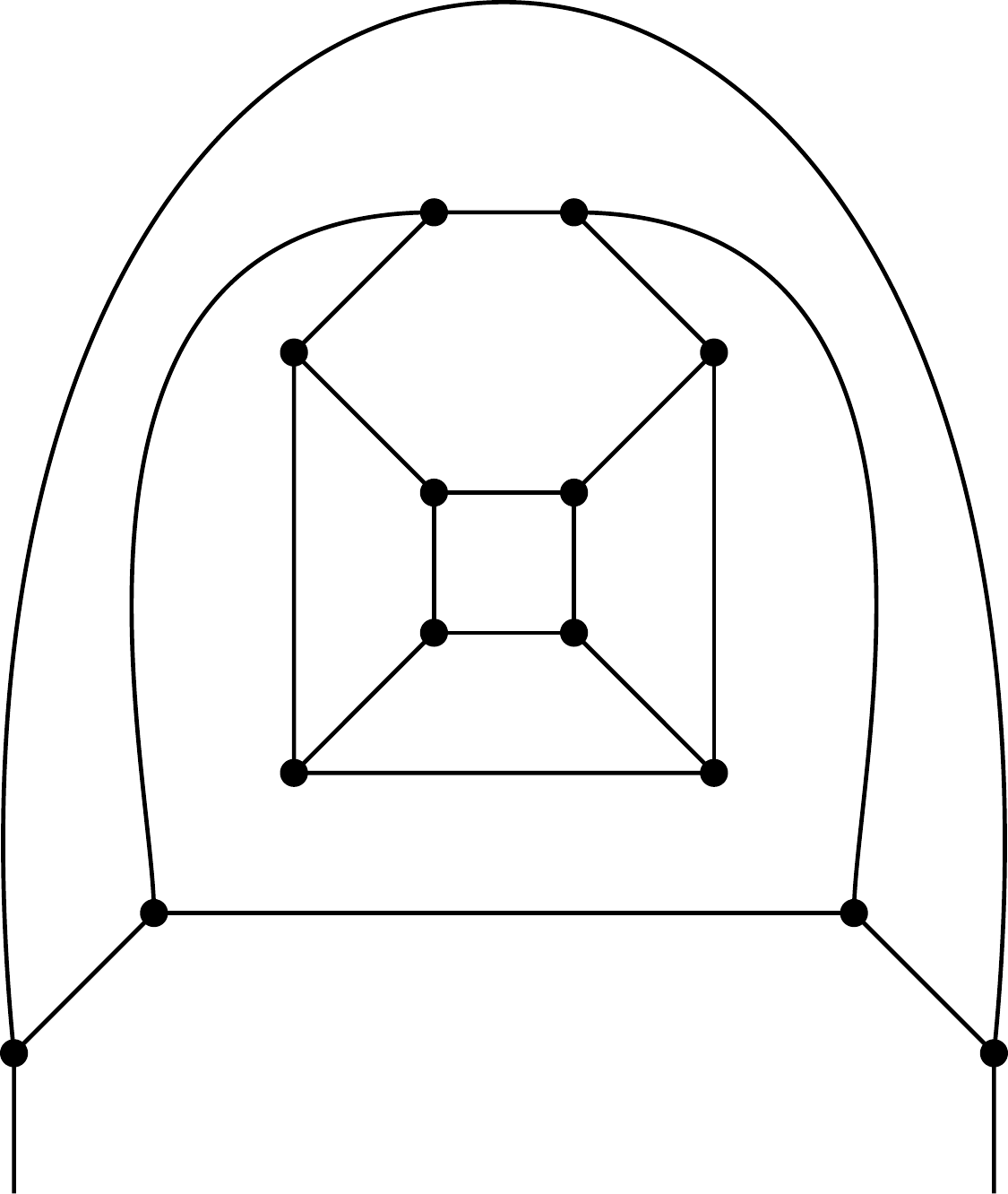}}
\subcaptionbox{Substitution\label{f:25absub}}[0.24\textwidth]{\includegraphics[width=0.22\textwidth]{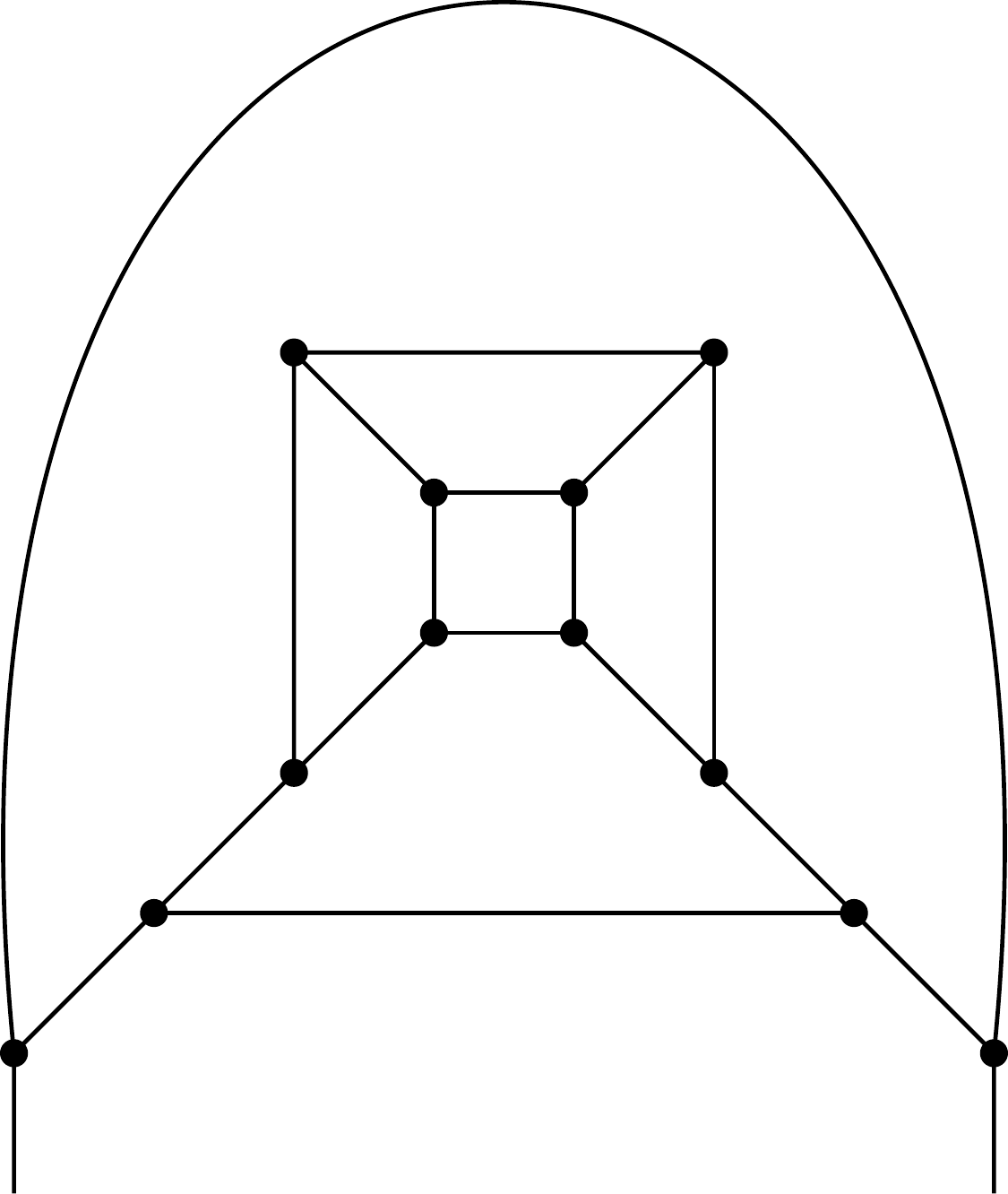}}
}
\end{figure}
\begin{figure}[hp]
\centering
\captionbox{Graph substitution\label{f:25aa}}[0.49\textwidth]{
\subcaptionbox{Neighbourhood\label{f:o25aa}}[0.24\textwidth]{\includegraphics[width=0.22\textwidth]{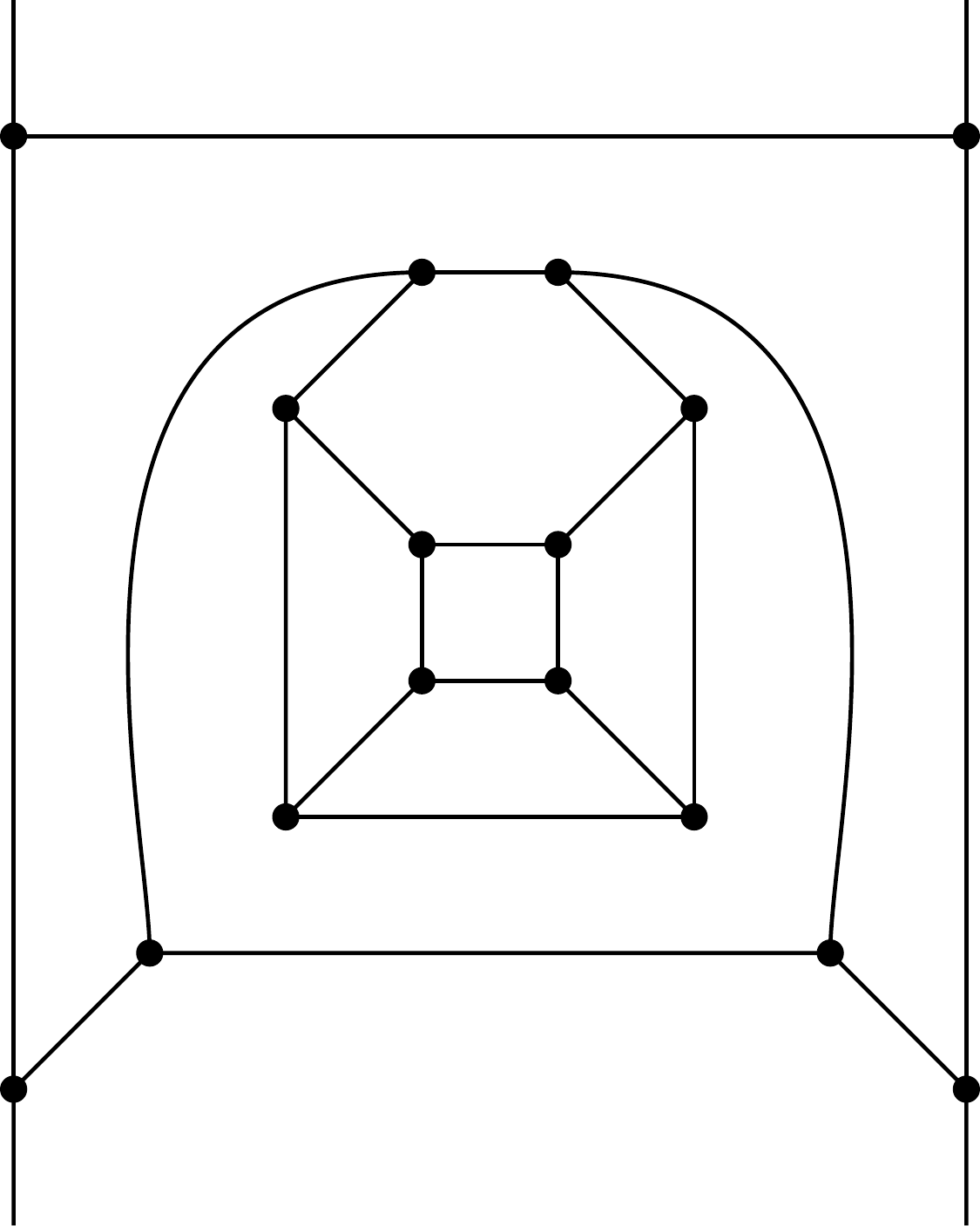}}
\subcaptionbox{Substitution\label{f:25aasub}}[0.24\textwidth]{\includegraphics[width=0.22\textwidth]{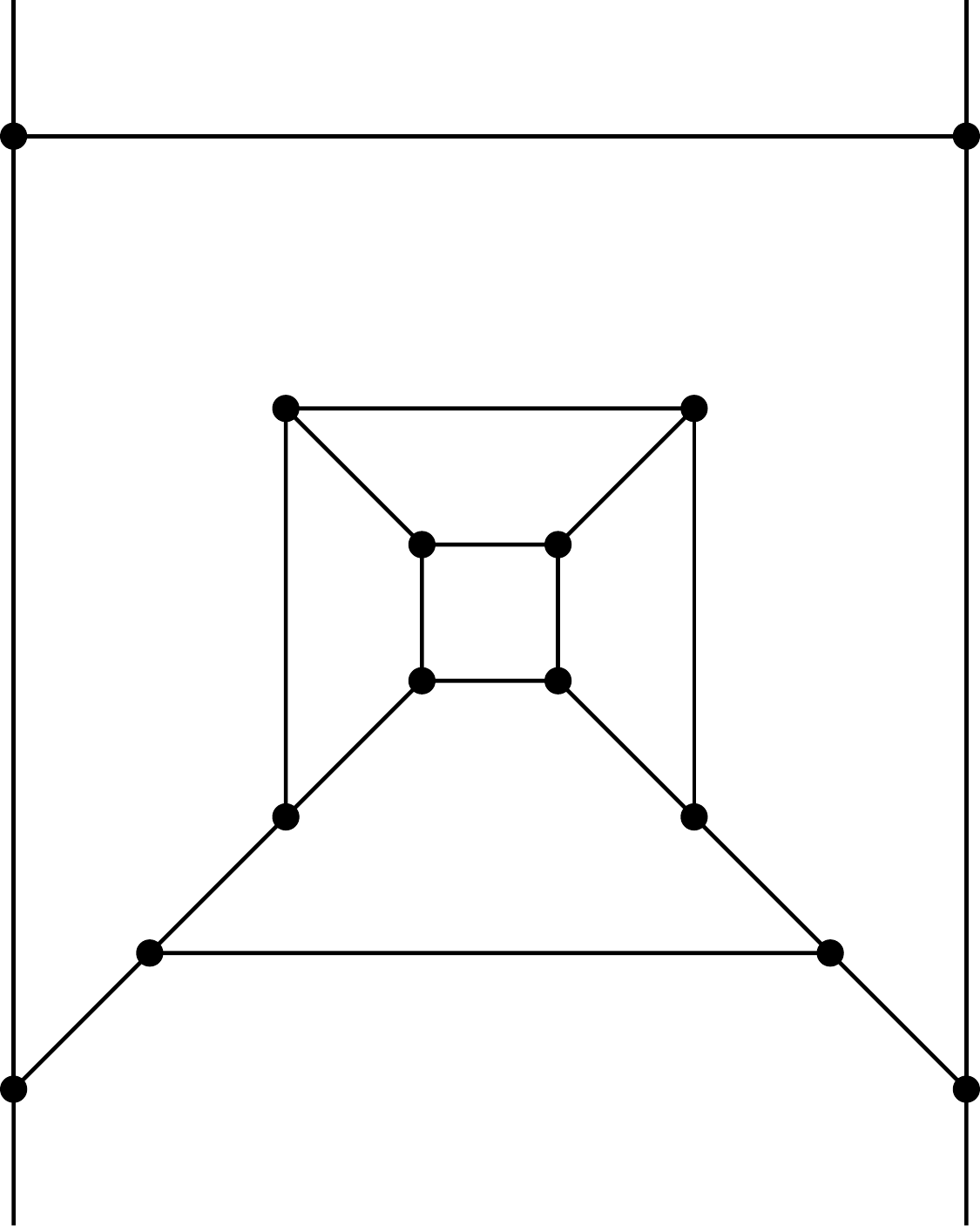}}
}
\captionbox{Graph substitution\label{f:16c}}[0.49\textwidth]{
\subcaptionbox{Neighbourhood\label{f:o16c}}[0.24\textwidth]{\includegraphics[width=0.22\textwidth]{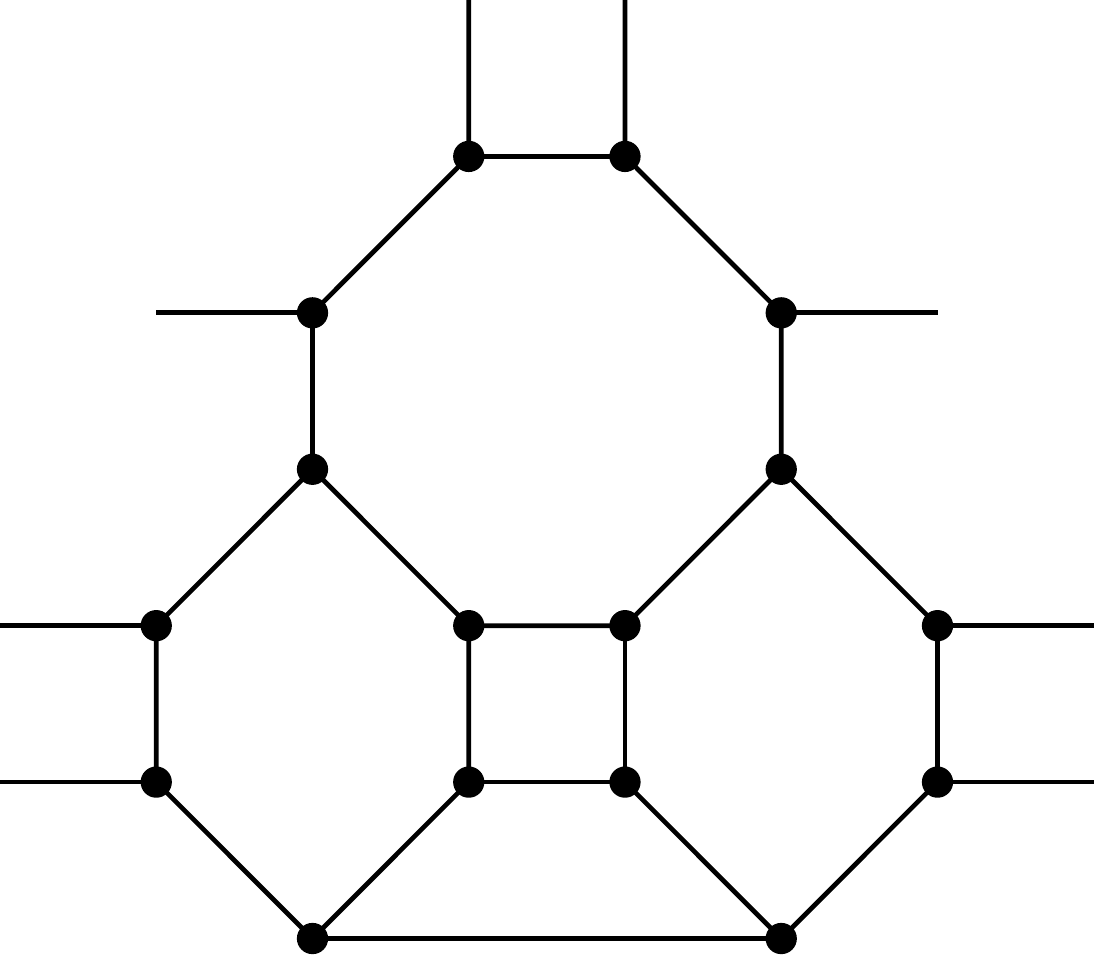}}
\subcaptionbox{Substitution\label{f:16csub}}[0.24\textwidth]{\includegraphics[width=0.22\textwidth]{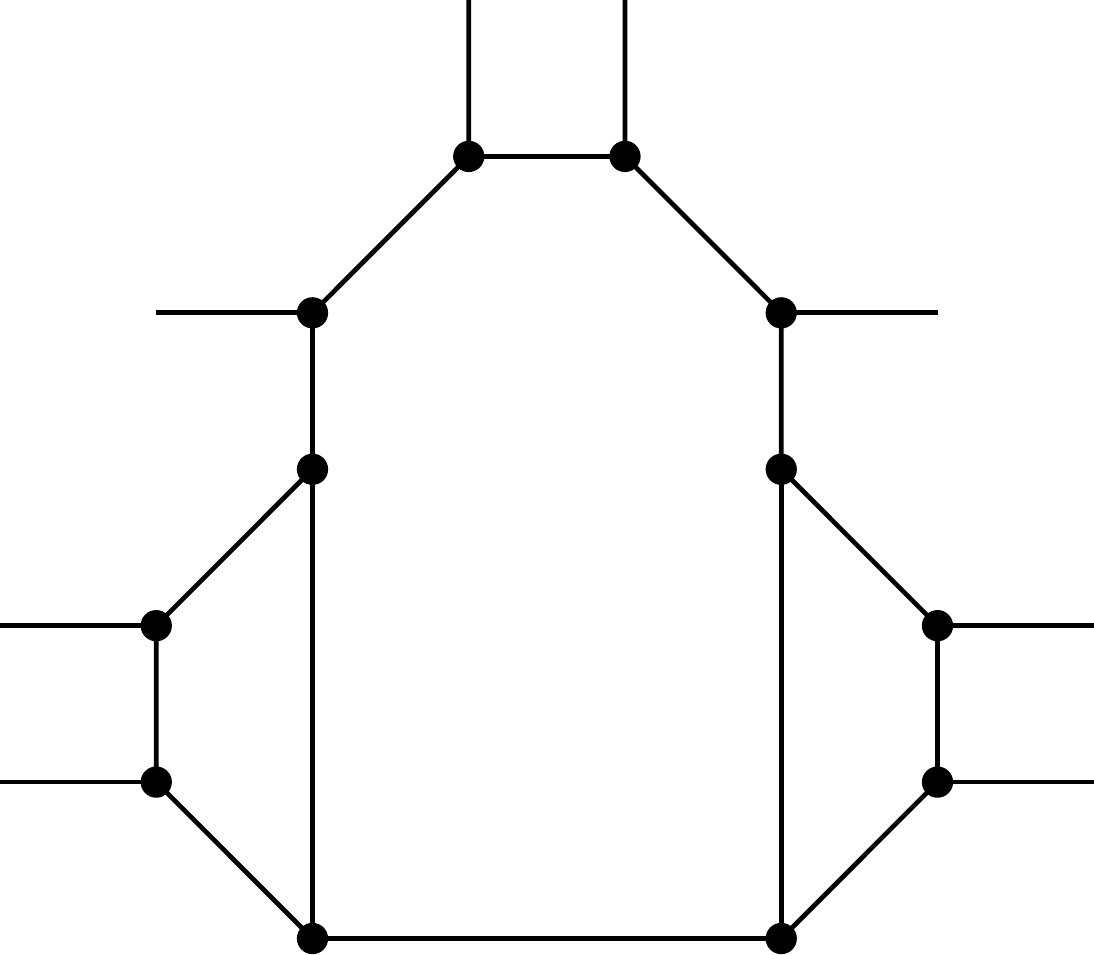}}
}
\end{figure}
\begin{figure}[hp]
\centering
\captionbox{Graph substitution\label{f:16b}}[0.49\textwidth]{
\subcaptionbox{Neighbourhood\label{f:o16b}}[0.24\textwidth]{\includegraphics[width=0.22\textwidth]{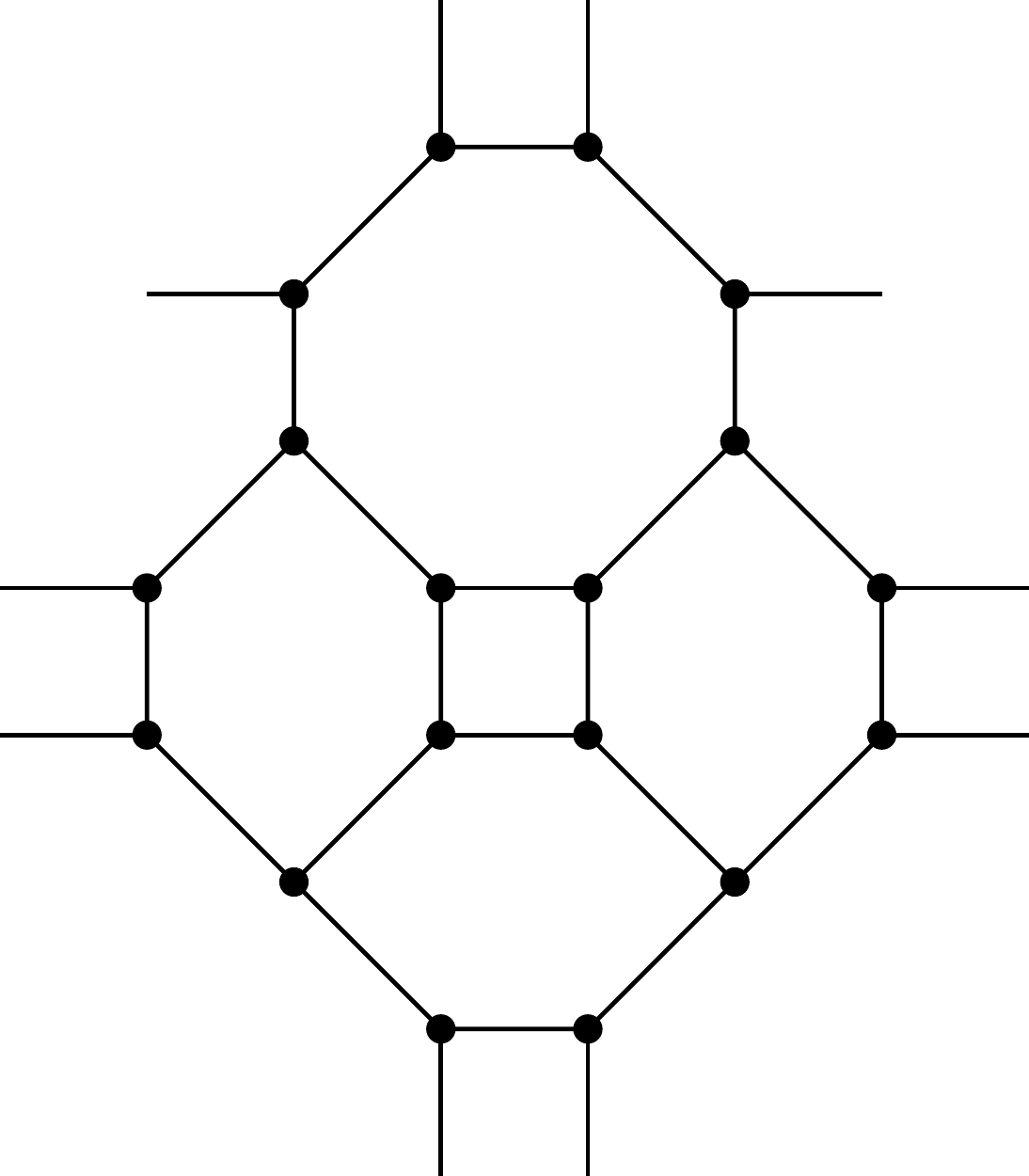}}
\subcaptionbox{Substitution\label{f:16bsub}}[0.24\textwidth]{\includegraphics[width=0.22\textwidth]{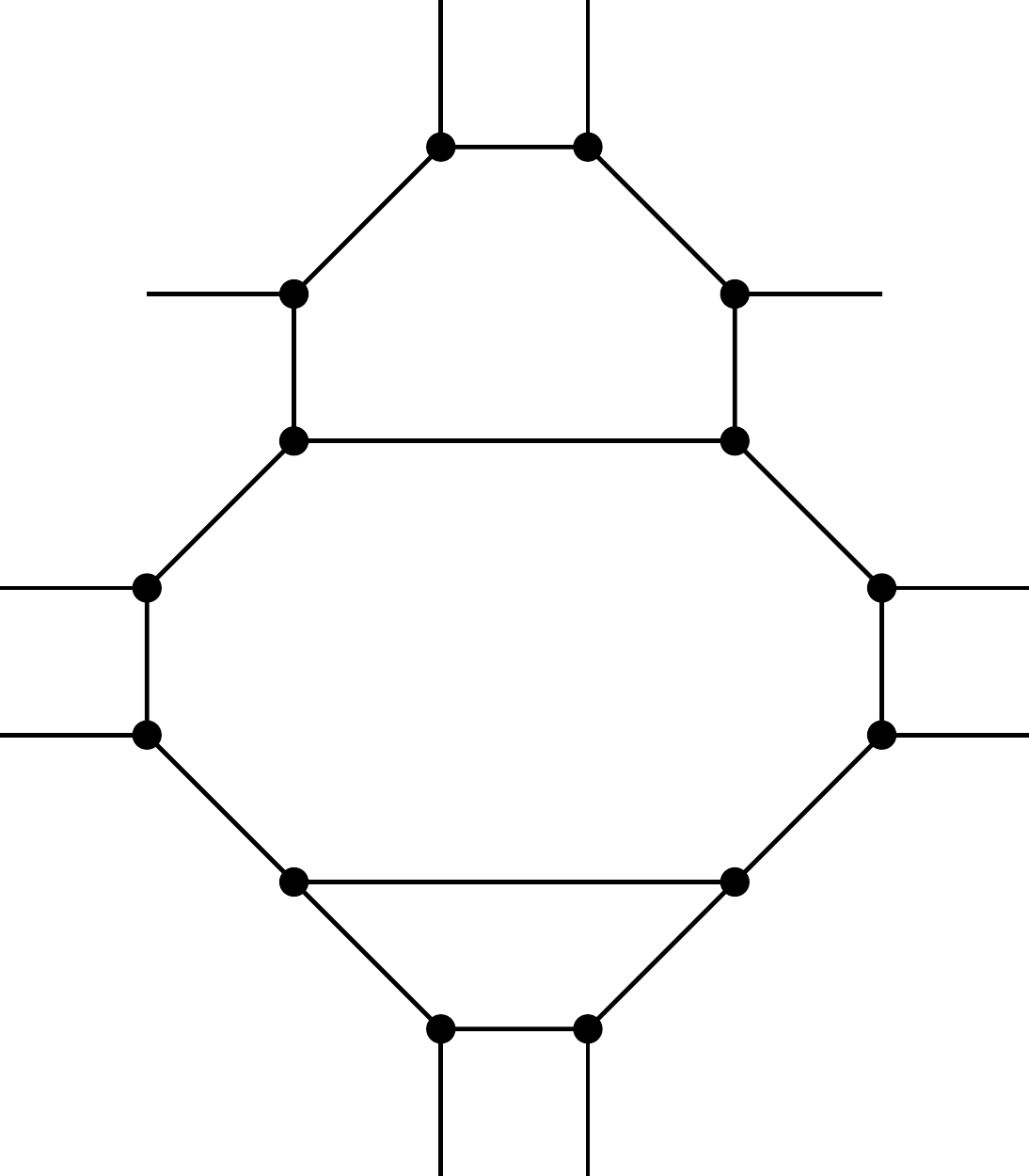}}
}
\captionbox{Graph substitution\label{f:14cb}}[0.49\textwidth]{
\subcaptionbox{Neighbourhood\label{f:o14cb}}[0.24\textwidth]{\includegraphics[width=0.22\textwidth]{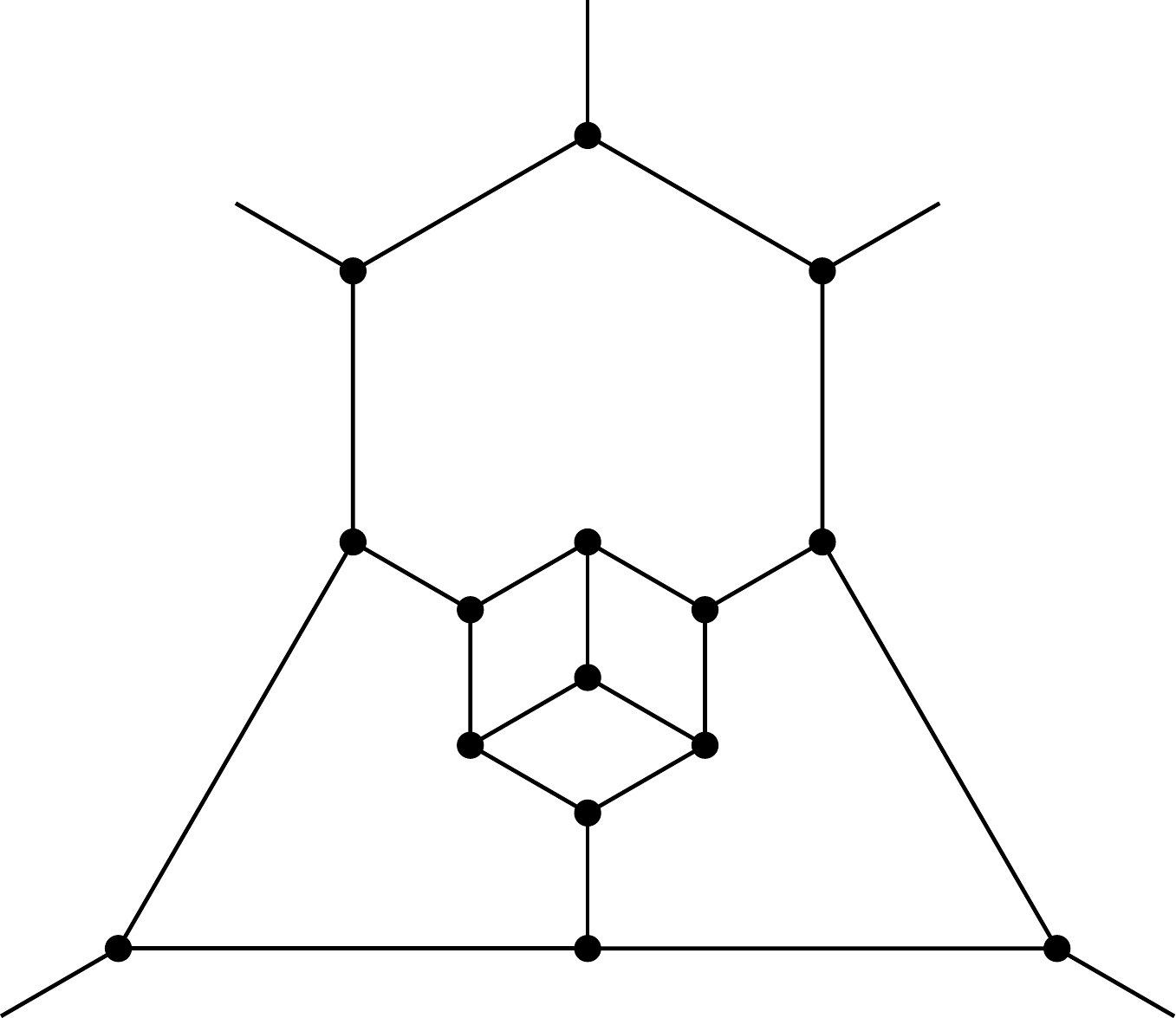}}
\subcaptionbox{Substitution\label{f:14cbsub}}[0.24\textwidth]{\includegraphics[width=0.22\textwidth]{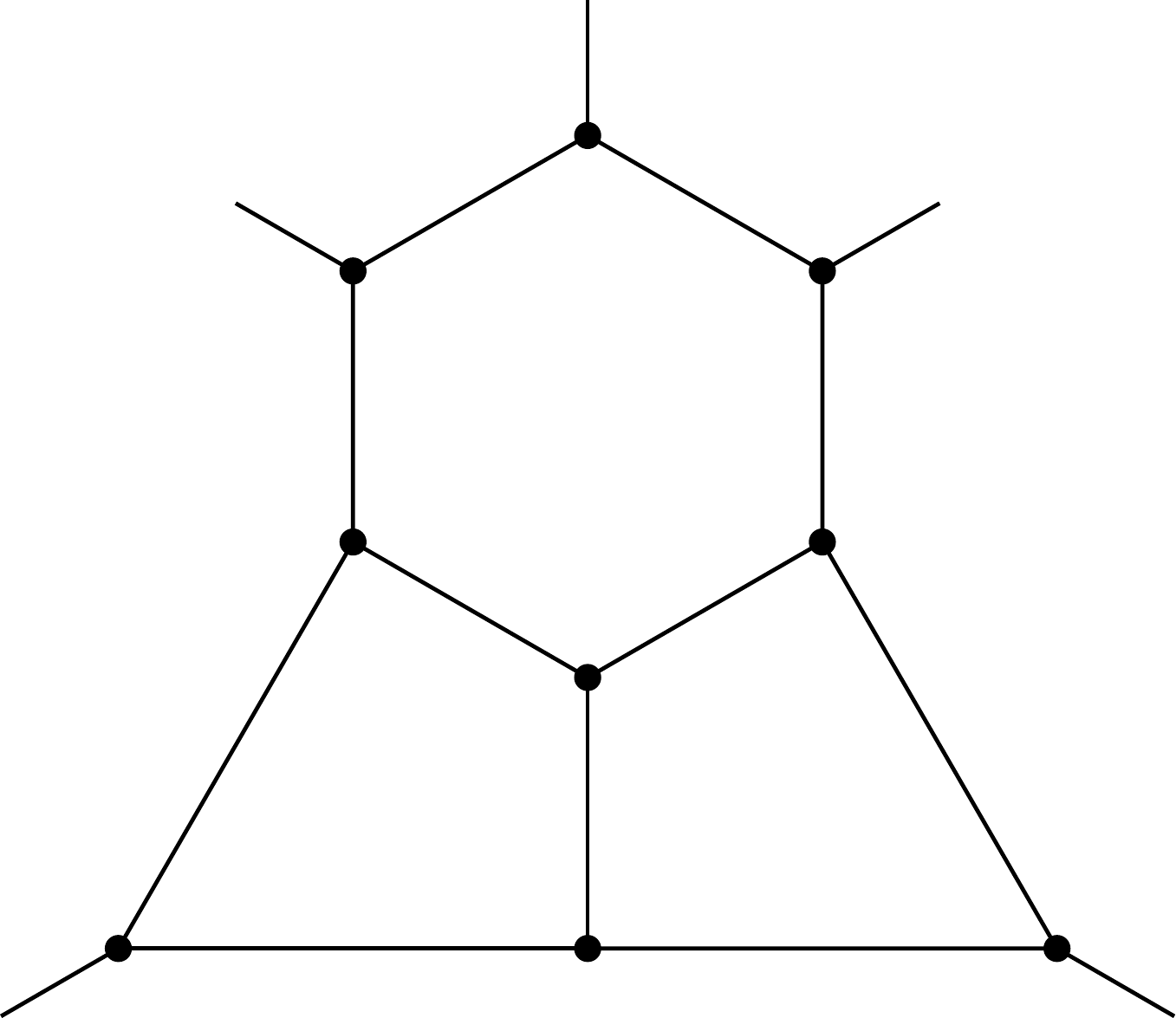}}
}
\end{figure}
\begin{figure}[hp]
\centering
\captionbox{Graph substitution\label{f:13cb}}[0.49\textwidth]{
\subcaptionbox{Neighbourhood\label{f:o13cb}}[0.24\textwidth]{\includegraphics[width=0.22\textwidth]{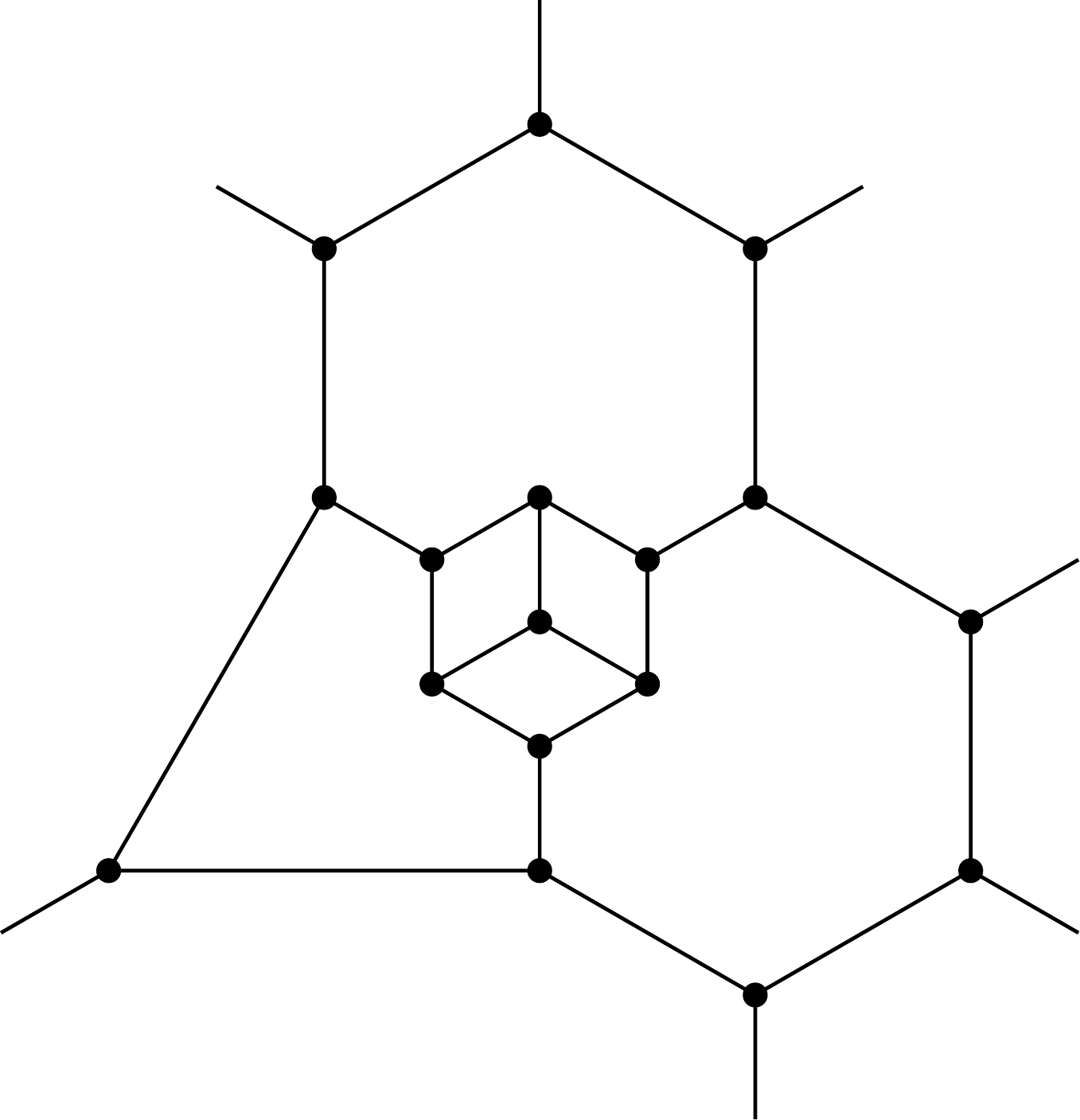}}
\subcaptionbox{Substitution\label{f:13cbsub}}[0.24\textwidth]{\includegraphics[width=0.22\textwidth]{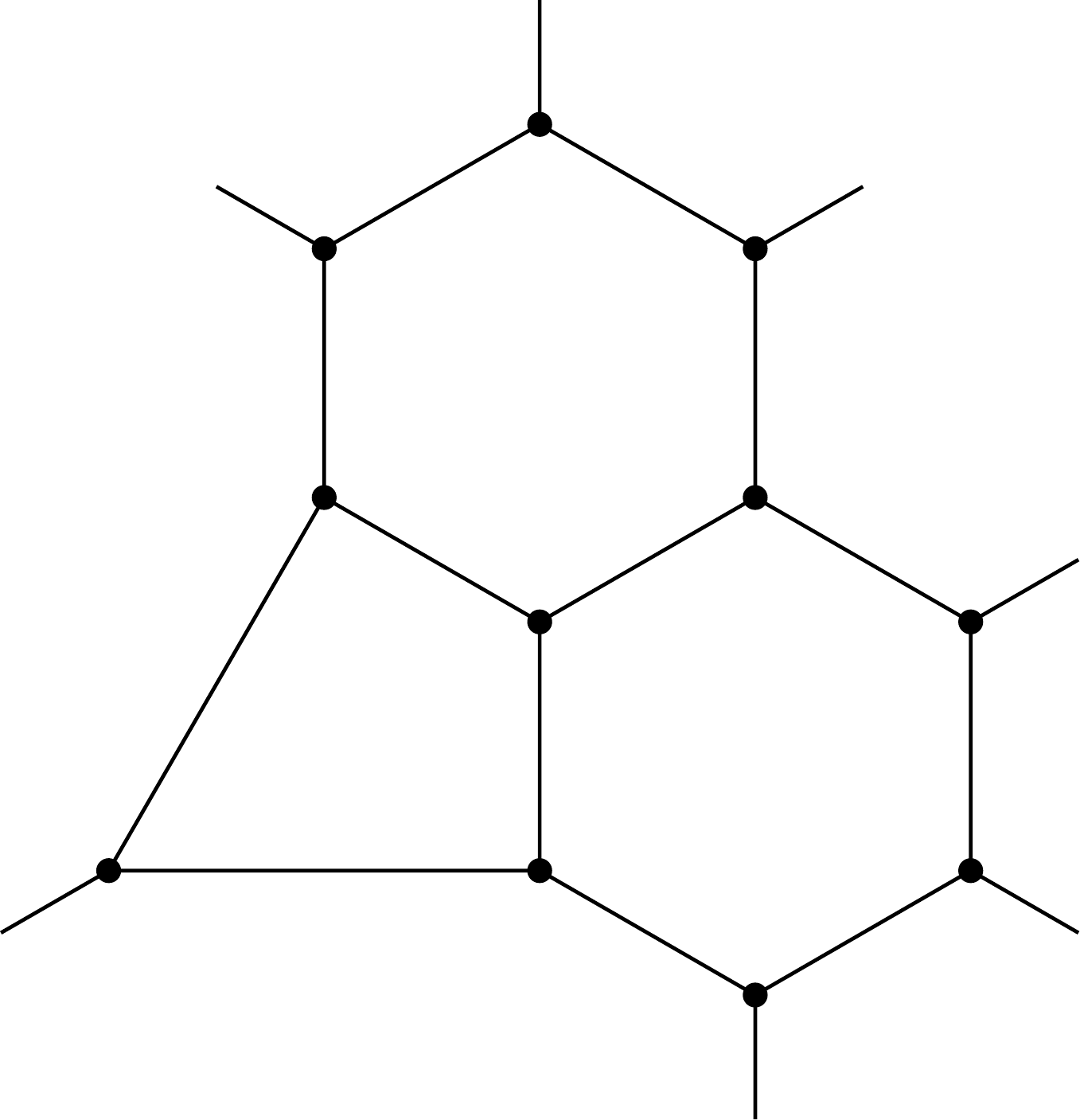}}
}
\captionbox{Graph substitution\label{f:14b}}[0.49\textwidth]{
\subcaptionbox{Neighbourhood\label{f:o14b}}[0.24\textwidth]{\includegraphics[width=0.22\textwidth]{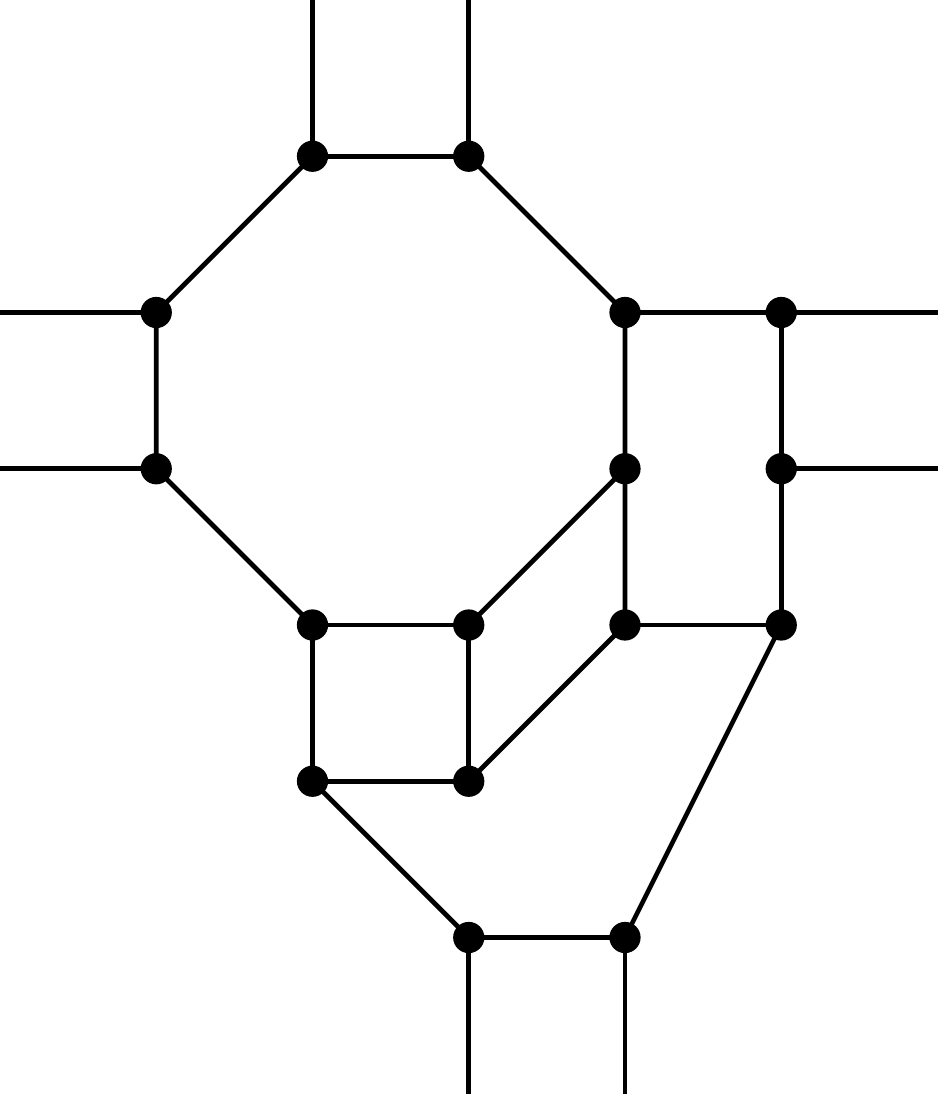}}
\subcaptionbox{Substitution\label{f:14bsub}}[0.24\textwidth]{\includegraphics[width=0.22\textwidth]{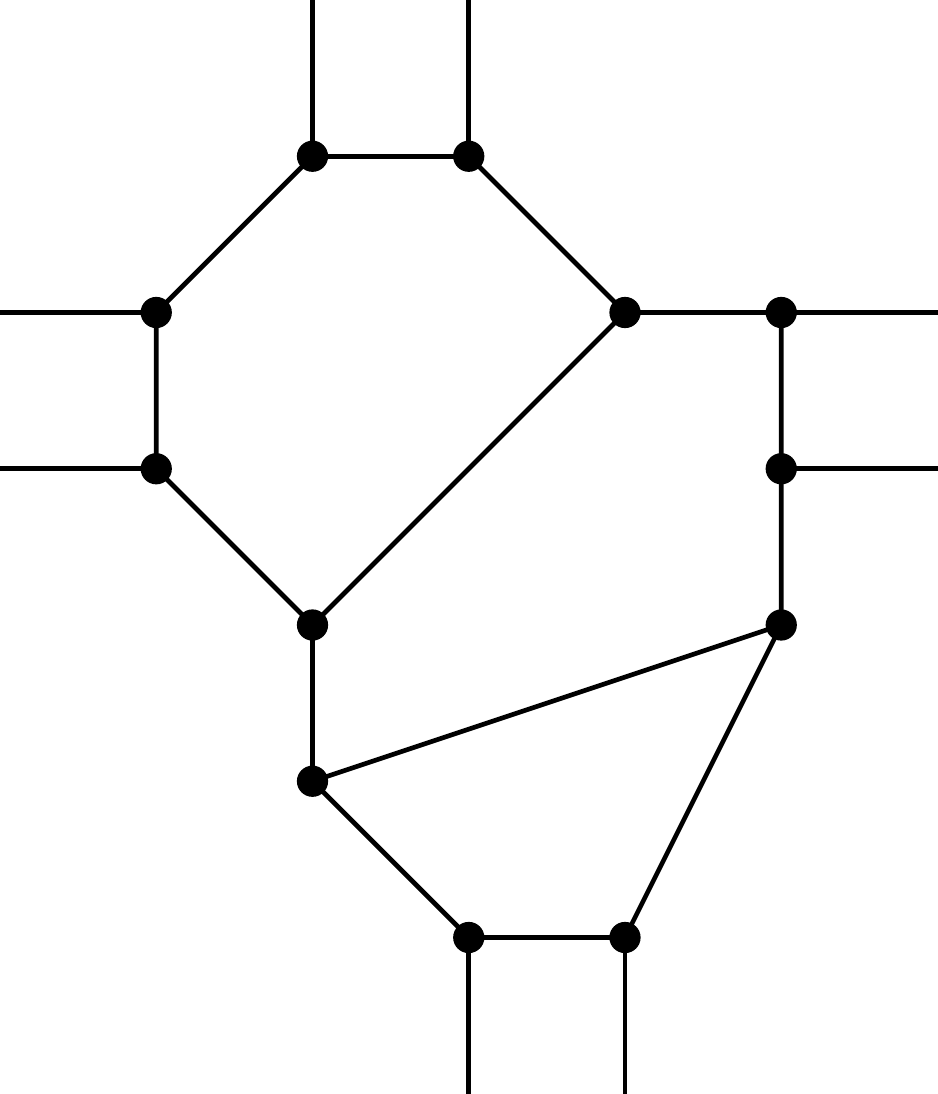}}
}
\end{figure}
\begin{figure}[hp]
\centering
\captionbox{Graph substitution\label{f:1b}}[0.49\textwidth]{
\subcaptionbox{Neighbourhood\label{f:o1b}}[0.24\textwidth]{\includegraphics[width=0.22\textwidth]{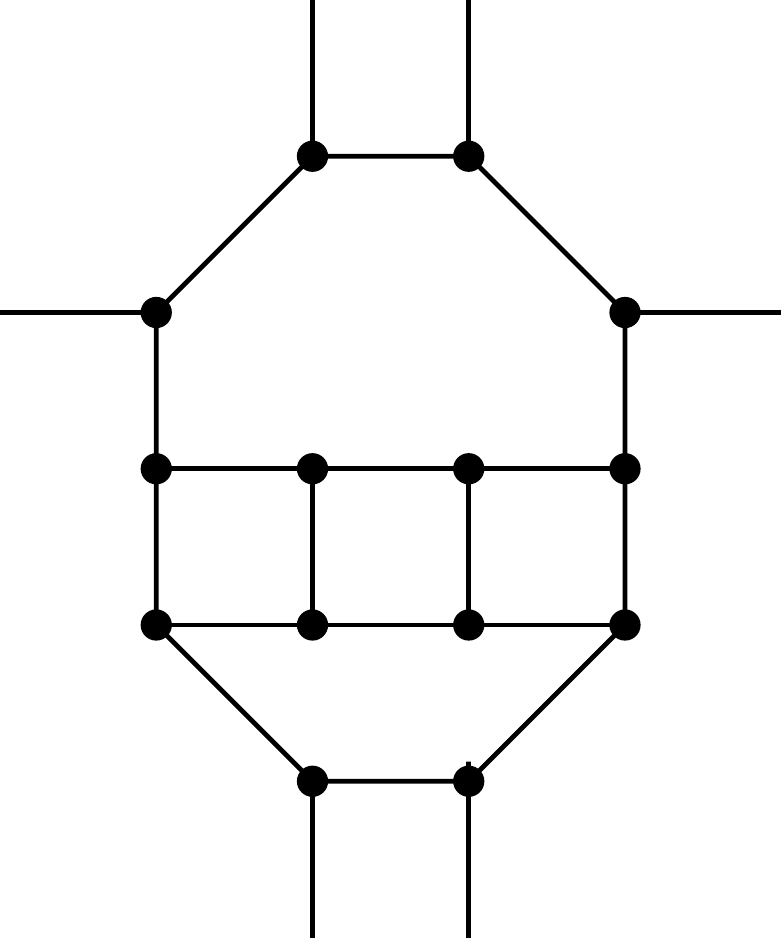}}
\subcaptionbox{Substitution\label{f:1bsub}}[0.24\textwidth]{\includegraphics[width=0.22\textwidth]{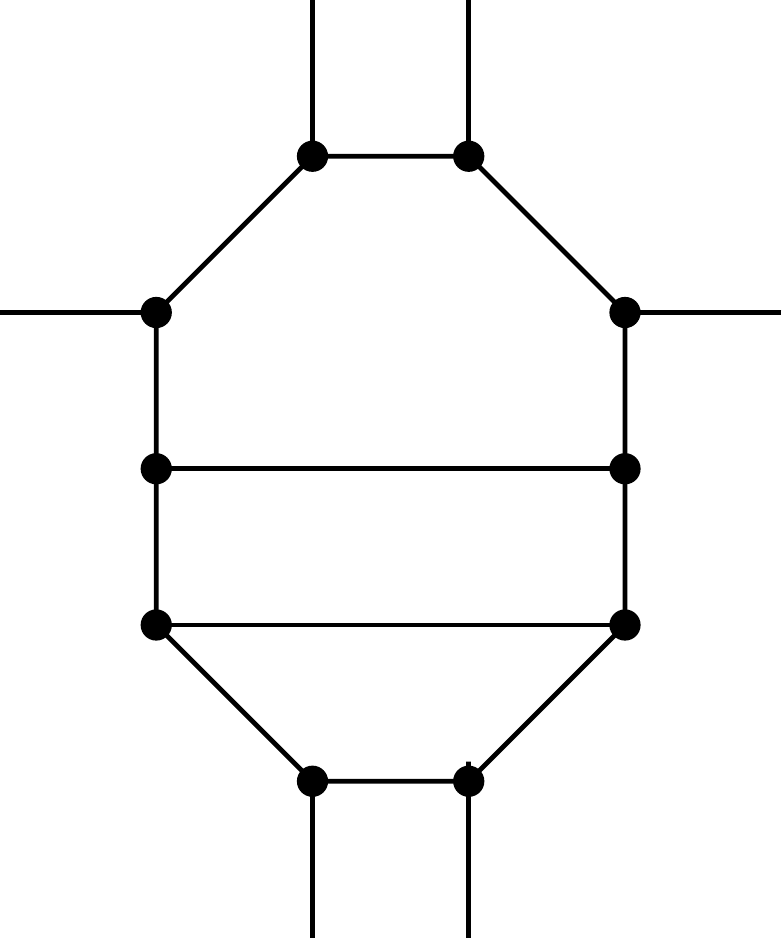}}
}
\captionbox{Graph substitution\label{f:1cbc}}[0.49\textwidth]{
\subcaptionbox{Neighbourhood\label{f:o1cbc}}[0.24\textwidth]{\includegraphics[width=0.22\textwidth]{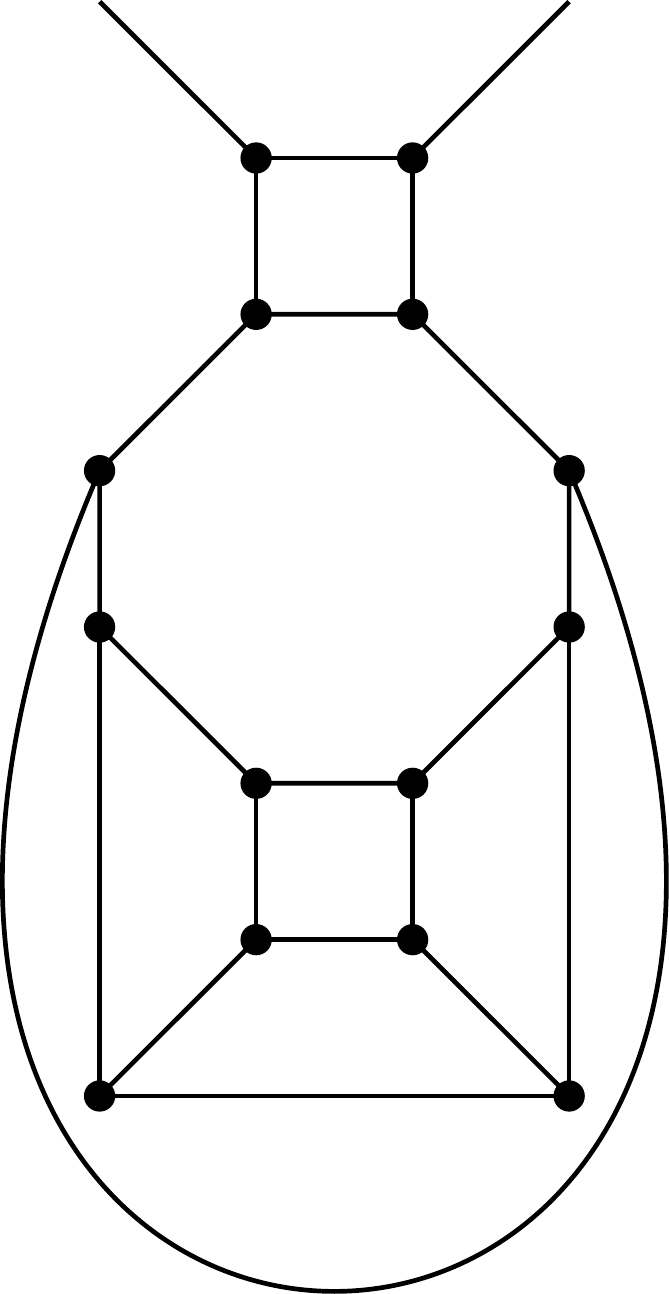}}
\subcaptionbox{Substitution\label{f:1cbcsub}}[0.24\textwidth]{\includegraphics[width=0.22\textwidth]{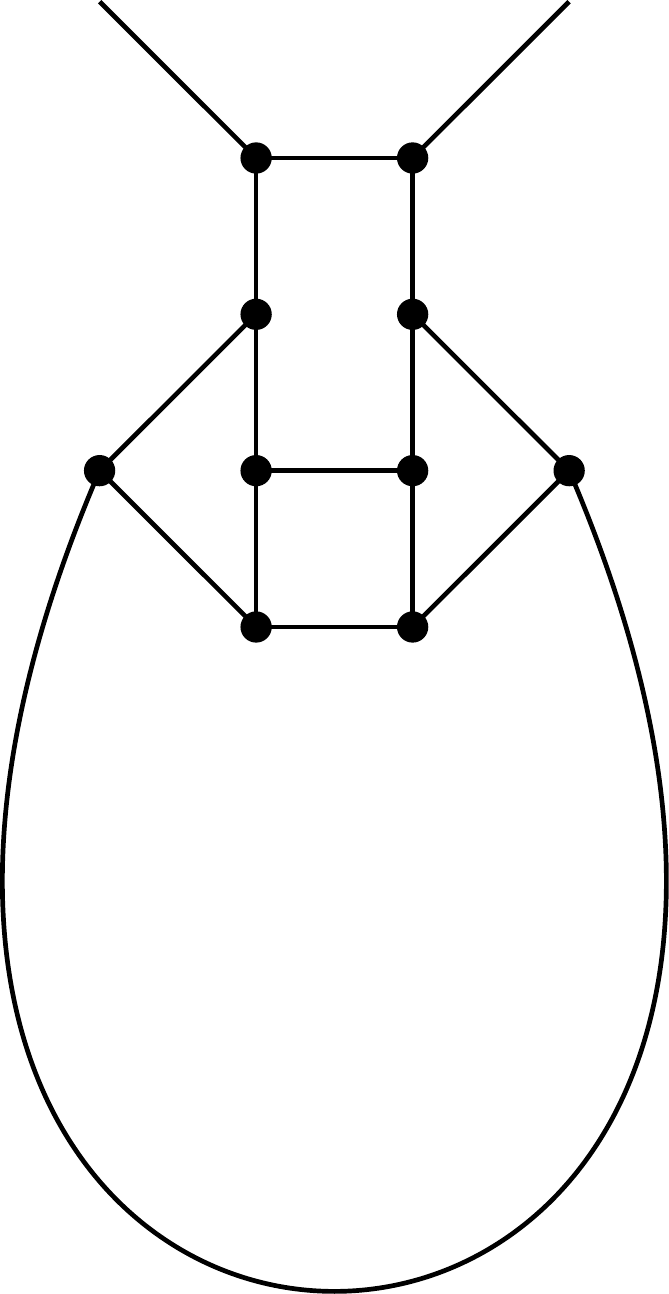}}
}
\end{figure}
\begin{figure}[hp]
\centering
\captionbox{Graph substitution\label{f:1cbbb}}[0.49\textwidth]{
\subcaptionbox{Neighbourhood\label{f:o1cbbb}}[0.24\textwidth]{\includegraphics[width=0.22\textwidth]{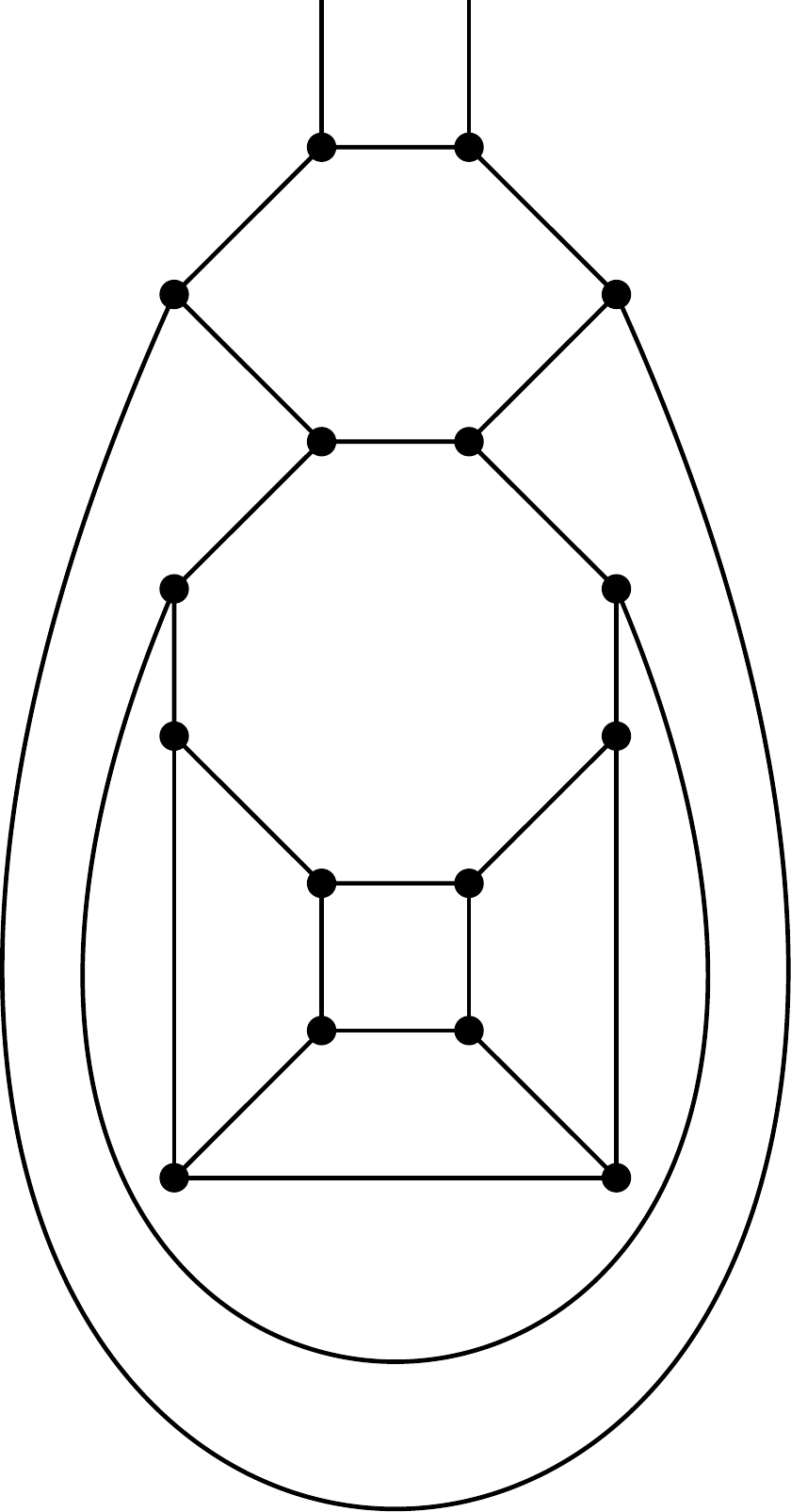}}
\subcaptionbox{Substitution\label{f:1cbbbsub}}[0.24\textwidth]{\includegraphics[width=0.22\textwidth]{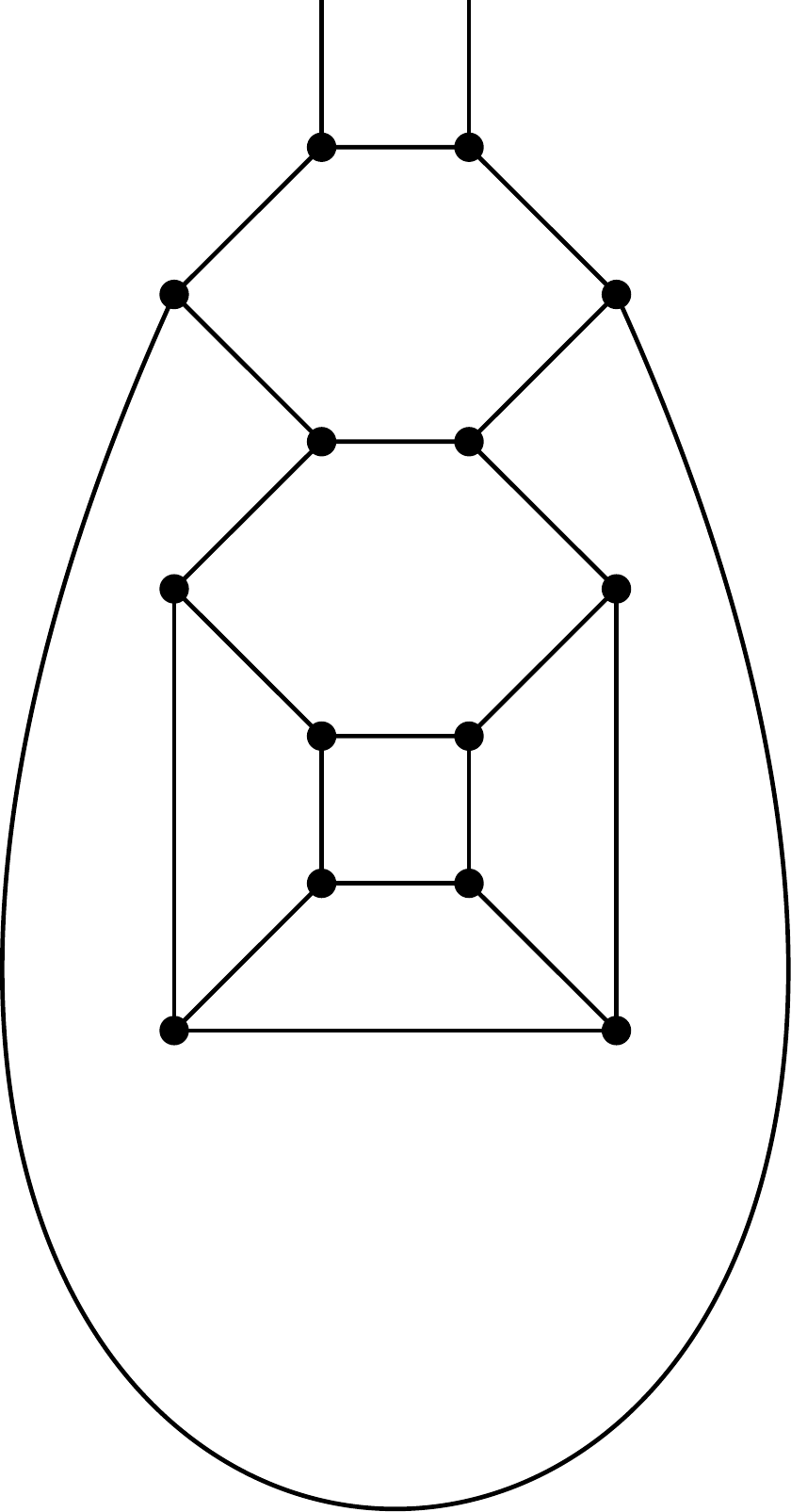}}
}
\captionbox{Graph substitution\label{f:1cbba}}[0.49\textwidth]{
\subcaptionbox{Neighbourhood\label{f:o1cbba}}[0.24\textwidth]{\includegraphics[width=0.22\textwidth]{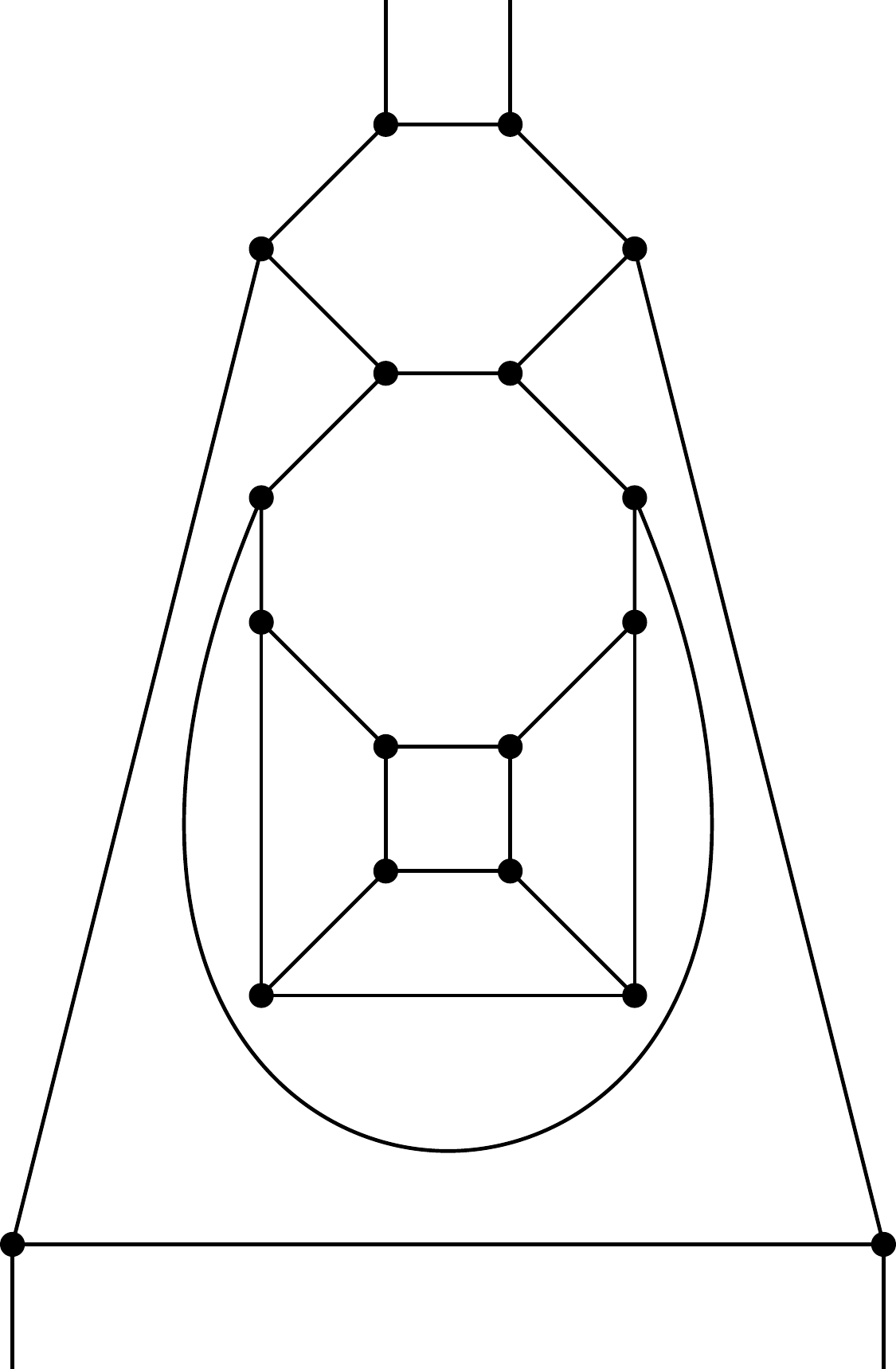}}
\subcaptionbox{Substitution\label{f:1cbbasub}}[0.24\textwidth]{\includegraphics[width=0.22\textwidth]{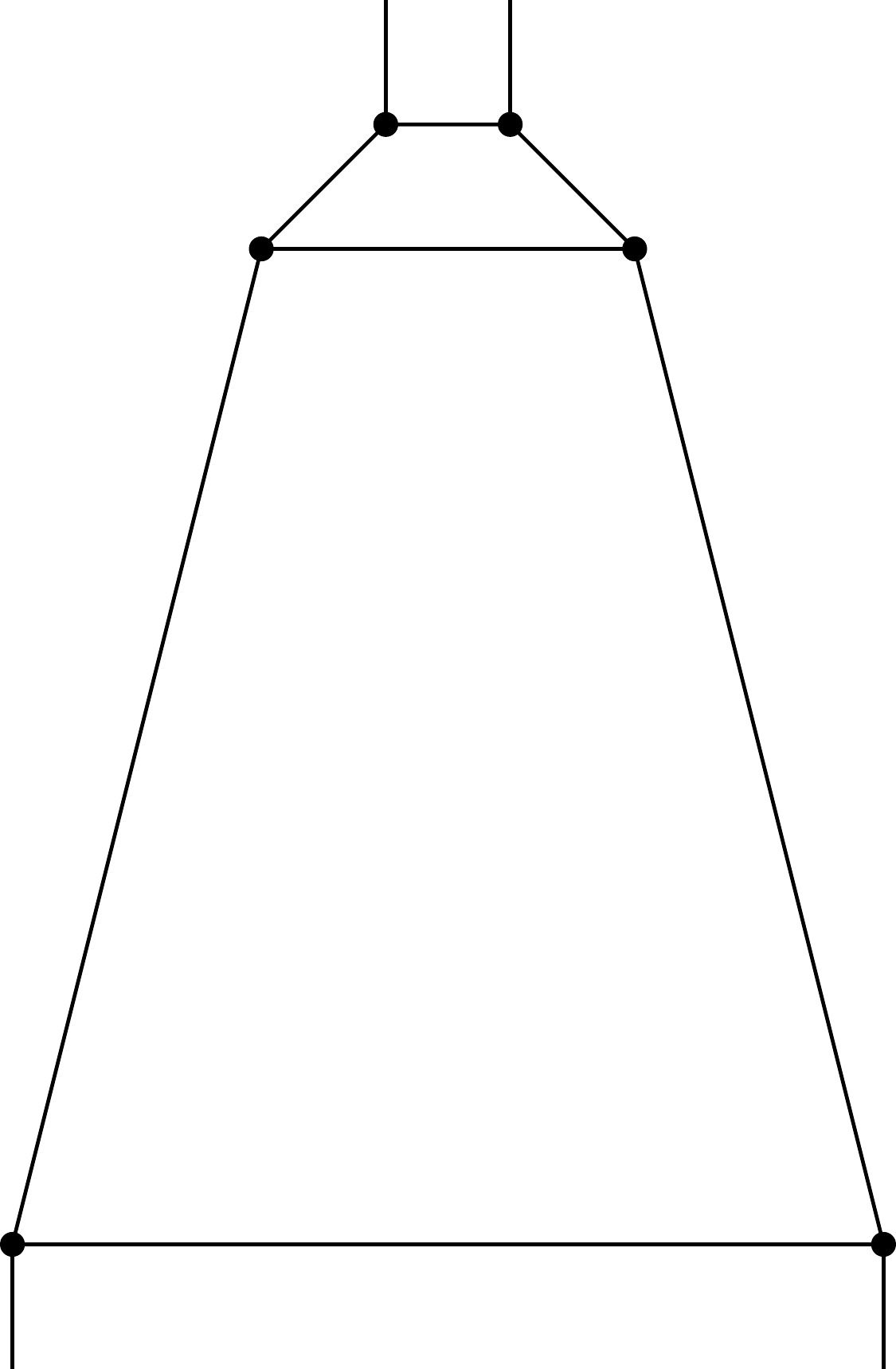}}
}
\end{figure}
\begin{figure}[hp]
\centering
\captionbox{Graph substitution\label{f:1cbab}}[0.49\textwidth]{
\subcaptionbox{Neighbourhood\label{f:o1cbab}}[0.24\textwidth]{\includegraphics[width=0.22\textwidth]{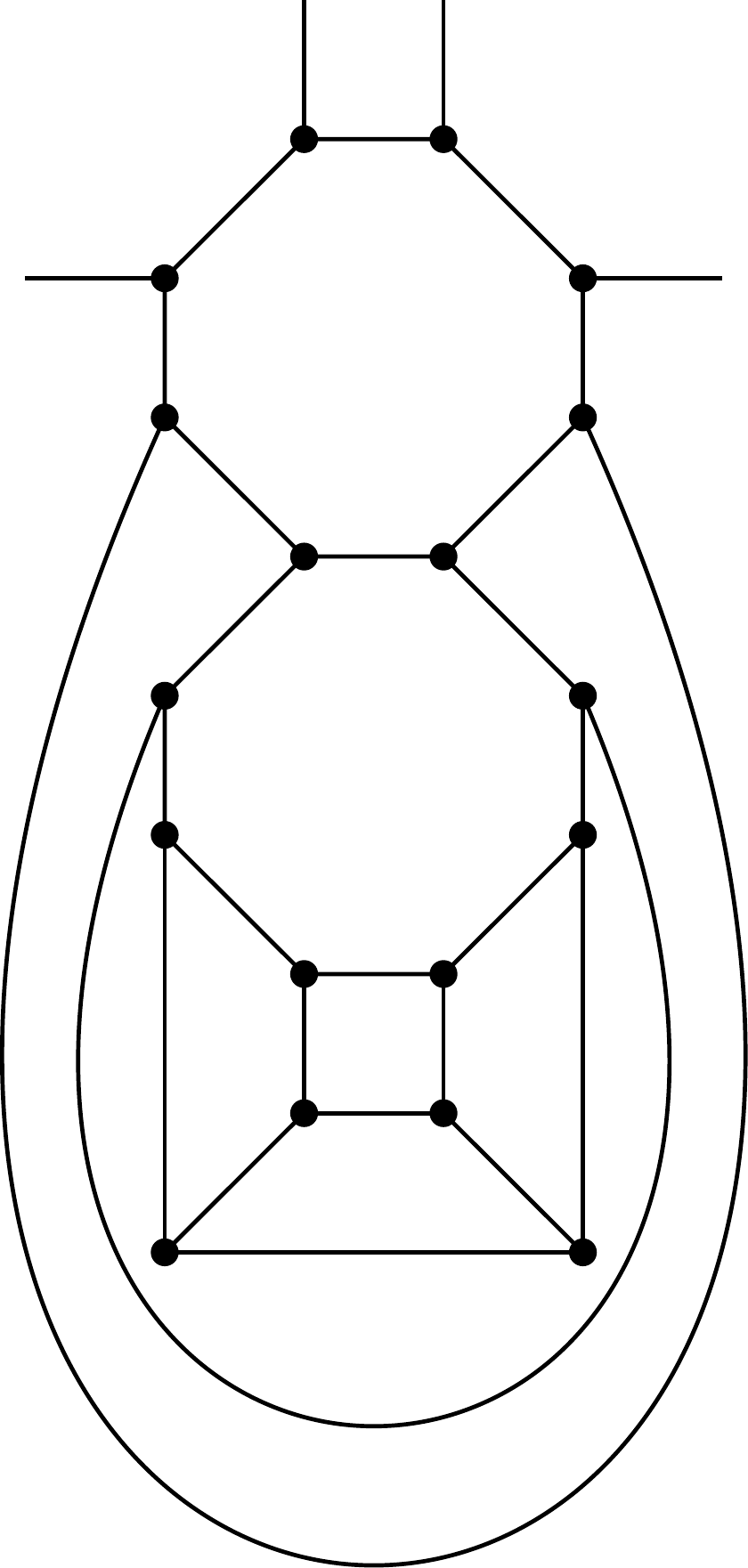}}
\subcaptionbox{Substitution\label{f:1cbabsub}}[0.24\textwidth]{\includegraphics[width=0.22\textwidth]{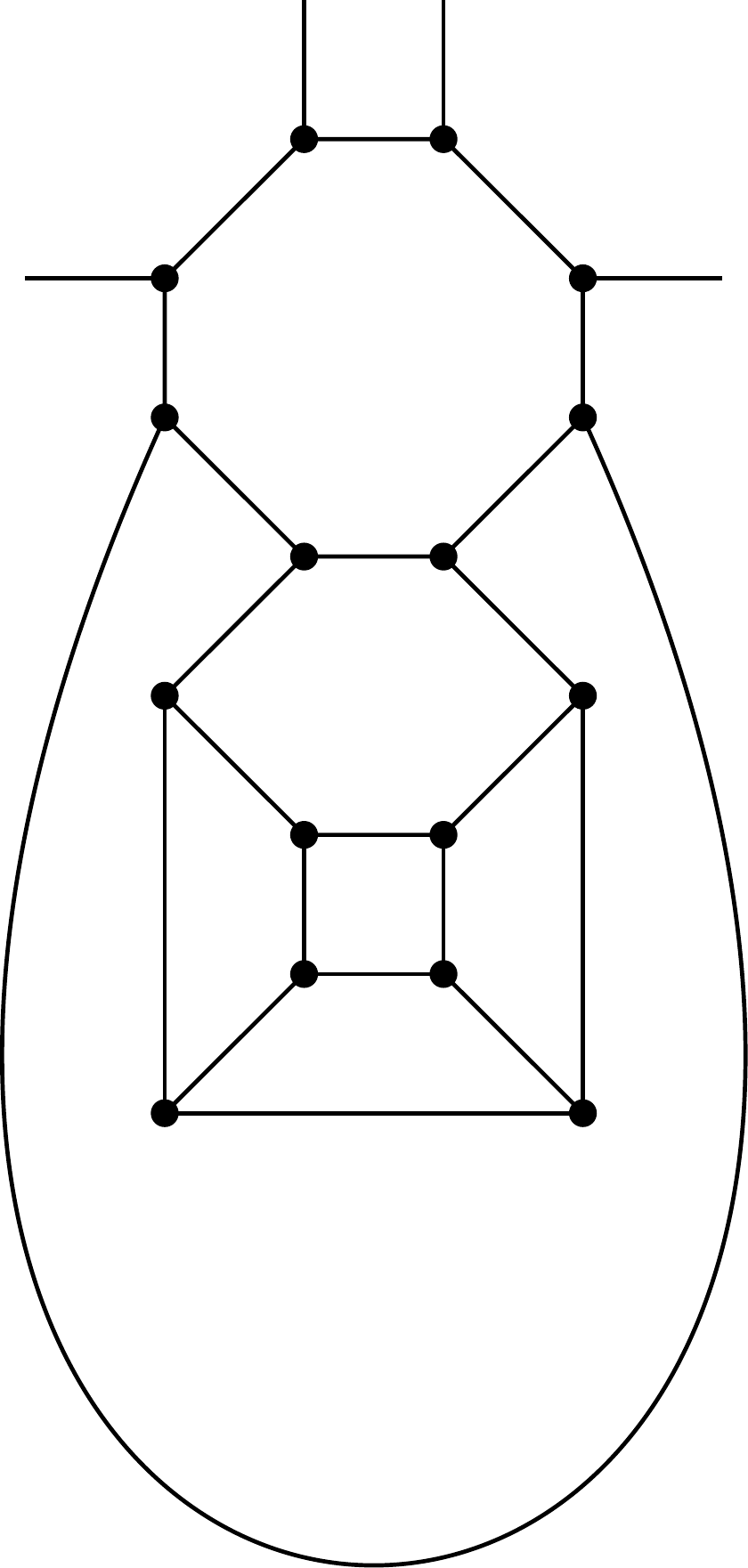}}
}
\captionbox{Graph substitution\label{f:1cbaa}}[0.49\textwidth]{
\subcaptionbox{Neighbourhood\label{f:o1cbaa}}[0.24\textwidth]{\includegraphics[width=0.22\textwidth]{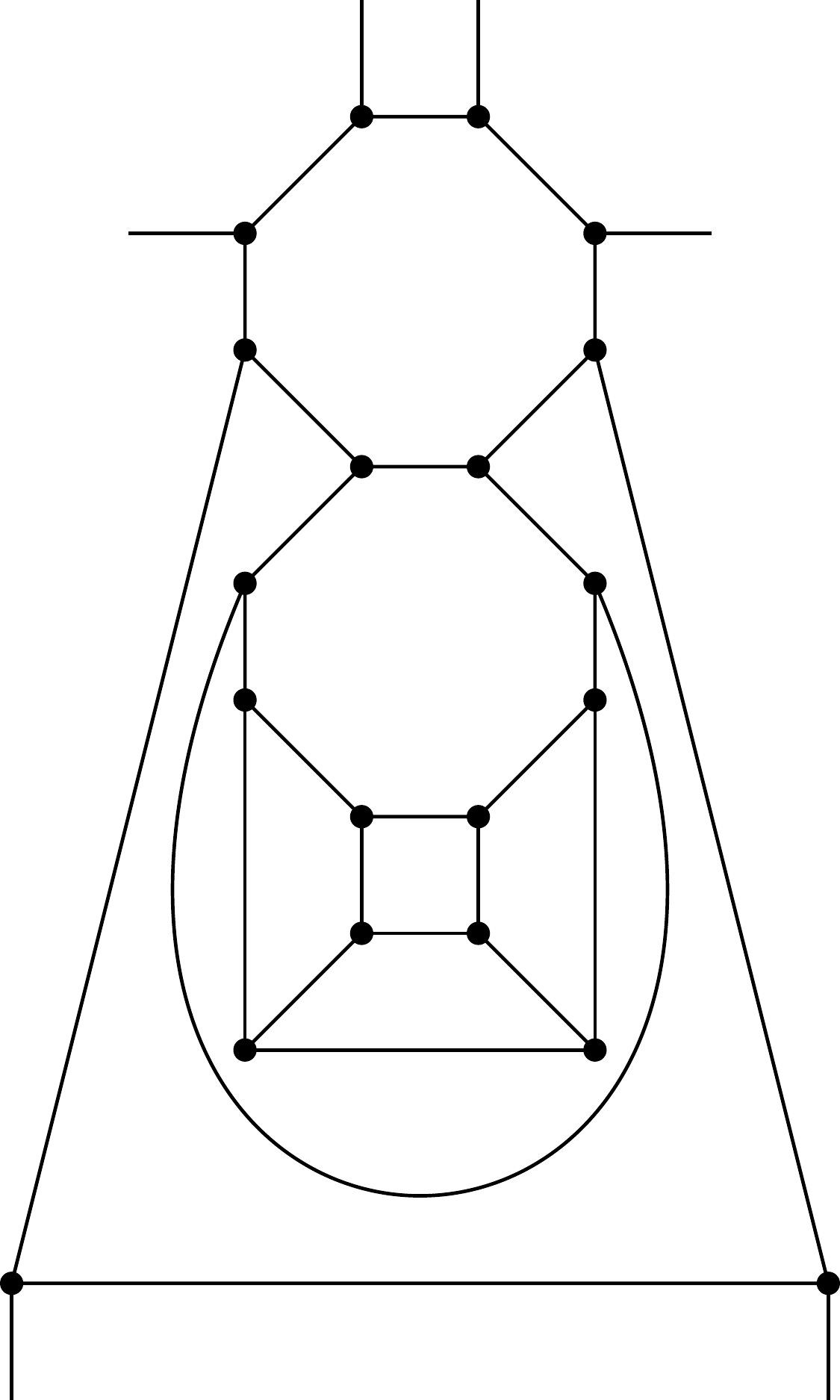}}
\subcaptionbox{Substitution\label{f:1cbaasub}}[0.24\textwidth]{\includegraphics[width=0.22\textwidth]{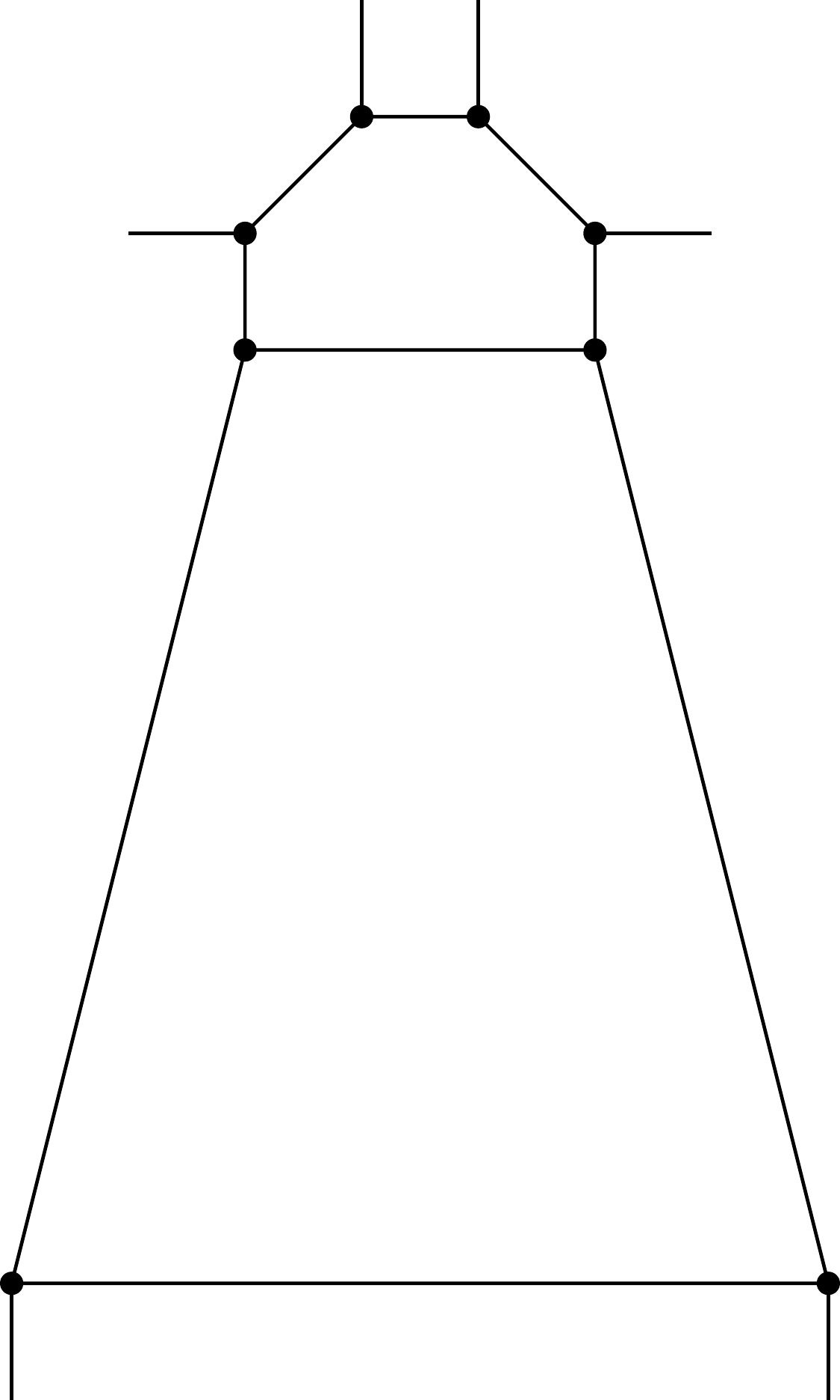}}
}
\end{figure}
\begin{figure}[hp]
\centering
\captionbox{Graph substitution\label{f:1ca}}[0.49\textwidth]{
\subcaptionbox{Neighbourhood\label{f:o1ca}}[0.24\textwidth]{\includegraphics[width=0.22\textwidth]{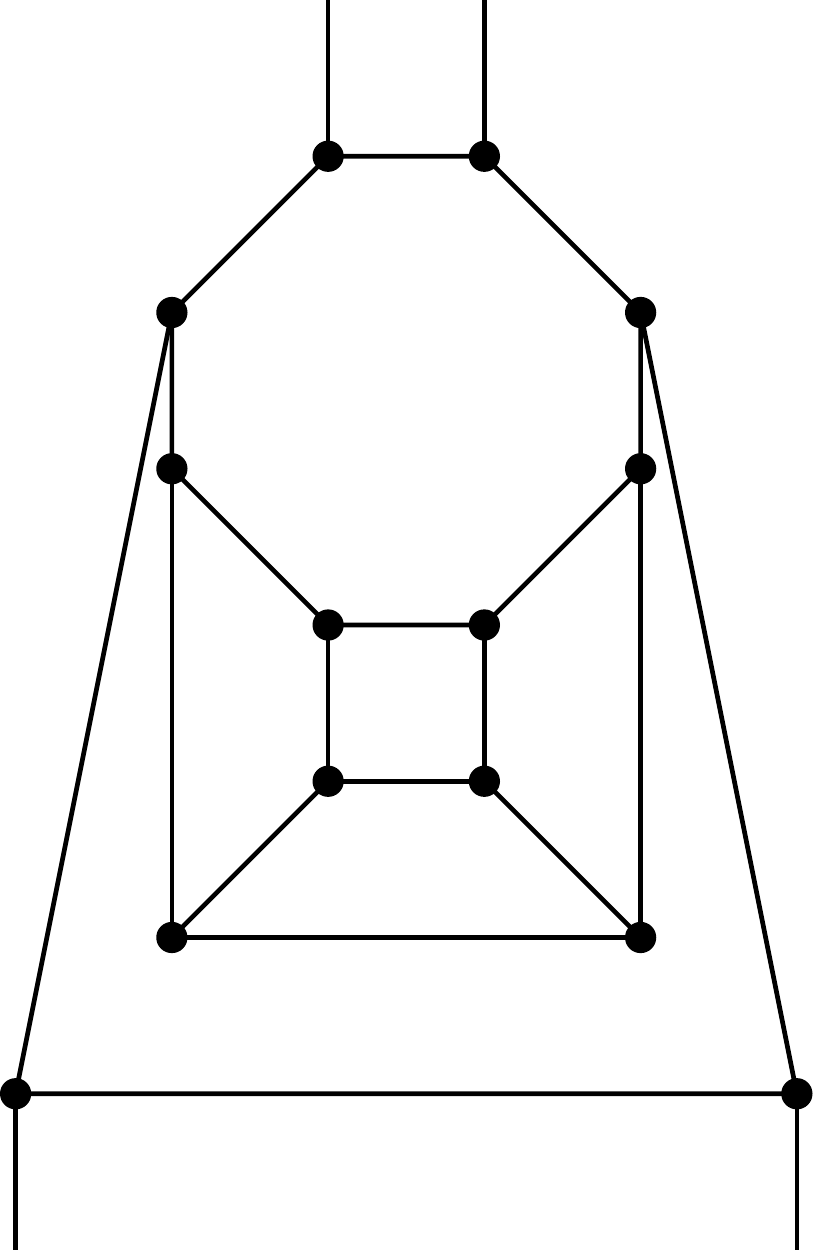}}
\subcaptionbox{Substitution\label{f:1casub}}[0.24\textwidth]{\includegraphics[width=0.22\textwidth]{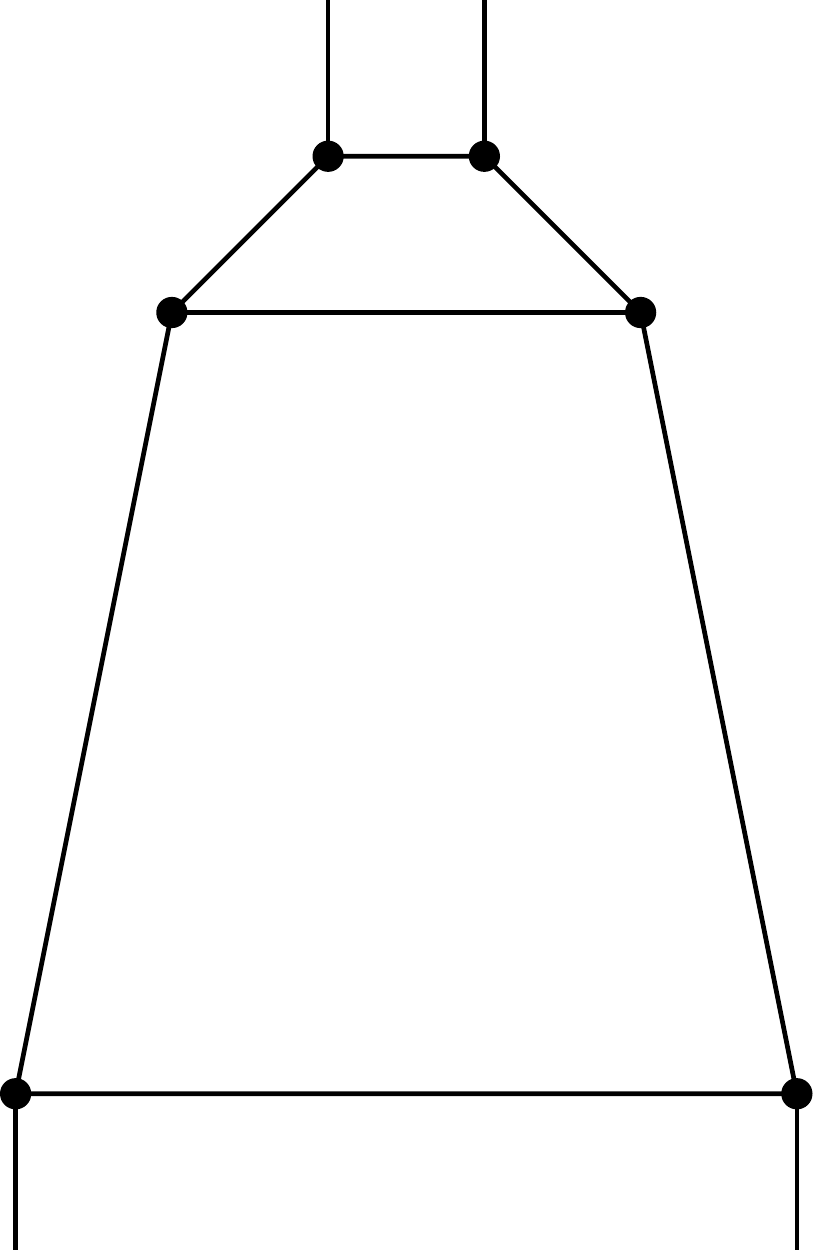}}
}
\captionbox{Graph substitution\label{f:subtestexample}}[0.49\textwidth]{
\subcaptionbox{Neighbourhood\label{f:osubtestexample}}[0.24\textwidth]{\includegraphics[width=0.22\textwidth]{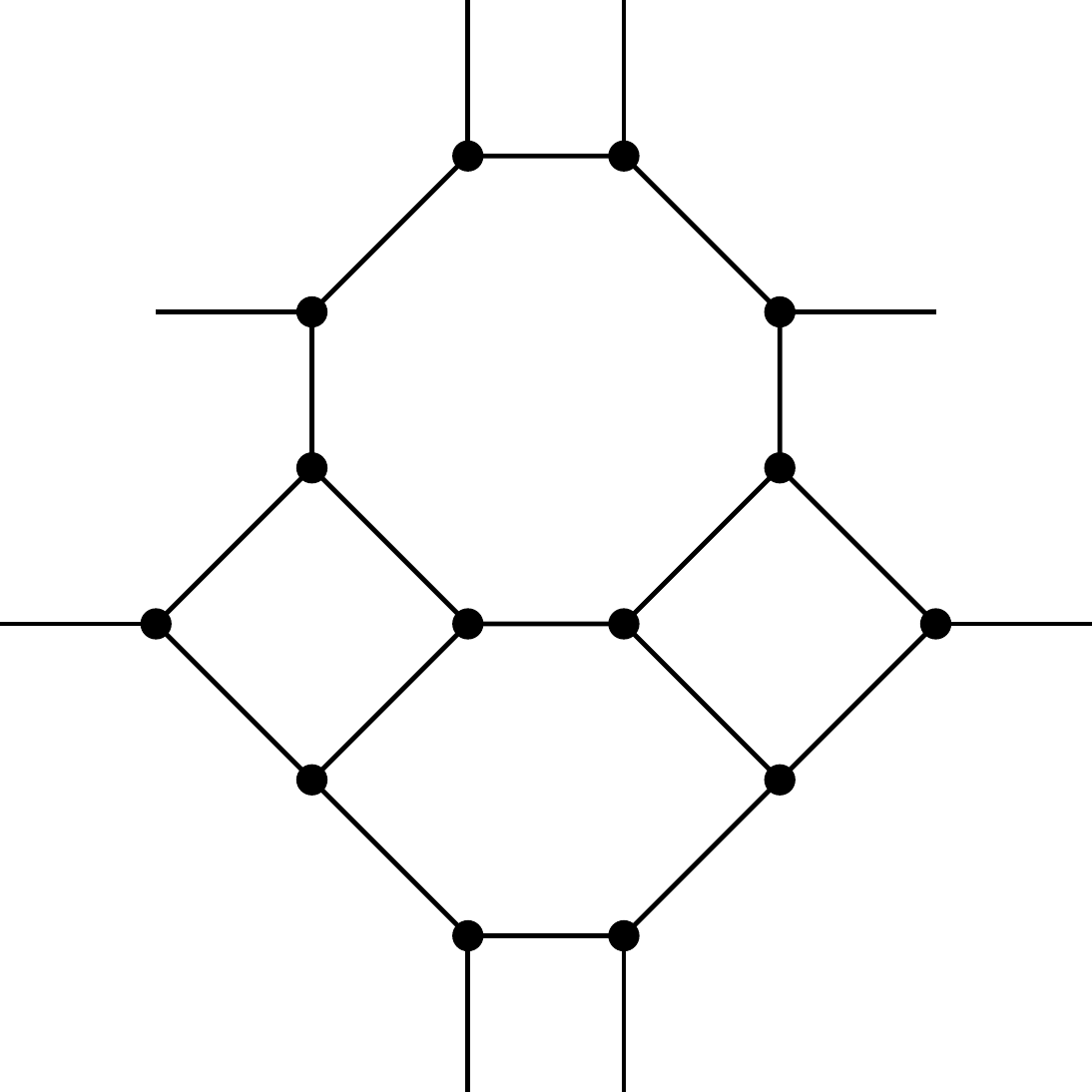}}
\subcaptionbox{Substitution\label{f:subtestexamplesub}}[0.24\textwidth]{\includegraphics[width=0.22\textwidth]{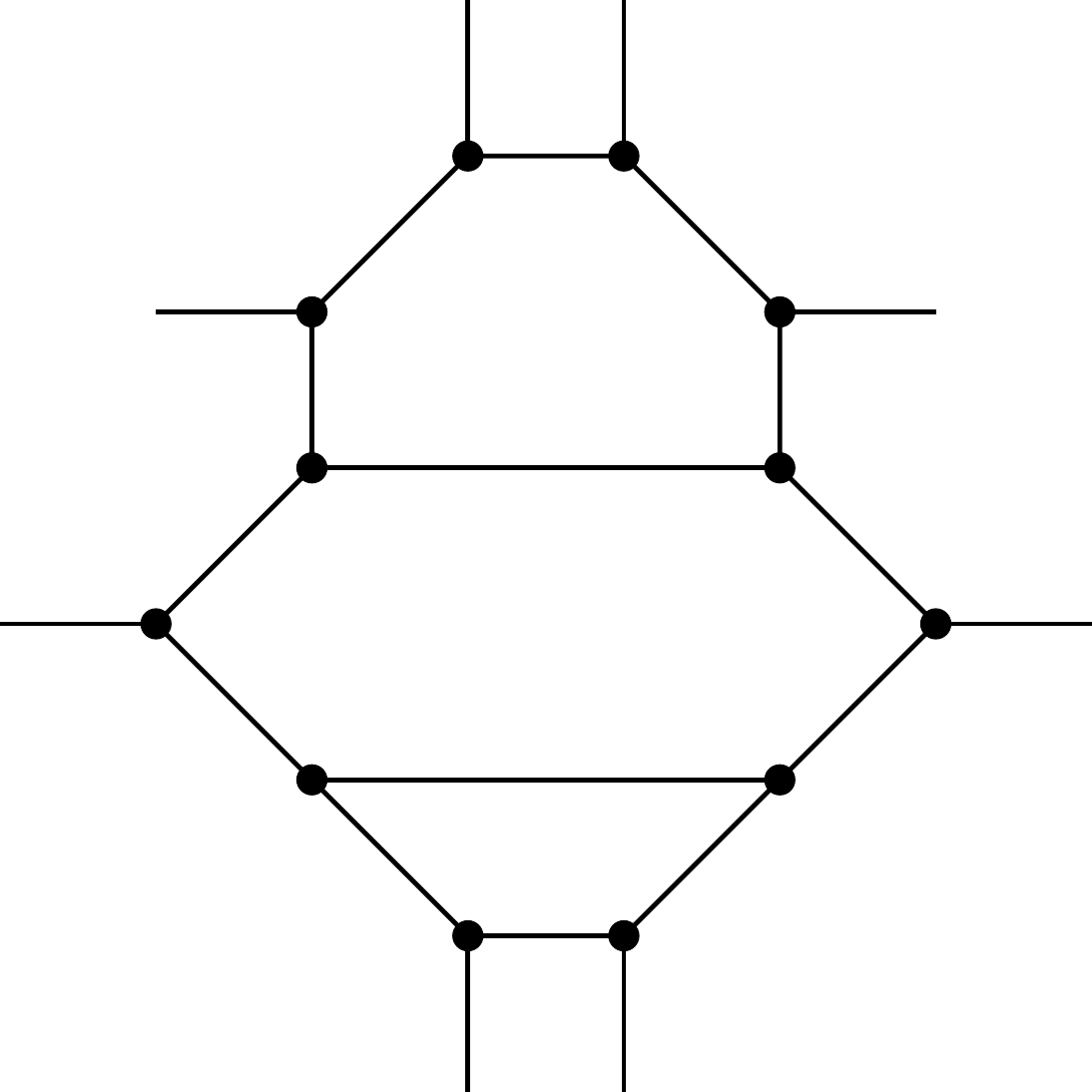}}
}
\end{figure}
\begin{figure}[hp]
\centering
\captionbox{Graph substitution\label{f:2}}[0.49\textwidth]{
\subcaptionbox{Neighbourhood\label{f:o2}}[0.24\textwidth]{\includegraphics[width=0.22\textwidth]{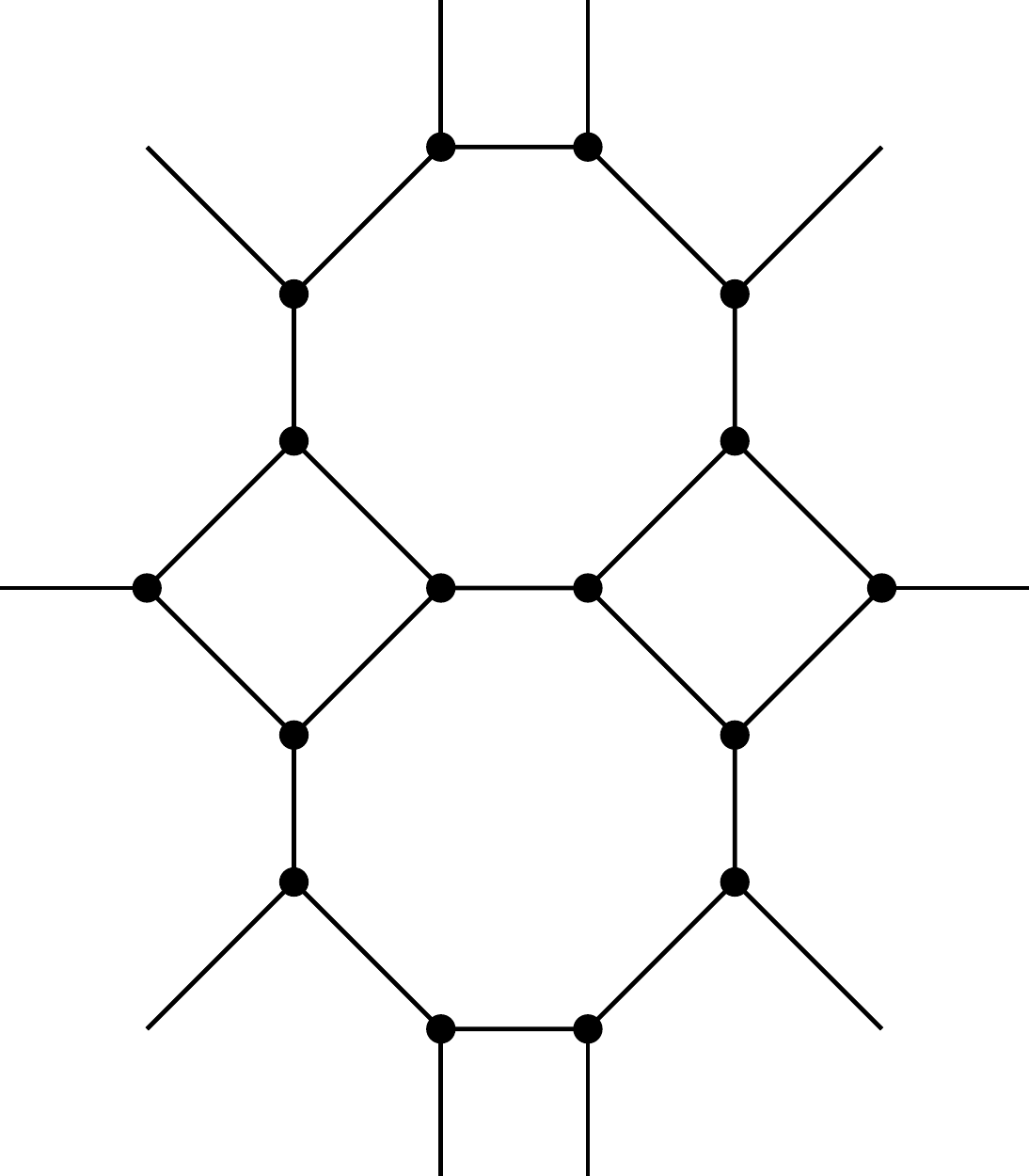}}
\subcaptionbox{Substitution\label{f:2sub}}[0.24\textwidth]{\includegraphics[width=0.22\textwidth]{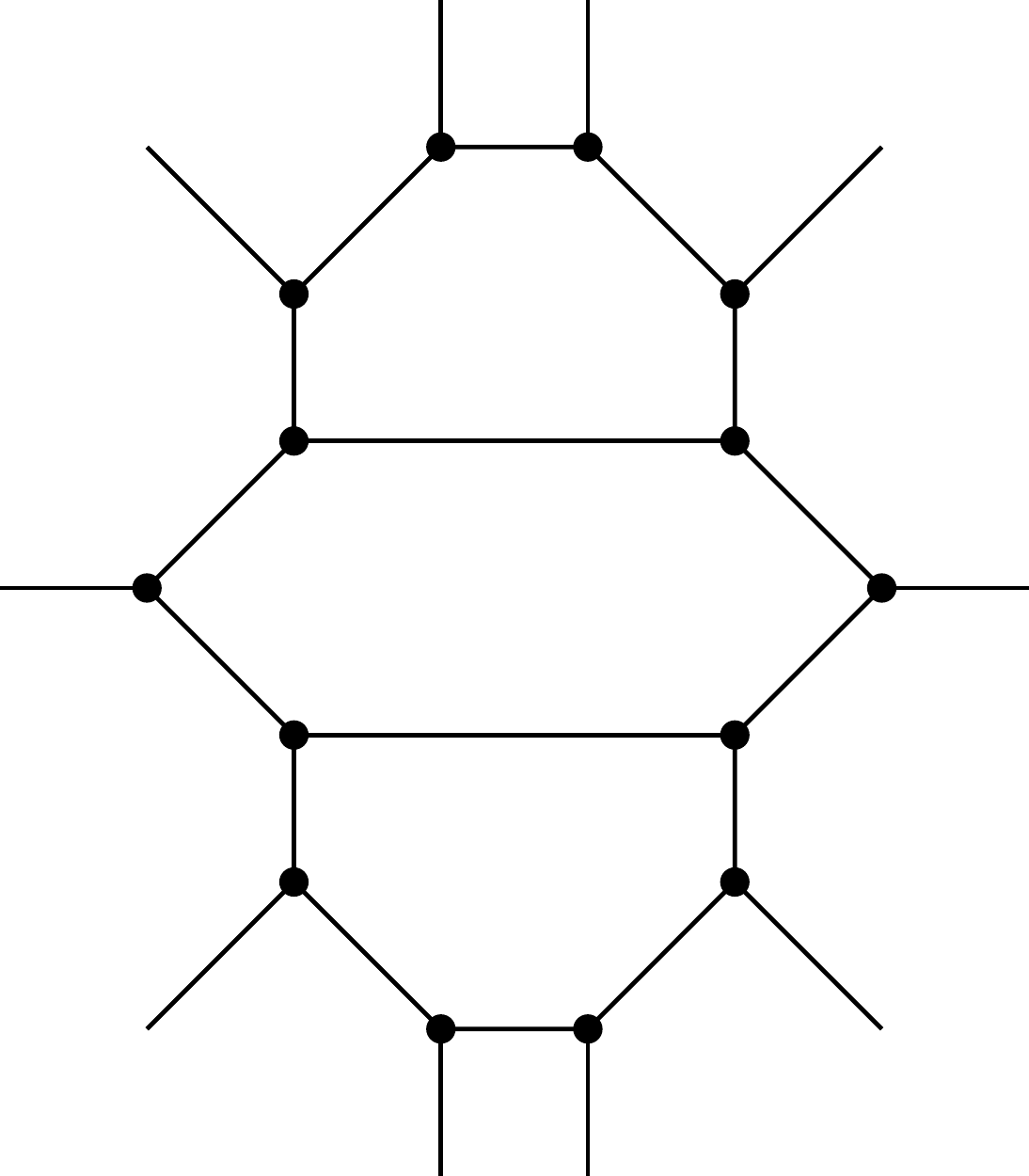}}
}
\captionbox{Graph substitution\label{f:13b}}[0.49\textwidth]{
\subcaptionbox{Neighbourhood\label{f:o13b}}[0.24\textwidth]{\includegraphics[width=0.22\textwidth]{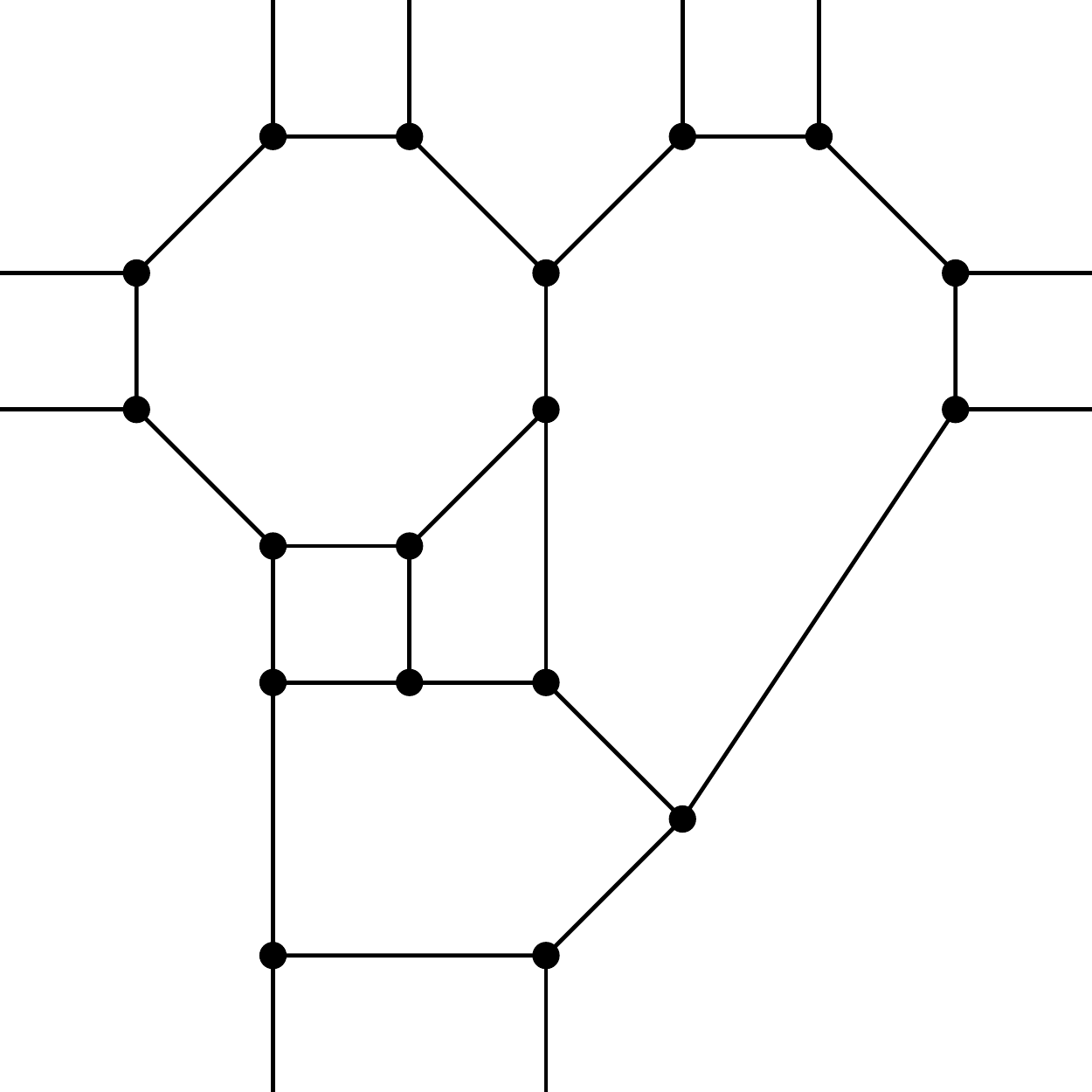}}
\subcaptionbox{Substitution\label{f:13bsub}}[0.24\textwidth]{\includegraphics[width=0.22\textwidth]{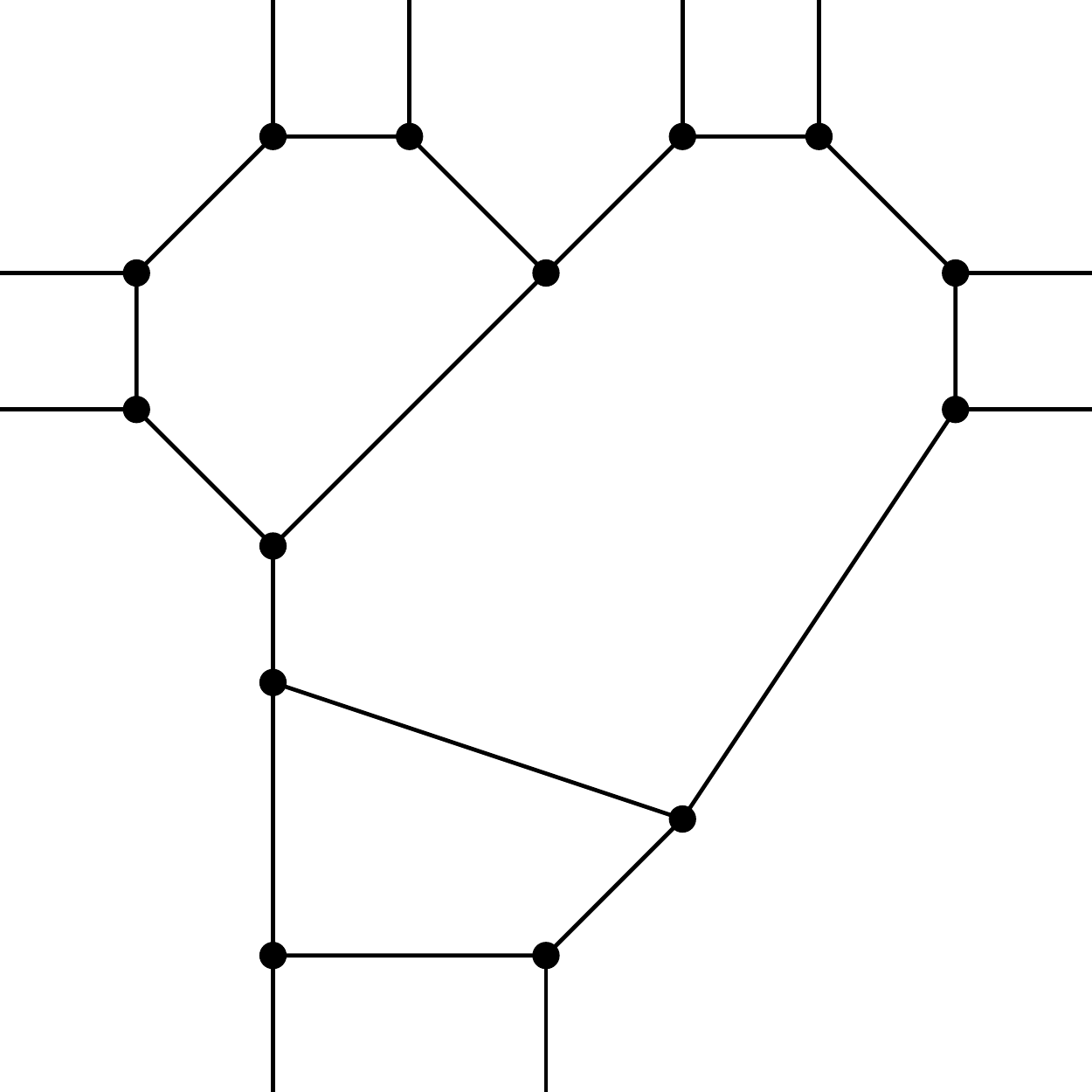}}
}
\end{figure}
\begin{figure}[hp]
\centering
\captionbox{Graph substitution\label{f:1a}}[0.49\textwidth]{
\subcaptionbox{Neighbourhood\label{f:o1a}}[0.24\textwidth]{\includegraphics[width=0.22\textwidth]{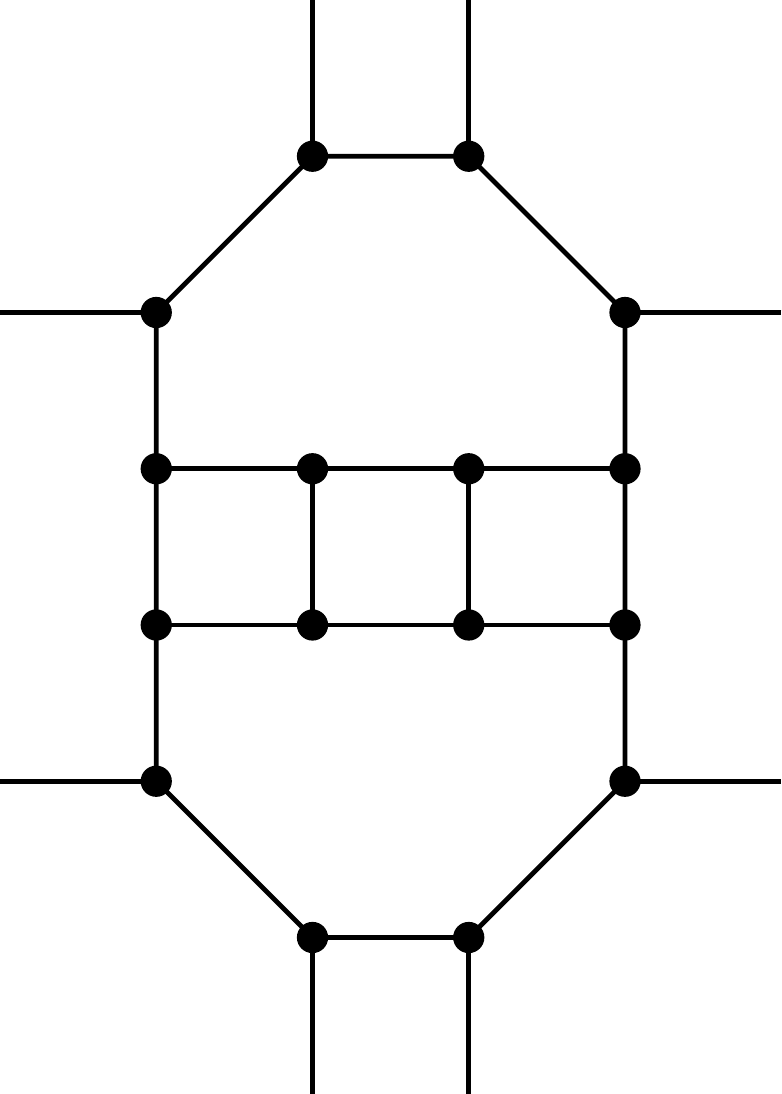}}
\subcaptionbox{Substitution\label{f:1asub}}[0.24\textwidth]{\includegraphics[width=0.22\textwidth]{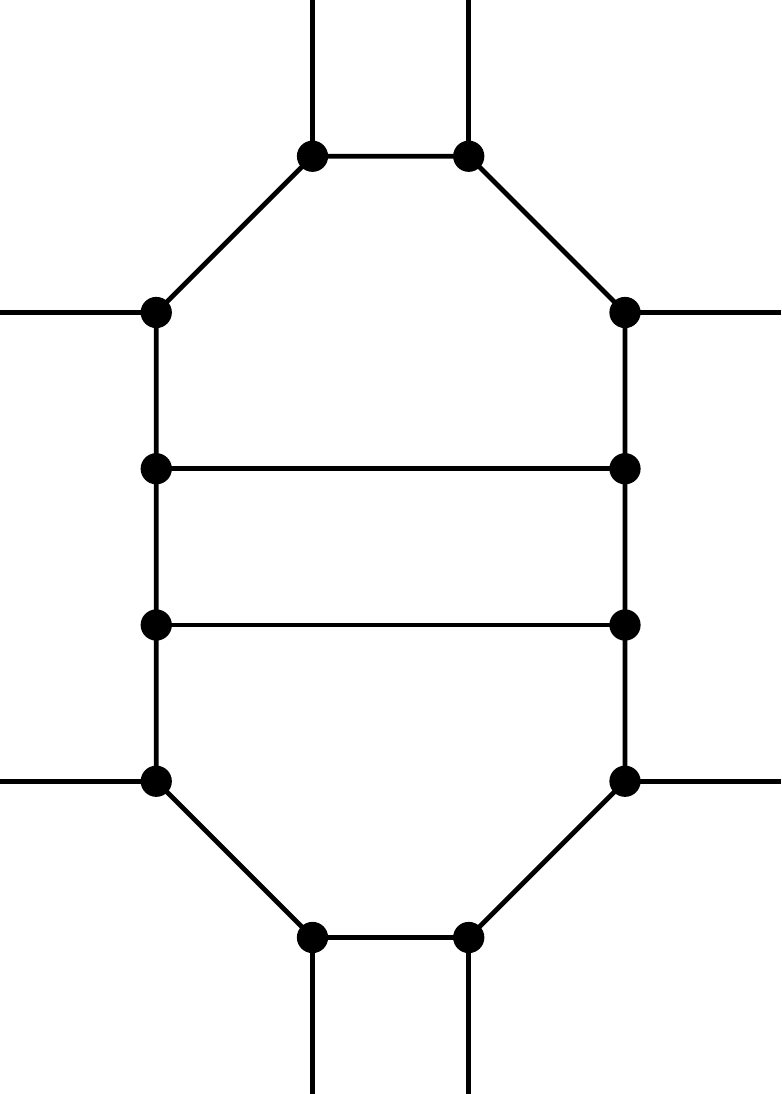}}
}
\captionbox{Graph substitution\label{f:14a}}[0.49\textwidth]{
\subcaptionbox{Neighbourhood\label{f:o14a}}[0.24\textwidth]{\includegraphics[width=0.22\textwidth]{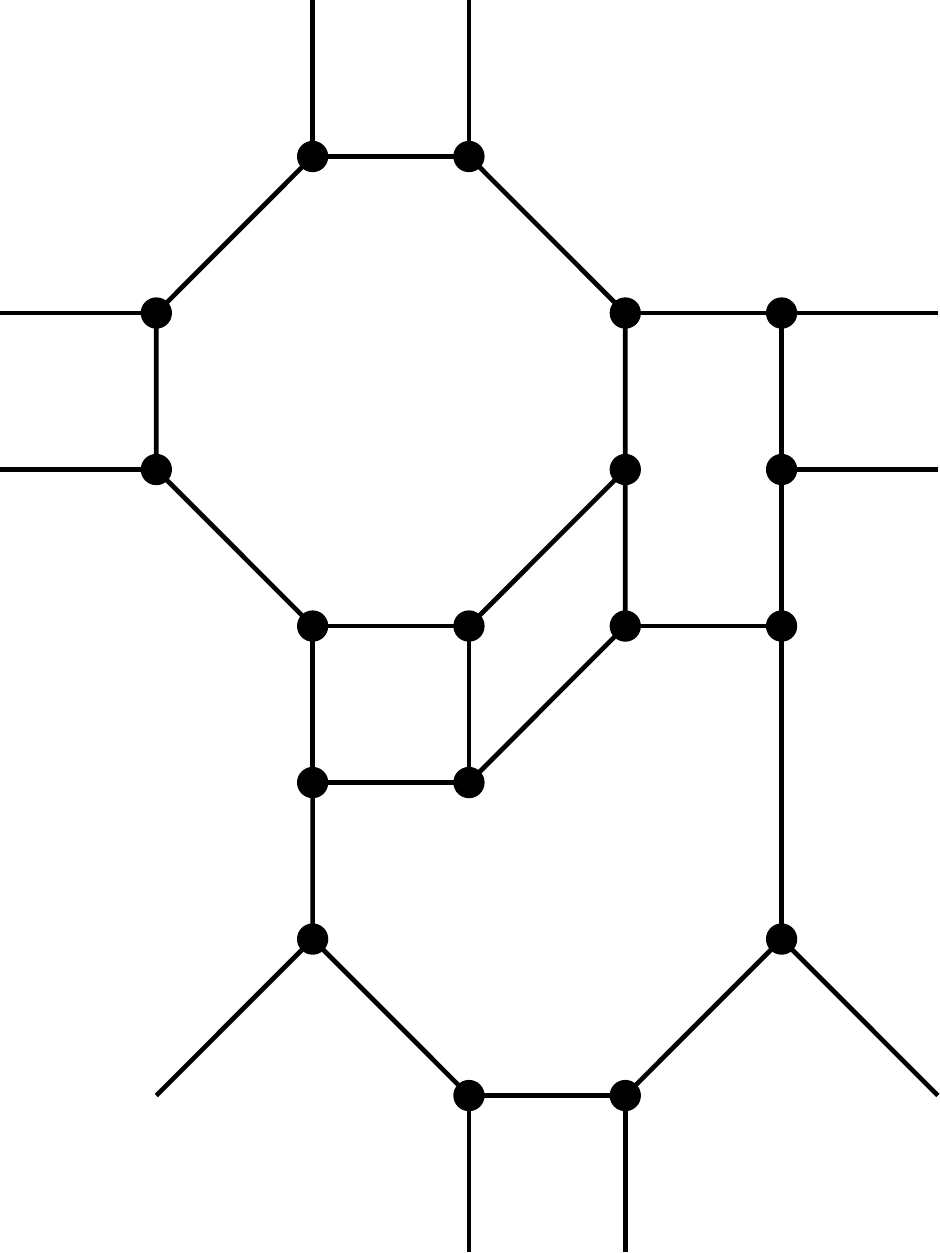}}
\subcaptionbox{Substitution\label{f:14asub}}[0.24\textwidth]{\includegraphics[width=0.22\textwidth]{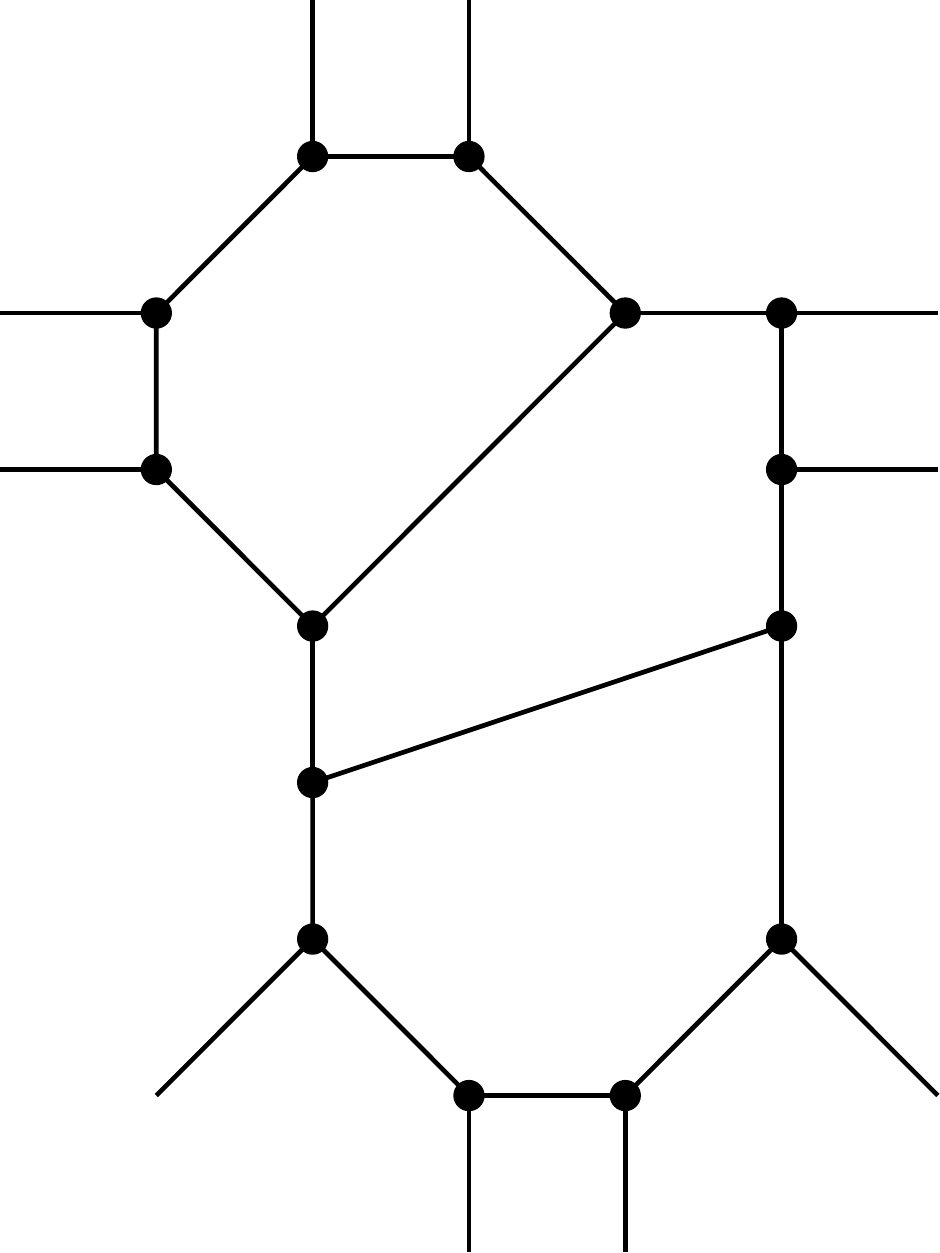}}
}
\end{figure}
\begin{figure}[hp]
\centering
\captionbox{Graph substitution\label{f:13a}}[0.49\textwidth]{
\subcaptionbox{Neighbourhood\label{f:o13a}}[0.24\textwidth]{\includegraphics[width=0.22\textwidth]{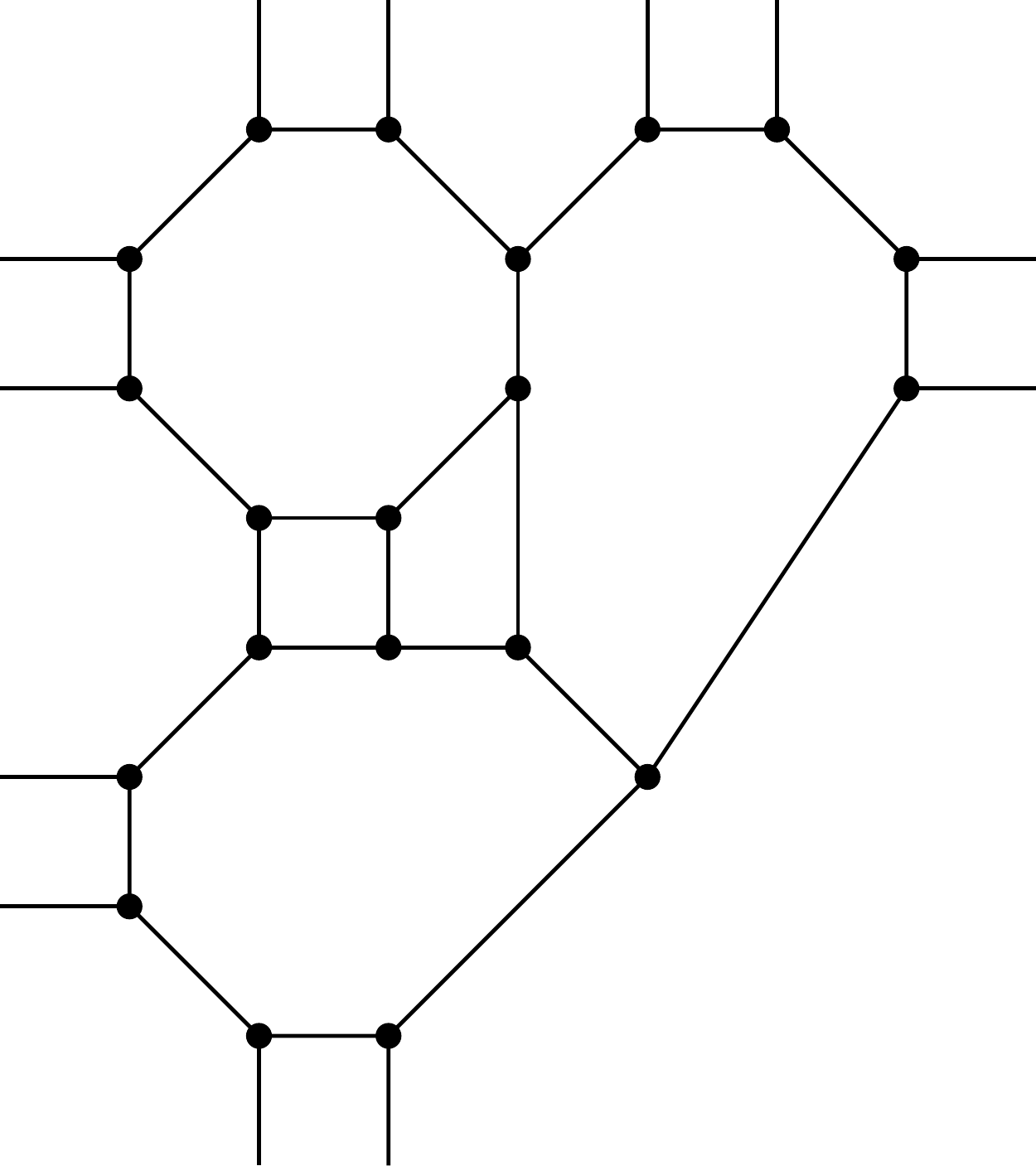}}
\subcaptionbox{Substitution\label{f:13asub}}[0.24\textwidth]{\includegraphics[width=0.22\textwidth]{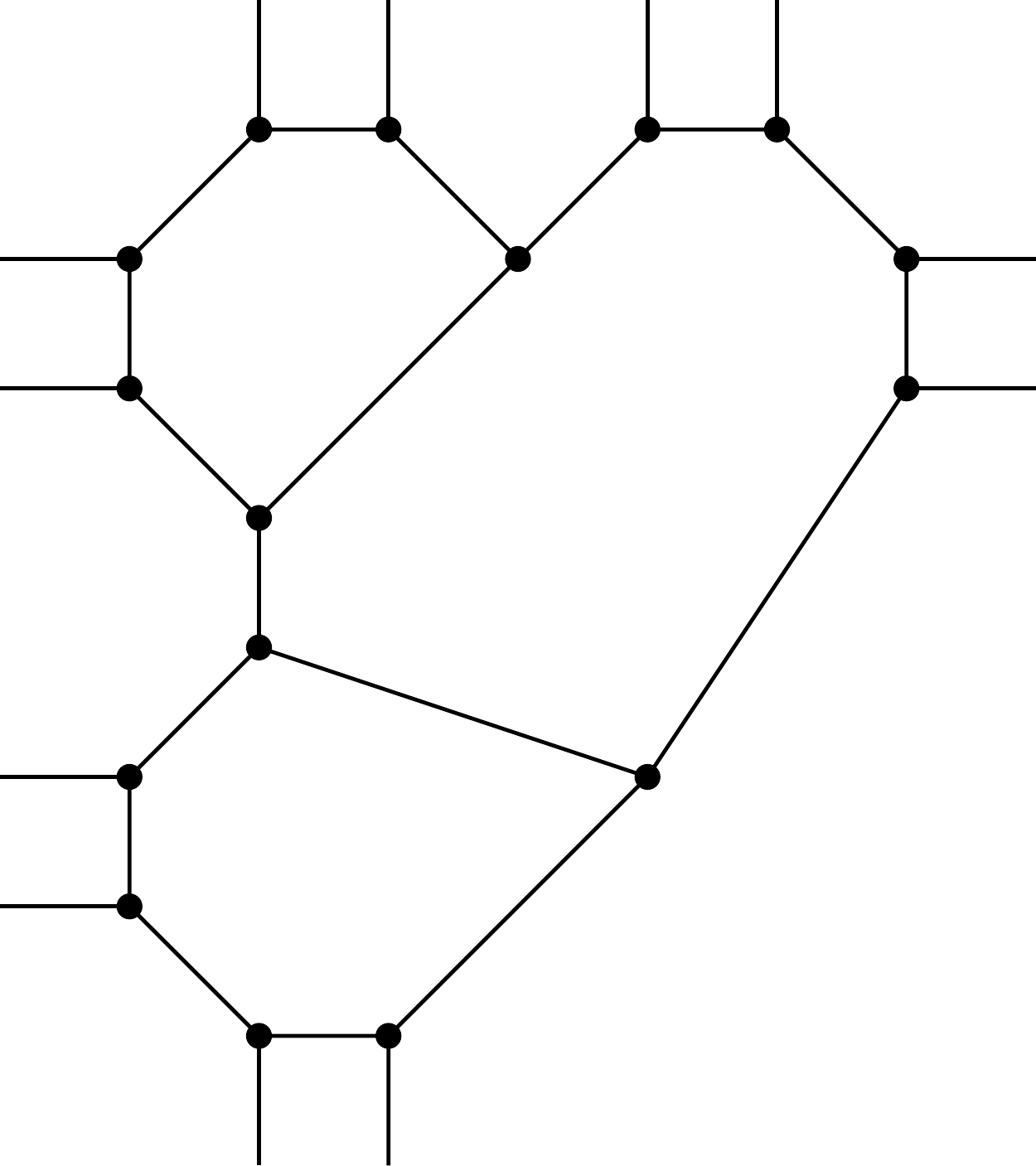}}
}
\captionbox{Graph substitution\label{f:13ca}}[0.49\textwidth]{
\subcaptionbox{Neighbourhood\label{f:o13ca}}[0.24\textwidth]{\includegraphics[width=0.22\textwidth]{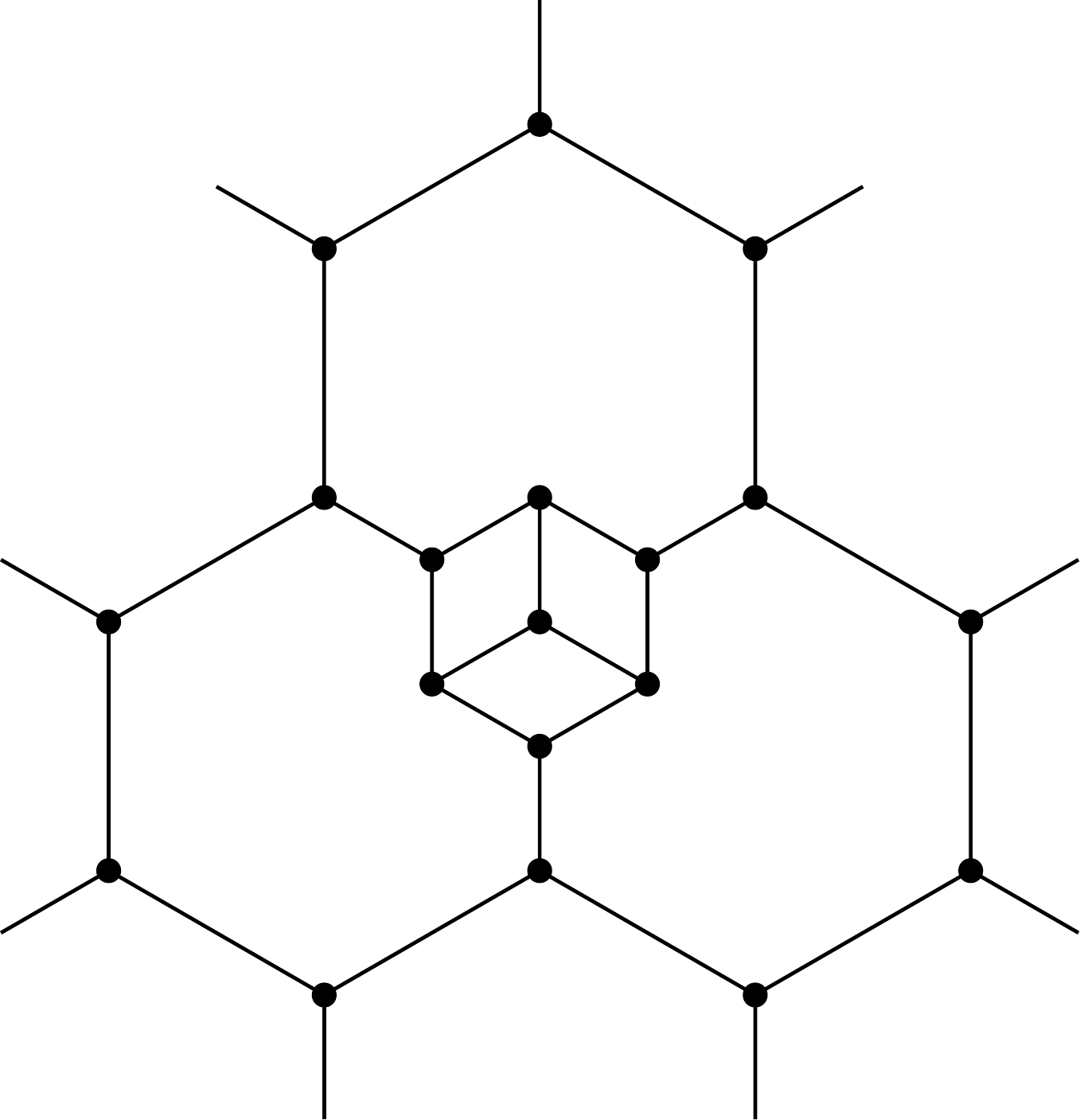}}
\subcaptionbox{Substitution\label{f:13casub}}[0.24\textwidth]{\includegraphics[width=0.22\textwidth]{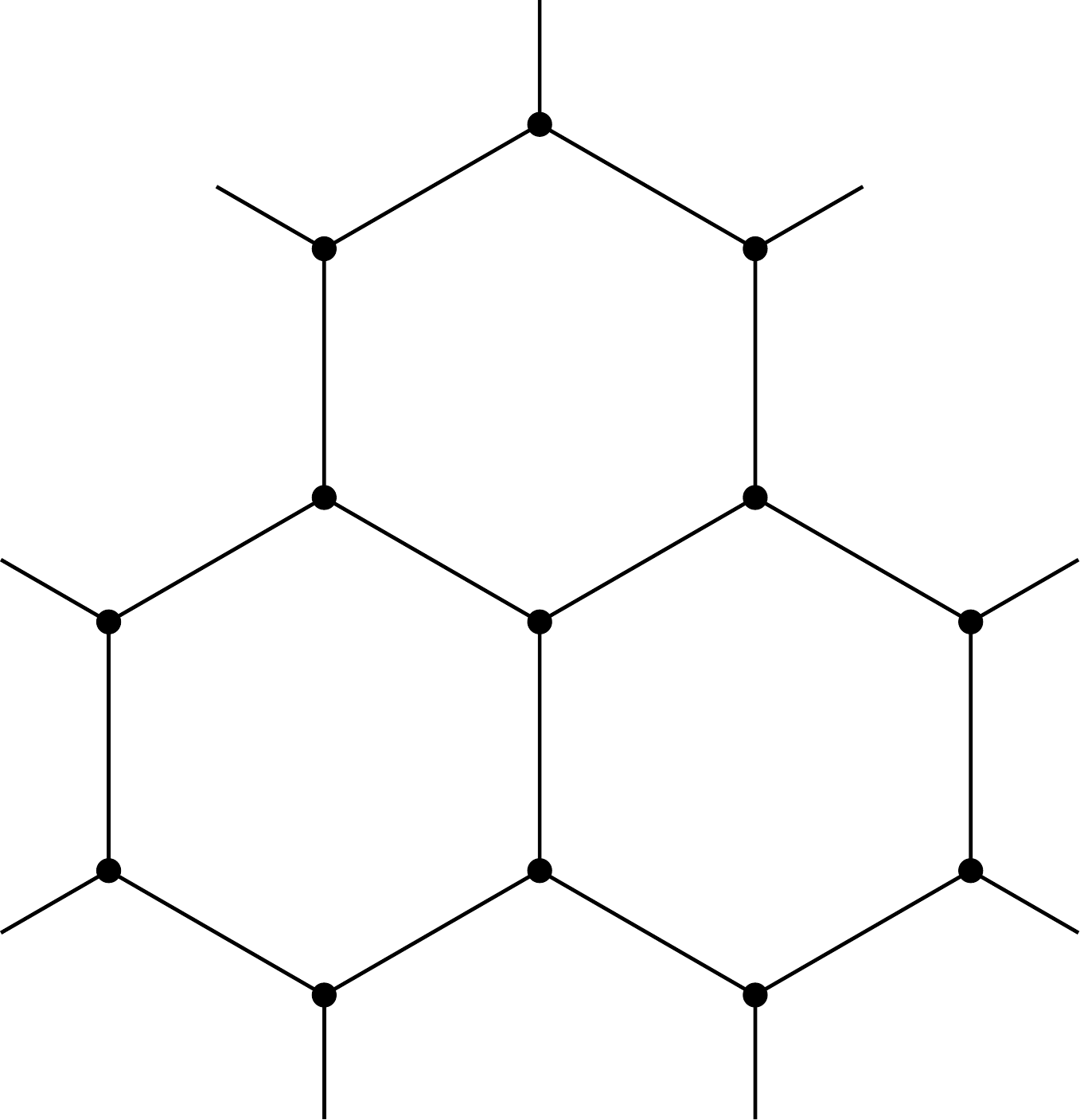}}
}
\end{figure}
\end{document}